\documentclass{amsart}

\usepackage{amsmath,amsthm,amssymb}
\usepackage{enumitem}
\usepackage{graphicx}
\usepackage{float}
\usepackage[margin=1.5in]{geometry}
\numberwithin{equation}{section}

\theoremstyle{plain}
\newtheorem{thm}{Theorem}[section]
\newtheorem{cor}[thm]{Corollary}
\newtheorem{lem}[thm]{Lemma}
\newtheorem{prop}[thm]{Proposition}

\theoremstyle{definition}
\newtheorem{defn}[thm]{Definition}

\theoremstyle{remark}
\newtheorem{rem}[thm]{Remark}

\usepackage{tikz}
\usetikzlibrary{arrows,shapes,trees}
\usepackage{graphicx}
\usepackage{hyperref}
\usepackage{amssymb}
\usepackage{amsfonts}
\usepackage{amsthm}
\usepackage{amsmath}
\usepackage{epstopdf}
\usepackage{mathrsfs}

\newcommand{\N}{\mathbb N}

\newcommand{\Z}{\mathbb Z}

\newcommand{\R}{\mathbb R}

\newcommand{\e}{\varepsilon}

\newcommand{\dist}{\mathrm{dist}}
\newcommand{\diam}{\mathrm{diam}}
\newcommand{\ls}{\lesssim}
\newcommand{\gs}{\gtrsim}
\newcommand{\Mod}{\mathrm{Mod}}
\newcommand{\loc}{\mathrm{loc}}
\newcommand{\supp}{\mathrm{supp}}
\newcommand{\ch}{\check}

\renewcommand{\t}{\tilde}
\newcommand{\la}{\lambda}

\newcommand{\p}{\partial}

\def\Xint#1{\mathchoice
{\XXint\displaystyle\textstyle{#1}}%
{\XXint\textstyle\scriptstyle{#1}}%
{\XXint\scriptstyle\scriptscriptstyle{#1}}%
{\XXint\scriptscriptstyle\scriptscriptstyle{#1}}%
\!\int}
\def\XXint#1#2#3{{\setbox0=\hbox{$#1{#2#3}{\int}$ }
\vcenter{\hbox{$#2#3$ }}\kern-.6\wd0}}

\def\dashint{\Xint-}

\title{Extension and trace theorems for noncompact doubling spaces}
\author{Clark Butler}

\begin{document}
\begin{abstract}
We generalize the extension and trace results of Bj\"orn-Bj\"orn-Shanmugalingam \cite{BBS21} to the setting of complete noncompact doubling metric measure spaces and their uniformized hyperbolic fillings. This is done through a uniformization procedure introduced by the author that uniformizes a Gromov hyperbolic space using a Busemann function instead of the distance functions considered in the work of Bonk-Heinonen-Koskela \cite{BHK}. We deduce several corollaries for the Besov spaces that arise as trace spaces in this fashion, including the existence of representatives that are quasicontinuous with respect to the Besov capacity, the existence of $L^p$-Lebesgue points quasieverywhere with respect to the Besov capacity, embeddings into H\"older spaces for appropriate exponents, and a stronger Lebesgue point result under an additional reverse doubling hypothesis on the measure. We also obtain several Poincar\'e-type inequalities relating integrals of Besov functions over balls to integrals of upper gradients of extension of these functions to a uniformized hyperbolic filling of the space. 
\end{abstract}

\maketitle

\section{Introduction}

In recent breakthrough work of Bj\"orn-Bj\"orn-Shanmugalingam \cite{BBS21} an extensive array of properties of Besov spaces on a compact doubling metric measure space $Z$ were established by exhibiting these spaces as the trace space of a Newton-Sobolev space on an associated incomplete metric graph $X_{\rho}$ having $Z$ as its boundary. These properties include density of Lipschitz functions, the existence of quasicontinuous representatives of Besov functions, embeddings into H\"older spaces under appropriate hypotheses on the regularity exponents, and improved integrability as well as the existence of Lebesgue points quasieverywhere under an additional reverse-doubling hypothesis on the measure \cite[Section 13]{BBS21}. In this work we extend these results to the setting of complete doubling metric measure spaces by removing the compactness hypothesis on $Z$. This is done through the use of a new uniformization construction for Gromov hyperbolic spaces introduced by the author in previous work \cite{Bu20} to similarly exhibit $Z$ as the boundary of an incomplete metric graph $X_{\rho}$, even in the case that $Z$ is not compact. For the compact case this uniformization is done in \cite{BBS21} using the procedure of Bonk-Heinonen-Koskela \cite{BHK}; our construction generalizes this procedure. 

To state our main theorem we need to introduce some important additional concepts, which are similar in nature to those considered in \cite{BBS21}. For precise definitions of the notions that follow we refer to Sections \ref{sec:capacities}, \ref{sec:fillings}, and \ref{sec:lift}. For the purposes of this paper a \emph{metric measure space} $(Z,d,\nu)$ is a metric space $(Z,d)$ equipped with a Borel regular measure $\nu$ such that for every ball $B \subset Z$ we have $0 < \nu(B) < \infty$. Given a complete doubling metric measure space $(Z,d,\nu)$, in Section \ref{sec:fillings} we construct a metric graph $X$ that is a proper geodesic Gromov hyperbolic space such that $Z$ can be canonically identified with the complement of a distinguished point $\omega$ in the Gromov boundary $\p X$ of $X$. Such a graph is known as a \emph{hyperbolic filling} of $Z$; these constructions were first considered by Bonk-Kleiner \cite{BK02} and Bourdon-Pajot \cite{BP03} for compact $Z$ (see also Bonk-Schramm \cite{BS00} for an alternative construction known as the \emph{hyperbolic cone}). Our construction for noncompact $Z$ is inspired by a construction due to Buyalo-Schroeder \cite[Chapter 6]{BS07}.

We then \emph{uniformize} $X$ to produce a uniform metric space $X_{\rho}$ such that the metric boundary $\p X_{\rho}$ of $X_{\rho}$ has a canonical biLipschitz identification with $Z$. After a biLipschitz change of coordinates on $Z$ we can then assume that $Z$ is isometrically identified with $\p X_{\rho}$. In Section \ref{sec:lift} we lift $\nu$ to a \emph{uniformly locally doubling} measure $\mu$ on the hyperbolic graph $X$. We uniformize this measure for each parameter $\beta > 0$ to obtain a new measure $\mu_{\beta}$ on $X_{\rho}$ such that the resulting metric measure spaces $(X_{\rho},d_{\rho},\mu_{\beta})$ and $(\bar{X}_{\rho},d_{\rho},\mu_{\beta})$ are each doubling and satisfy a $1$-Poincar\'e inequality, where $\mu_{\beta}$ is extended to $\bar{X}_{\rho}$ by setting $\mu_{\beta}(\p X_{\rho}) = 0$. Newtonian functions on these metric measure spaces will serve as counterparts to Besov functions on $(Z,d,\nu)$. This uniformization procedure for measures is based on a generalization by the author \cite{Bu22} of several results of Bj\"orn-Bj\"orn-Shanmugalingam \cite{BBS20} concerning the transformation of doubling measures and Poincar\'e inequalities under the Bonk-Heinonen-Koskela uniformization to the setting of our uniformization. 

For a given $p \geq 1$, $\theta > 0$, and $u \in L^{p}_{\loc}(Z)$, the \emph{Besov norm} on $Z$ is defined by 
\begin{equation}\label{Besov norm}
\|u\|_{B^{\theta}_{p}(Z)}^{p} = \int_{Z}\int_{Z}\frac{|u(x)-u(y)|^{p}}{d(x,y)^{p \theta}}\frac{d\nu(x)d\nu(y)}{\nu(B(x,d(x,y)))}.
\end{equation}
We write $\t{B}^{\theta}_{p}(Z) \subset L^{p}_{\loc}(Z)$ for the subspace of all functions $u$ such that $\|u\|_{B^{\theta}_{p}(Z)} < \infty$. We note that, properly speaking, \eqref{Besov norm} is actually a seminorm on $L^{p}_{\loc}(Z)$ since any constant function has Besov norm $0$. If one defines $B^{\theta}_{p}(Z) = \t{B}^{\theta}_{p}(Z)/\sim$ to be the quotient of $\t{B}^{\theta}_{p}(Z)$ by those measurable functions $u$ with $\|u\|_{B^{\theta}_{p}(Z)} = 0$ then the resulting space is a Banach space, see for instance \cite[Theorem 2.9]{S16}. We provide an alternative proof of this fact in Corollary \ref{Besov Banach}.

We define $\ch{B}^{\theta}_{p}(Z) = L^{p}(Z) \cap \t{B}^{\theta}_{p}(Z)$ to be the subspace of $\t{B}^{\theta}_{p}(Z)$ consisting of those functions that are $p$-integrable over $Z$. We equip this subspace with the norm 
\begin{equation}\label{check Besov norm}
\|u\|_{\ch{B}^{\theta}_{p}(Z)} = \|u\|_{L^{p}(Z)} + \|u\|_{B^{\theta}_{p}(Z)},
\end{equation}
With this norm $\ch{B}^{\theta}_{p}(Z)$ is a Banach space by the arguments in \cite[Remark 9.8]{BBS21}. We always have $\ch{B}^{\theta}_{p}(Z) = \t{B}^{\theta}_{p}(Z)$ whenever $Z$ is bounded. For bounded $Z$ these Besov spaces were introduced by Bourdon-Pajot in a more restrictive setting in \cite{BP03}. In \cite{GKS10} it was shown by Gogatishvili-Koskela-Shanmugalingam that this definition coincides with a more classical formulation for Besov spaces on metric spaces. We will use the notation of Bourdon-Pajot in this paper. 

On $X_{\rho}$ and its completion $\bar{X}_{\rho} = X_{\rho} \cup Z$, we will consider the \emph{Newtonian} spaces defined through upper gradients. These spaces are based on a notion of gradient in the metric space setting introduced by Heinonen-Koskela \cite{HK98}, and their formal study was initiated by Shanmugalingam \cite{S04}. We give an abbreviated description of these spaces here; a more detailed description can be found in Section \ref{sec:capacities}. Beginning with a metric measure space $(Y,d,\mu)$, a Borel function $g:Y \rightarrow [0,\infty]$ is an \emph{upper gradient} of a function $u: Y \rightarrow [-\infty,\infty]$ if for each nonconstant compact rectifiable curve $\gamma$ in $Y$ joining two points $x,y \in Y$ we have that
\begin{equation}\label{upper gradient inequality}
|u(x)-u(y)| \leq \int_{\gamma} g \, ds. 
\end{equation}
By convention we require that $\int_{\gamma} g\,ds = \infty$ if $u(x) = \pm \infty$ or $u(y)  = \pm \infty$. The \emph{Newtonian norm} of $u$ for a given $p \geq 1$ is then defined to be
\begin{equation}\label{first Newton norm}
\|u\|_{N^{1,p}(Y)} = \|u\|_{L^{p}(Y)} + \inf_{g}\|g\|_{L^{p}(Y)},
\end{equation}
where the infimum is taken over all upper gradients $g$ of $u$. We write $\t{N}^{1,p}(Y)$ for the collection of functions such that $\|u\|_{N^{1,p}(Y)} < \infty$. Similarly the \emph{Dirichlet norm} of $u$ for a given $p \geq 1$ is defined to be the Dirichlet $p$-energy, 
\begin{equation}\label{first Dirichlet norm}
\|u\|_{D^{1,p}(Y)} = \inf_{g}\|g\|_{L^{p}(Y)},
\end{equation}
with the infimum taken over all upper gradients $g$ of $u$. We then write $\t{D}^{1,p}(Y)$ for the collection of functions such that $\|u\|_{D^{1,p}(Y)} < \infty$, which can equivalently be thought of as the collection of functions $u$ with a $p$-integrable upper gradient. We again note that \eqref{first Newton norm} and \eqref{first Dirichlet norm} actually define seminorms on $\t{N}^{1,p}(Y)$ and $\t{D}^{1,p}(Y)$ and one must pass to a quotient to obtain a true norm. This issue will not be relevant to any of the statements we make in this first section. 

We can now state our main theorem. Given $p \geq 1$ and $0 < \theta < 1$, the uniformized hyperbolic filling $\bar{X}_{\rho}$ is considered to be equipped with the uniformized measure $\mu_{\beta}$ for the specific parameter value $\beta = p(1-\theta)$. We refer to Section \ref{sec:lift} for a precise description of the measure $\mu_{\beta}$ for each $\beta > 0$. We warn the reader that our definition of $\beta$ is slightly different than the one used in \cite{BBS21}: our $\beta$ corresponds to $\beta/\epsilon$ in their work. 

The linear operators below are maps between seminormed spaces, and should be understood as being bounded with respect to these seminorms; to simplify terminology we will generally refer to the seminorms \eqref{Besov norm}, \eqref{first Newton norm}, and \eqref{first Dirichlet norm} as norms on $\t{B}^{\theta}_{p}(Z)$, $\t{N}^{1,p}(Y)$, and $\t{D}^{1,p}(Y)$ respectively, even though one must pass to a quotient of these spaces to actually obtain a norm. The expression \eqref{check Besov norm} does define a genuine norm on $\ch{B}^{\theta}_{p}(Z)$, however. These matters are treated in greater detail in Section \ref{sec:capacities}. 

\begin{thm}\label{thm:extendtrace}
There are bounded linear operators $T: \t{D}^{1,p}(X_{\rho}) \rightarrow \t{B}^{\theta}_{p}(Z)$ and $P: \t{B}^{\theta}_{p}(Z) \rightarrow \t{D}^{1,p}(\bar{X}_{\rho})$ such that for each $f \in \t{B}^{\theta}_{p}(Z)$ we have $T(Pf) = f$ $\nu$-a.e., i.e., $T \circ P$ is the identity on $\t{B}^{\theta}_{p}(Z)$.

Furthermore $T$ restricts to a bounded linear operator $T: \t{N}^{1,p}(X_{\rho}) \rightarrow \ch{B}^{\theta}_{p}(Z)$ and there is a truncation $P_{0}$ of $P$ that defines a bounded linear operator $P_{0}: \ch{B}^{\theta}_{p}(Z) \rightarrow \t{N}^{1,p}(\bar{X}_{\rho})$ with $T \circ P_{0}$ being the identity on $\ch{B}^{\theta}_{p}(Z)$. 
\end{thm}

The operators $T$ and $P$ are known as \emph{trace} and \emph{extension} operators respectively. We refer to Sections \ref{sec:trace} and \ref{sec:extend} for a description of these operators, as well as a description of the parameters on which the norms of these linear operators depend. The truncation $P_{0}$ of $P$ is defined in Section \ref{sec:extend}; we remark that $P$ itself does not define a bounded linear operator from $\ch{B}^{\theta}_{p}(Z)$ to $\t{N}^{1,p}(\bar{X}_{\rho})$, as an adaptation of the arguments of Proposition \ref{infinite measure} shows.  We emphasize that the domain of $T$ consists of functions defined on $X_{\rho}$, not on $\bar{X}_{\rho}$; even though $\p X_{\rho}$ has measure zero with respect to $\mu_{\beta}$, it is not at all obvious that functions in $\t{D}^{1,p}(X_{\rho})$ and $\t{N}^{1,p}(X_{\rho})$ can naturally be extended to functions in $\t{D}^{1,p}(\bar{X}_{\rho})$ and $\t{N}^{1,p}(\bar{X}_{\rho})$. For this we rely on work of J. Bj\"orn and Shanmugalingam \cite{BJS07}. Theorem \ref{thm:extendtrace} can be interpreted as saying that the trace space of $\t{D}^{1,p}(X_{\rho})$ on $Z$ is $\t{B}^{\theta}_{p}(Z)$, and the trace space of $\t{N}^{1,p}(X_{\rho})$ on $Z$ is $\ch{B}^{\theta}_{p}(Z) = L^{p}(Z) \cap \t{B}^{\theta}_{p}(Z)$. Results of this type in a closely related setting were obtained previously for Triebel-Lizorkin spaces by Bonk-Saksman-Soto \cite{BSS18} (see \cite{S16} for an adaptation of these results to Besov spaces). 

The proof of Theorem \ref{thm:extendtrace} can be modified to derive yet another nontrivial characterization of Besov spaces, this time in terms of a concept that we will refer to as \emph{hyperbolic upper gradients} (see Definition \ref{define hyp upper gradient}). This characterization, given in Proposition \ref{characterize hyp upper}, passes through the pointwise characterization of Besov spaces using  fractional Haj\l asz gradients introduced by Koskela-Yang-Zhou \cite{KYZ11}. We also derive several Poincar\'e-type inequalities controlling integrals of functions over balls in $Z$ using integrals of upper gradients of extensions of these functions over corresponding balls in $X_{\rho}$, see Propositions \ref{Poincare type inequality}, \ref{extension hyperbolic Poincare}, and \ref{hyperbolic Poincare}.

Theorem \ref{thm:extendtrace} yields numerous corollaries for the Besov space $\t{B}^{\theta}_{p}(Z)$, most of which follow from this theorem in essentially the same way that the corollaries of \cite[Corollary 1.2]{BBS21} follow from \cite[Theorem 1.1]{BBS21}. We split these consequences into three corollaries. We refer to Section \ref{subsec:Besov} for a precise definition of the Besov capacity referenced in Corollaries \ref{cor:quasicontinuous} and \ref{cor:Lebesgue points}, as well as a precise definition of quasicontinuity with respect to the Besov capacity, which can be considered to be a strong refinement of Lusin's theorem since the Besov capacity is always absolutely continuous with respect to the measure $\nu$ on $Z$ (see Propositions \ref{absolutely continuous capacity} and \ref{compute single capacity}). The notion of $L^{p}(Z)$-Lebesgue points for a given $p \geq 1$ is standard and is defined in \eqref{Lebesgue point}. For an extensive discussion of predecessors to these corollaries we refer to the discussion after \cite[Corollary 1.2]{BBS21}. With the exception of the Lebesgue point assertion in Corollary \ref{cor:quasicontinuous}, all of these corollaries were established in the case of compact $Z$ in \cite{BBS21}; however the generalization from compact $Z$ to noncompact $Z$ is highly nontrivial due to the nonlocal nature of the Besov norm. For each of the corollaries below we always assume that $p \geq 1$ and $0 < \theta < 1$. 

%Our first corollary shows that every function in  $\t{B}^{\theta}_{p}(Z)$ has a representative that is quasicontinuous and has $L^{p}(Z)$-Lebesgue points quasieverywhere with respect to the Besov capacity.

\begin{cor}\label{cor:quasicontinuous}
Each function in $\t{B}^{\theta}_{p}(Z)$ has a representative which is quasicontinuous with respect to the Besov capacity. This representative has $L^{p}(Z)$-Lebesgue points quasieverywhere with respect to the Besov capacity. 
\end{cor}

The existence of Lebesgue points quasieverywhere and the quasicontinuity assertion of Corollary \ref{cor:quasicontinuous} were previously obtained by Heikkinen-Koskela-Tuominen \cite[Theorem 8.1]{HKT17}. Our result shows for the Besov spaces $\t{B}^{\theta}_{p}(Z)$ that we consider that one can obtain a stronger $L^p$-Lebesgue point result compared to the $L^1$-Lebesgue point result obtained in \cite{HKT17}. The Lebesgue point result we obtain here is new even in the setting of compact $Z$ considered in \cite{BBS21}, since their Lebesgue point result  in \cite[Corollary 1.2]{BBS21} is only obtained under an additional reverse-doubling hypothesis on the measure $\nu$. 

Our second corollary concerns embeddings of Besov spaces into H\"older spaces. For this we will use the notion of \emph{relative lower volume decay of order $Q$} for a doubling measure $\nu$, defined by the inequality \eqref{lower volume} for a given exponent $Q > 0$. The exponent $Q$ is also known as a \emph{doubling dimension} for $\nu$; we remark that the relative volume decay estimate always holds with $Q = \log_{2} C_{\nu}$, where $C_{\nu}$ is the constant in the doubling inequality for $\nu$, but it can also hold for smaller $Q$. The notation $Q_{\beta}$ below is motivated by the equality $\beta = p(1-\theta)$ that we assumed in Theorem \ref{thm:extendtrace}. See also Lemma \ref{doubling dimension}.

\begin{cor}\label{cor:holder embed}
Suppose that $\nu$ has relative lower volume decay of order $Q > 0$. Set $Q_{\beta} = \max\{1,Q+p(1-\theta)\}$ and assume that $p > Q_{\beta}$. Then every function in $\t{B}^{\theta}_{p}(Z)$ has a representative that is locally $(1-Q_{\beta}/p)$-H\"older continuous. 
\end{cor}

Lastly we show that Corollary \ref{cor:quasicontinuous} can be improved when $\nu$ additionally satisfies a \emph{reverse-doubling} property \eqref{reverse doubling}, which holds in particular when $Z$ is uniformly perfect \cite[Lemma 4.1]{MO19}.

\begin{cor}\label{cor:Lebesgue points}
Suppose that $\nu$ has relative lower volume decay of order $Q > 0$ and that $\nu$ satisfies the reverse-doubling property \eqref{reverse doubling}. We assume further that $p \theta < Q$ and set $Q_{*} = Qp/(Q-p\theta)$. Then functions in $\t{B}^{\theta}_{p}(Z)$ belong to $L^{Q_{*}}_{\loc}(Z)$ and have representatives which have $L^{Q_{*}}(Z)$-Lebesgue points quasieverywhere with respect to the Besov capacity. 
\end{cor}

The H\"older embedding result in Corollary \ref{cor:holder embed} and the higher (local) integrability result in Corollary \ref{cor:Lebesgue points} have previously been obtained in the work of Mal\'y using the method of Haj\l asz gradients \cite[Corollary 3.18]{M17}. However the embedding results we obtain are actually sharper, as they control the local H\"older norm and local $L^{Q_{*}}$-norm by local integrals of upper gradients of the extension of the function to the uniformized hyperbolic filling $\bar{X}_{\rho}$ (which can be regarded as a localized version of the Besov energy) as opposed to the total Besov energy considered in \cite[Corollary 3.18]{M17}. See in particular Propositions \ref{prop:holder embed} and  \ref{hyperbolic Poincare}. When $Z$ is unbounded Mal\'y also obtains a global embedding result of $\t{B}^{\theta}_{p}(Z)$ into $L^{Q_{*}}(Z)$ under an additional lower mass bound on $\nu$, see \cite[(3.2), Corollary 3.18(i)]{M17}. The work of Mart\'in-Ortiz \cite{MO19} expands on these Sobolev-type embedding theorems of $\ch{B}^{\theta}_{p}(Z)$ into $L^{q}(Z)$ for appropriate $q \geq 1$ when $Z$ is unbounded, and shows that the reverse doubling condition is not sufficient on its own for such an embedding theorem to hold. 

Lastly we remark that \cite[Corollary 1.2]{BBS21} also shows that Lipschitz functions are dense in $\t{B}^{\theta}_{p}(Z)$ when $Z$ is compact. This is derived as an immediate consequence of \cite[Theorem 1.1]{BBS21} and the fact that Lipschitz functions are dense in $N^{1,p}(X_{\rho})$. Adapting their argument to our setting using Theorem \ref{thm:extendtrace} shows that Lipschitz functions with compact support are dense in $\ch{B}^{\theta}_{p}(Z) = L^{p}(Z) \cap \t{B}^{\theta}_{p}(Z)$. We avoid treating this topic here as the density of Lipschitz functions with compact support in both $\t{B}^{\theta}_{p}(Z)$ and $\ch{B}^{\theta}_{p}(Z)$ has already been obtained by Soto \cite[Corollary 3.3]{S16} even in the case of noncompact $Z$, using an adaptation of the proof of the same result for Triebel-Lizorkin spaces by Bonk-Saksman-Soto \cite[Corollary 3.4]{BSS18}. We note that Proposition \ref{converge Besov} gives a direct proof that locally Lipschitz functions are dense in both $\t{B}^{\theta}_{p}(Z)$ and $\ch{B}^{\theta}_{p}(Z)$. 

As we mentioned previously, this work is a direct descendant of a number of recent works connecting function spaces on doubling metric spaces with function spaces on hyperbolic fillings associated to the space. In addition to all of the works we mentioned previously, other notable works include Bonk-Saksman \cite{BS18} (from whom we've taken a significant amount of direct inspiration), Bonk-Saksman-Soto \cite{BSS18} (who consider a wider class of function spaces on $Z$ including the Triebel-Lizorkin spaces), and Saksman-Soto \cite{SS17} (who consider traces onto Ahlfors-regular subsets of $Z$). We also highlight the work of Mal\'y \cite{M17}, who obtains many similar results via different methods in the setting of John domains with compact closure (as well as several more general settings) in metric measure spaces. Lastly we note that the idea of considering Besov spaces as trace spaces of Newtonian spaces on a uniformized Gromov hyperbolic space previously appeared in work of Bj\"orn-Bj\"orn-Gill-Shanmugalingam \cite{BBGS17} on geometric analysis on trees, their uniformizations, and their Cantor set boundaries. 

The reader will find that the structure of our paper is broadly similar to that of \cite{BBS21}. This is because the bulk of the theory in \cite{BBS21} relies only on the fact that the uniformized hyperbolic filling $X_{\rho}$ equipped with the lifted measure $\mu_{\beta}$ for a given $\beta > 0$ is a doubling metric measure space satisfying a $1$-Poincar\'e inequality. We show that these claims carry over to the setting of noncompact $Z$. Nevertheless the boundedness of $Z$ (and consequently of $X_{\rho}$ in their setting) enters into the proofs at several critical junctures, requiring us to modify their arguments at many steps. In particular there is no distinction between the spaces $\t{D}^{1,p}(X_{\rho})$ and $\t{N}^{1,p}(X_{\rho})$ when $X_{\rho}$ is bounded, as is noted in Section \ref{sec:capacities}. We are also able to give independent (and in some cases, simpler) proofs of several of their claims due to the more restrictive nature of our hyperbolic filling construction; our hyperbolic fillings in general contain more edges than the hyperbolic fillings considered in \cite{BBS21}. A drawback of our approach is that while we are able to recover their results on Besov spaces in \cite[Section 13]{BBS21} via specialization of our corollaries to the case of compact $Z$, we are not able to recover their results on trace and extension theorems for trees (see \cite[Section 7]{BBS21}) due to the fact that our hyperbolic filling construction never yields a tree unless $Z$ is a single point (see \cite[Remark 7.2]{BBS21}; the same reasoning applies in our setting due to our parameter restriction \eqref{tau requirement}). 

We now give an overview of the structure of the paper. Section \ref{sec:capacities} describes much of the notation we will use throughout the paper and gives a brief overview of the aspects of the theory of Dirichlet and Newtonian spaces for doubling metric spaces satisfying a Poincar\'e inequality that we will need. We then define the Besov capacity and note some useful lemmas regarding Besov spaces. In Section \ref{sec:fillings} we construct a hyperbolic filling $X$ of a complete doubling metric space $Z$ and then construct a uniformization $X_{\rho}$ of $X$ whose boundary can be biLipschitz identified with $Z$. This primarily relies on our previous results in \cite{Bu20}. In Section \ref{sec:lift} we construct the lifts $\mu_{\beta}$ of the measure $\nu$ on $Z$ to $X_{\rho}$ for each $\beta > 0$ and establish important properties of the resulting metric measure spaces $(X_{\rho},d_{\rho},\mu_{\beta})$ and $(\bar{X}_{\rho},d_{\rho},\mu_{\beta})$. Section \ref{sec:trace} is devoted to trace results for Newtonian and Dirichlet functions defined on the uniformized filling $X_{\rho}$, while Section \ref{sec:extend} concerns results on extending Besov functions from $Z$ to $\bar{X}_{\rho}$.  Theorem \ref{thm:extendtrace} is then proved in Proposition \ref{extend to trace} by combining the results of Sections \ref{sec:trace} and \ref{sec:extend}. Finally in Section \ref{sec:properties} we prove Corollaries \ref{cor:quasicontinuous}, \ref{cor:holder embed}, and \ref{cor:Lebesgue points}.

We offer extensive thanks to Nageswari Shanmugalingam, who provided us with several early drafts of the work \cite{BBS21} that inspired both our work here and our previous papers \cite{Bu20}, \cite{Bu21}.

\section{Function spaces and capacities}\label{sec:capacities}

In this section we review some material from \cite[Section 9]{BBS21} and make some adjustments to deal with the fact that the spaces we will be considering are unbounded. The claims of this section are not new and are for the most part well known to experts; we have provided our own proofs wherever we were not able to find an adequate reference.

\subsection{Notation}\label{subsec:defn} We begin by fixing some notation for the rest of the paper. Let $X$ be a set and let $f$, $g$ be real-valued functions defined on $X$. For $c \geq 0$ we will write $f \doteq_{c} g$ if
\[
|f(x)-g(x)| \leq c,
\] 
for all $x \in X$. If the exact value of the constant $c$ is not important or implied by context we will often just write $f \doteq g$. The relation $f \doteq g$ will sometimes be referred to as a \emph{rough equality} between $f$ and $g$. Similarly for $C \geq 1$ and functions $f,g:X \rightarrow (0,\infty)$, we will write $f \asymp_{C} g$ if for all $x \in X$, 
\[
C^{-1}g(x) \leq f(x) \leq C g(x).
\]
We will write $f \asymp g$ if the value of $C$ is implied by context. We will write $f \ls_{C} g$ if $f(x) \leq Cg(x)$ for all $x \in X$ and $f \gs_{C} g$ if $f(x) \geq C^{-1}g(x)$ for $x \in X$. Thus $f \asymp_{C} g$ if and only if $f \ls_{C} g$ and $f \gs_{C} g$. As with the other notation, we will drop the constant $C$ and just write $f \ls g$ or $f \gs g$ if the value of $C$ is implied by context. We will generally stick to the convention of using $c \geq 0$ for additive constants and $C \geq 1$ for multiplicative constants. To indicate on what parameters -- such as $\delta$ -- the constants depend on we will write $c = c(\delta)$, etc. At the beginning of each section we will indicate on what parameters the implied constants of the inequalities $\ls$ and $\gs$, the comparisons $\asymp$, and the rough equalities $\doteq$ are allowed to depend. We will often reiterate these conditions for emphasis. 

For a metric space $(X,d)$ we will write $B_{X}(x,r) = \{y \in X:d(x,y) < r\}$ for the open ball of radius $r > 0$ centered at a point $x \in X$. We write $\bar{B}_{X}(x,r) = \{y \in X:d(x,y) \leq r\}$ for the closed ball of radius $r > 0$ centered at $x$. We note that the inclusion $\overline{B_{X}(x,r)} \subset \bar{B}_{X}(x,r)$ of the closure of the open ball into the closed ball can be strict in general. By convention all balls $B \subset X$ are considered to have a fixed center and radius, even though it may be the case that we have $B_{X}(x,r) = B_{X}(x',r')$ as sets for some $x \neq x'$, $r \neq r'$. All balls $B \subset X$ are also considered to be open balls unless otherwise specified.  We will write $r(B)$ for the radius of a ball $B$. For a ball $B = B_{X}(x,r)$ in $X$ and a constant $c > 0$ we write $cB = B_{X}(x,cr)$ for the corresponding ball with radius scaled by $c$. A metric space $(X,d)$ is \emph{proper} if its closed balls are compact.

For a subset $Y \subset  X$ we write $\diam(Y) = \sup\{d(x,y):x,y \in Y\}$ for the diameter of $Y$. For a point $x \in X$ we write $\dist(x,Y) = \inf\{d(x,y):y \in Y\}$ for the minimum distance of $x$ to $Y$. We write $\chi_{Y}: X \rightarrow \{0,1\}$ for the characteristic function of $Y$, which is defined by $\chi_{Y}(x) = 1$ if $x \in Y$ and $\chi_{Y}(x) = 0$ if $x \notin Y$. 

Let $f:(X,d) \rightarrow (X',d')$ be a map between metric spaces. The map $f$ is \emph{isometric} if $d'(f(x),f(y)) = d(x,y)$ for $x$, $y\in X$. We recall that a curve $\gamma: I \rightarrow X$ is a \emph{geodesic} if it is an isometric mapping of the interval $I \subset \R$ into $X$.  The space $X$ is \emph{geodesic} if any two points in $X$ can be joined by a geodesic.

\subsection{Newtonian spaces} For this next part we will be expanding on the discussion from the introduction; we refer to \cite{HKST} for more details. We start with a metric measure space $(Y,d,\mu)$. For $p \geq 1$ and a measurable function $u:Y \rightarrow [-\infty,\infty]$ we will use the notation $\|u\|_{L^{p}(Y)} = \left(\int_{Y}|u|^{p}\,d\mu\right)^{1/p}$ for the $L^p$ norm of $u$, and write $L^{p}(Y)$ for the associated $L^p$ space on $Y$. We will also say that functions $u \in L^{p}(Y)$ are \emph{$p$-integrable}. We write $L^{p}_{\loc}(Y)$ for the local $L^{p}$ space on $Y$, defined to be the set of all measurable functions $u:Y \rightarrow [-\infty,\infty]$ such that $u|_{B} \in L^{p}(B)$ for all balls $B \subset Y$,  where we consider $B$ as being equipped with the restricted measure $\mu|_{B}$. For a subset $G \subset Y$ satisfying $0 < \mu(G) < \infty$ and a function $u$ that is $\mu$-integrable over $G$ we write 
\begin{equation}\label{average notation}
u_{G} = \dashint_{G} u \, d\mu = \frac{1}{\mu(G)}\int_{G} u \, d\mu
\end{equation}
for the mean value of $u$ over $G$.

\begin{rem}\label{condense notation}
In order to condense notation throughout the paper we will be omitting the measure from the notation for the function spaces that we consider. Thus we write $L^{p}(Y) = L^{p}(Y,\mu)$, etc. We will always assume that $p \geq 1$ wherever it appears as an exponent, and we will always exclude the case $p = \infty$. We will always be using a fixed choice of measure on each metric space that we consider. Balls $B \subset Y$ in a metric measure space $(Y,d,\mu)$ will always be considered as metric measure spaces $(B,d,\mu|_{B})$ with the restriction of the distance and measure on $Y$ to $B$. Similarly we will sometimes write ``a.e."~ for ``$\mu$-a.e."~ when the measure $\mu$ is understood.  
\end{rem}

We next define the $p$-modulus for $p \geq 1$. Let $\Gamma$ be a family of curves in $Y$. A Borel function $\rho: Y \rightarrow [0,\infty]$ is \emph{admissible} for $\Gamma$ if for each curve $\gamma \in \Gamma$ we have $\int_{\gamma} \rho \, ds \geq 1$. The \emph{$p$-modulus} of $\Gamma$ is then defined as 
\[
\mathrm{Mod}_{p}(\Gamma) = \inf_{\rho} \int_{Y} \rho^{p} \, d\mu,
\]
with the infimum taken over all admissible Borel functions $\rho$ for $\Gamma$. We say that $\Gamma$ is \emph{$p$-exceptional} if $\Mod_{p}(\Gamma) = 0$.  A property $P$ is said to hold for \emph{$p$-a.e.~ curve} if the collection of curves for which $P$ fails is $p$-exceptional. A subset $G \subset Y$ is \emph{$p$-exceptional} if the family of all nonconstant curves meeting $G$ is $p$-exceptional.

We recall the definition \eqref{upper gradient inequality} of upper gradients in the introduction. It is natural to relax this definition by allowing an exceptional set of curves on which the inequality \eqref{upper gradient inequality} can potentially fail. For a function $u:Y \rightarrow [-\infty,\infty]$ and an exponent $p \geq 1$ we say that a Borel function $g:Y \rightarrow [0,\infty]$ is a \emph{$p$-weak upper gradient} for $u$ if the upper gradient inequality \eqref{upper gradient inequality} holds for $p$-a.e.~ curve in $Y$. The following standard lemma shows that $p$-weak upper gradients of $u$ are not far from being true upper gradients of $u$ from the perspective of the $L^p$ norm.

\begin{lem}\label{weak upper approximate} \cite[Lemma 6.2.2]{HKST} Let $u: Y \rightarrow [-\infty,\infty]$ be a function and suppose that $g$ is a $p$-weak upper gradient of $u$. Then there is a monotone decreasing sequence of upper gradients $\{g_{n}\}$ of $u$, with $g_{n} \geq g$ for each $n$, such that $\|g_{n}-g\|_{L^{p}(Y)} \rightarrow 0$. 
\end{lem}

In particular if $u$ has a $p$-integrable $p$-weak upper gradient then it has a $p$-integrable upper gradient. We will be using Lemma \ref{weak upper approximate} implicitly in many of the statements that follow. We will also use the product rule for upper gradients frequently in the rest of the paper without further citation. We refer the reader to \cite[Proposition 6.3.28]{HKST} for the precise hypotheses under which the product rule holds; for our purposes it suffices to note that the product rule holds for products of functions $u$ and $v$ that each have locally $p$-integrable $p$-weak upper gradients $f$ and $g$ respectively, in which case any Borel representative of $g|u| + f|v|$ defines a $p$-weak upper gradient for the product $uv$ (note that the function $g|u| + f|v|$ need not itself be $p$-integrable). The reference requires that $u$ and $v$ are absolutely continuous along $p$-a.e.~ compact curve in $Y$, which follows immediately from the local $p$-integrability of $f$ and $g$ and \cite[Proposition 6.3.2]{HKST} since each compact curve in $Y$ is contained in some ball $B \subset Y$ on which $f$ and $g$ are $p$-integrable.

If $u$ has a $p$-integrable upper gradient $g$ then it has a \emph{minimal $p$-weak upper gradient} $g_{u}$ that satisfies $g_{u} \leq g$ a.e.~ for all $p$-integrable $p$-weak upper gradients $g$ of $u$. This minimal $p$-weak upper gradient is unique up to a set of measure zero. The \emph{Newtonian space} $\t{N}^{1,p}(Y)$ is the space of all measurable functions $u: Y \rightarrow [-\infty,\infty]$ such that $\|u\|_{L^{p}(Y)} < \infty$ and such that $u$ has a $p$-weak upper gradient $g \in L^{p}(Y)$. This space comes equipped with the seminorm
\begin{equation}\label{Newton norm}
\|u\|_{N^{1,p}(Y)} = \|u\|_{L^{p}(Y)} + \|g_{u}\|_{L^{p}(Y)},
\end{equation}
where $g_{u}$ is a minimal $p$-weak upper gradient for $u$. Lemma \ref{weak upper approximate} shows that this definition is equivalent to the previous definition \eqref{first Newton norm} that we gave for the norm $\|u\|_{N^{1,p}(Y)}$. We emphasize that functions in $\t{N}^{1,p}(Y)$ are required to be defined pointwise everywhere, in constrast to what is typically required in the standard Sobolev space theory. We write $N^{1,p}(Y) = \t{N}^{1,p}(Y)/\sim$ be the quotient by the equivalence relation $u \sim v$ if $\|u-v\|_{N^{1,p}} = 0$. The space $N^{1,p}(Y)$ will also be referred to as the Newtonian space, and we will engage in the standard practice from the theory of $L^p$ spaces of not distinguishing the notation between a function $u \in \t{N}^{1,p}(Y)$ and its corresponding equivalence class $[u] \in N^{1,p}(Y)$. Equipped with the norm \eqref{Newton norm} the space $N^{1,p}(Y)$ is a Banach space \cite{S04}. As with $L^p$ and Sobolev spaces, we also define a local version $\t{N}^{1,p}_{\loc}(Y)$ consisting of those functions $u: Y \rightarrow [-\infty,\infty]$ such that for each ball $B \subset Y$ we have $u|_{B} \in \t{N}^{1,p}(B)$.

The \emph{$C^{Y}_{p}$-capacity} of a set $G \subset Y$ is defined as 
\begin{equation}\label{define capacity}
C_{p}^{Y}(G) = \inf_{u} \|u\|^{p}_{N^{1,p}(Y)},
\end{equation}
with the infimum taken over all functions $u \in \t{N}^{1,p}(Y)$ satisfying $u \geq 1$ on $G$.  A property $P$ is said to hold \emph{quasieverywhere} (q.e.) if the set $G$ of points at which it fails has zero $p$-capacity, i.e., it satisfies $C_{p}^{Y}(G) = 0$. By \cite[Proposition 7.2.8]{HKST} a set $G$ has zero $p$-capacity if and only if $\mu(G) = 0$ and $G$ is $p$-exceptional. Two functions $u,v \in \t{N}^{1,p}(Y)$ satisfy $u \sim v$ if and only if $u = v$ quasieverywhere \cite[Proposition 7.1.31]{HKST}. Furthermore if $u = v$ a.e.~ then $u = v$ q.e. Thus the $C_{p}^{Y}$-capacity captures the degree of ambiguity one is allowed in the definition of Newtonian functions in $N^{1,p}(Y)$. The $C_{p}^{Y}$-capacity is countably subadditive \cite[Lemma 7.2.4]{HKST} and therefore defines an outer measure on $Y$. 

%By \cite[Lemma 7.2.6]{HKST} one can equivalently restrict to functions $u \in \t{N}^{1,p}(Y)$ satisfying $0 \leq u \leq 1$ on $Y$ and $u = 1$ on $G$.

The \emph{Dirichlet space} $\t{D}^{1,p}(Y)$ consists of all measurable functions $u:Y \rightarrow [-\infty,\infty]$ such that $u$ has an upper gradient $g \in L^{p}(Y)$, or equivalently, such that $u$ has a $p$-weak upper gradient $g \in L^{p}(Y)$. We equip $\t{D}^{1,p}(Y)$ with the seminorm
\[
\|u\|_{D^{1,p}(Y)} = \|g_{u}\|_{p} = \inf_{g} \|g\|_{p},
\]
with $g_{u}$ denoting a minimal $p$-weak upper gradient for $u$ and the infimum being taken over all $p$-integrable upper gradients of $u$, or equivalently, over all $p$-integrable $p$-weak upper gradients of $u$. This matches our first definition \eqref{first Dirichlet norm} by Lemma \ref{weak upper approximate}. Similarly to $N^{1,p}(Y)$, we have for $u,v \in \t{D}^{1,p}(Y)$ that $u = v$ a.e.~ if and only if $u = v$ q.e.~\cite[Lemma 7.1.6]{HKST}. As with $N^{1,p}(Y)$, in order to obtain a norm we set $D^{1,p}(Y) = \t{D}^{1,p}(Y)/\sim$, with $u \sim v$ if $\|u-v\|_{D^{1,p}(Y)} = 0$. We define a local version $\t{D}^{1,p}_{\loc}(Y)$ of the Dirichlet space exactly as we did for the Newtonian space, writing $u \in \t{D}^{1,p}_{\loc}(Y)$ for a measurable function $u: Y \rightarrow [-\infty,\infty]$ if $u|_{B} \in \t{D}^{1,p}(B)$ for each ball $B \subset Y$. 

\begin{rem}\label{local discrepancy} When $Y$ is proper the spaces $\t{N}^{1,p}_{\loc}(Y)$ and $\t{D}^{1,p}_{\loc}(Y)$ can equivalently be described as the set of measurable functions $u: Y \rightarrow [-\infty,\infty]$ such that each $x \in Y$ has an open neighborhood $U_{x}$ on which $u|_{U_{x}} \in \t{N}^{1,p}(U_{x})$ or $u|_{U_{x}} \in \t{D}^{1,p}(U_{x})$ respectively.  This is the definition of the local spaces given in \cite[Chapter 7]{HKST}. The definition of the local spaces that we use here is more restrictive when $Y$ is not proper. In Section \ref{sec:trace} we will be applying this definition in the particular case that $Y$ is incomplete.
\end{rem}

In order to obtain further properties of these function spaces we need to make some additional assumptions on $Y$. The metric measure space $(Y,d,\mu)$ is \emph{doubling} if there is a constant $C_{\mu} \geq 1$ such that for each ball $B \subset Y$ we have
\begin{equation}\label{doubling inequality}
\mu(2B) \leq C_{\mu}\mu(B).
\end{equation} 
The metric measure space $(Y,d,\mu)$ \emph{supports a $p$-Poincar\'e inequality} for a given $p \geq 1$ if for all measurable functions $u: Y \rightarrow [-\infty,\infty]$ that are integrable on balls, all $p$-integrable upper gradients $g: Y \rightarrow [0,\infty]$ for $u$, and all balls $B \subset Y$, 
\begin{equation}\label{weak Poincare}
\dashint_{B}|u-u_{B}|\,d\mu \leq C_{\mathrm{PI}} \diam(B)\left(\dashint_{B}g^{p}\,d\mu\right)^{1/p},
\end{equation}  
for a constant $C_{\mathrm{PI}} > 0$. We note that if $(Y,d,\mu)$ supports a $p$-Poincar\'e inequality then $(Y,d,\mu)$ supports a $q$-Poincar\'e inequality for all $q \geq p$ by H\"older's inequality, with new constants depending only on $p$, $q$, and $C_{\mathrm{PI}}$. We will be assuming for the rest of this subsection that $Y$ is a geodesic metric space that contains at least two points and that $(Y,d,\mu)$ is a doubling metric measure space that supports a $p$-Poincar\'e inequality for a given $p \geq 1$.

Under these assumptions $Y$ supports the following stronger form of the Poincar\'e inequality for any ball $B \subset Y$, any integrable function $u: B \rightarrow \R$, and any $p$-integrable upper gradient $g$ of $u$ in $B$, 
\begin{equation}\label{pre Poincare}
\left(\dashint_{B}|u-u_{B}|^{p}\,d\mu\right)^{1/p} \leq C_{0} \diam(B)\left(\dashint_{B}g^{p}\,d\mu\right)^{1/p},
\end{equation} 
with the constant $C_{0} > 0$ depending only on the constants of the $p$-Poincar\'e inequality for $Y$ and the doubling constant for $\mu$ \cite[Remark 9.1.19]{HKST}. This follows by applying H\"older's inequality to a class of Sobolev-Poincar\'e inequalities for $Y$ \cite[Theorem 9.1.15]{HKST} (see also \cite{HK00}). This same application of H\"older's inequality also shows that \eqref{weak Poincare} similarly holds when $u: B \rightarrow [-\infty,\infty]$ is defined only on $B$ and $g: B \rightarrow [0,\infty]$ is a $p$-integrable upper gradient of $u$ on $B$. We will usually use the inequality \eqref{pre Poincare} in the reformulated form,
\begin{equation}\label{stronger Poincare}
\int_{B}|u-u_{B}|^{p}\,d\mu \leq C_{0}^{p} \diam(B)^{p}\int_{B}g^{p}\,d\mu.
\end{equation} 
In keeping with standard conventions, we will refer to both \eqref{pre Poincare} and \eqref{stronger Poincare} as \emph{$(p,p)$-Poincar\'e inequalities}. We will use a generic constant $C > 0$ in place of the specific constants $C_{\mathrm{PI}}$ and $C_{0}$ in \eqref{weak Poincare} and \eqref{stronger Poincare} when applying these inequalities.

\begin{rem}\label{constant quotient}The $p$-Poincar\'e inequality \eqref{weak Poincare} implies that any measurable function $u:Y \rightarrow [-\infty,\infty]$ with $\|u\|_{D^{1,p}(Y)} = 0$ is constant q.e., see \cite[Proposition 7.5.2, Theorem 9.3.4]{HKST}. Hence $D^{1,p}(Y)$ can equivalently be viewed as the quotient of $\t{D}^{1,p}(Y)$ by all measurable functions $u: Y \rightarrow [-\infty,\infty]$ that are constant q.e.
\end{rem}

 The $(p,p)$-Poincar\'e inequality \eqref{stronger Poincare} implies by a straightforward truncation argument that all functions $u \in \t{D}^{1,p}_{\loc}(Y)$ are $p$-integrable over balls. The proof is very similar to the proof of \cite[Lemma 8.1.5]{HKST}. 

\begin{prop}\label{integrable over balls}
Let $B \subset Y$ be any ball and let $u: B \rightarrow \R$ be a measurable function such that $u$ has an upper gradient $g: B \rightarrow [0,\infty]$ with $g \in L^{p}(B)$. Then $u \in L^{p}(B)$. Consequently $\t{D}^{1,p}(B) = \t{N}^{1,p}(B)$ (as sets). Therefore $\t{D}^{1,p}_{\loc}(Y) = \t{N}^{1,p}_{\loc}(Y)$ and $\t{D}^{1,p}(Y) \subset \t{N}^{1,p}_{\loc}(Y)$.
\end{prop}

\begin{proof}
For each $n \in \N$ we let $u_{n} = \max\{-n,\min\{u,n\}\}$. Then $|u_{n}| \leq n$ and consequently $u_{n}$ is integrable over $B$. Furthermore $g$ is also an upper gradient for $u_{n}$ on $B$ \cite[Proposition 6.3.23]{HKST}. Thus by the $p$-Poincar\'e inequality \eqref{weak Poincare} we have
\[
\dashint_{B}|u_{n}-(u_{n})_{B}|\,d\mu \leq C \diam(B)\left(\dashint_{B}g^{p}\,d\mu\right)^{1/p}. 
\]
Since 
\[
\dashint_{B} \dashint_{B} |u_{n}(x)-u_{n}(y)|\,d\mu(x)d\mu(y) \leq 2\dashint_{B}|u_{n}-(u_{n})_{B}|\,d\mu,
\]
we can apply the monotone convergence theorem to the sequence of functions $\varphi_{n}: B \times B \rightarrow \R$ given by $\varphi_{n}(x,y) = |u_{n}(x)-u_{n}(y)|$ to obtain that
\[
\dashint_{B} \dashint_{B} |u(x)-u(y)|\,d\mu(x)d\mu(y) \leq 2C \diam(B)\left(\dashint_{B}g^{p}\,d\mu\right)^{1/p} < \infty.
\]
It follows that the function $x \rightarrow |u(x)-u(y)|$ is integrable over $B$ for a.e.~ $y \in B$, which immediately implies that $u$ is integrable over $B$. Applying the $(p,p)$-Poincar\'e inequality \eqref{stronger Poincare} then gives $u \in L^{p}(B)$. We conclude that $\t{D}^{1,p}(B) = \t{N}^{1,p}(B)$. The equality $\t{D}^{1,p}_{\loc}(Y) = \t{N}^{1,p}_{\loc}(Y)$ then follows from the definitions. The inclusion $\t{D}^{1,p}(Y) \subset \t{N}^{1,p}_{\loc}(Y)$ then follows from the inclusion $\t{D}^{1,p}(Y) \subset \t{D}^{1,p}_{\loc}(Y)$. 
\end{proof}

The $(p,p)$-Poincar\'e inequality \eqref{stronger Poincare} can be used to show that $D^{1,p}(Y)$ is a Banach space. In fact one obtains that $D^{1,p}(B)$ is a Banach space for any ball $B \subset Y$ as well.

\begin{prop}\label{Dirichlet Banach}
The normed space $D^{1,p}(Y)$ is a Banach space. The same is true of $D^{1,p}(B)$ for any ball $B \subset Y$.
\end{prop}

\begin{proof}
We prove the second claim first. By Proposition \ref{integrable over balls} we have $\t{D}^{1,p}(B) = \t{N}^{1,p}(B)$ as sets. Since $\|u\|_{D^{1,p}(B)} \leq \|u\|_{N^{1,p}(B)}$ for any $u \in \t{D}^{1,p}(B)$, it follows that the quotient projection $N^{1,p}(B) \rightarrow D^{1,p}(B)$ is continuous. Let $\{u_{n}\}$ be a Cauchy sequence in $D^{1,p}(B)$, and choose a sequence of representatives $\{\t{u}_{n}\} \subset \t{D}^{1,p}(B) = \t{N}^{1,p}(B)$. Since constant functions have norm $0$ in $\t{D}^{1,p}(B)$, by adding an appropriate constant to $\t{u}_{n}$ for each $n$ we can assume that $(\t{u}_{n})_{B} = 0$ for all $n$.  

For each $m,n \in \N$ we let $g_{m,n}$ be a minimal $p$-weak upper gradient of $\t{u}_{m}-\t{u}_{n}$ in $B$. Then inequality \eqref{stronger Poincare} implies for $m,n \in \N$ that
\[
\int_{B}|\t{u}_{m}-\t{u}_{n}|^{p}\,d\mu \leq C\diam(B)^{p}\int_{B}g_{m,n}^{p}\,d\mu.
\]
By hypothesis the right side converges to $0$ as $m,n\rightarrow \infty$. Thus $\{\t{u}_{n}\}$ defines a Cauchy sequence in $L^{p}(B)$, which implies that $\{\t{u}_{n}\}$ defines a Cauchy sequence in $\t{N}^{1,p}(B)$. Since $N^{1,p}(B)$ is a Banach space, it follows that there is some function $\t{u} \in N^{1,p}(B)$ such that $\t{u}_{n} \rightarrow \t{u}$ in $N^{1,p}(B)$. Letting $u$ denote the projection of $\t{u}$ to $D^{1,p}(B)$, this implies that $u_{n} \rightarrow u$ in $D^{1,p}(B)$. It follows that $D^{1,p}(B)$ is a Banach space. 

Now let $\{u_{n}\}$ be a Cauchy sequence in $D^{1,p}(Y)$ and choose a sequence of representatives $\{\t{u}_{n}\} \subset \t{D}^{1,p}(Y)$. Fix a point $x \in Y$ and for each $k \in \N$ let $B_{k} = B_{Y}(x,k)$ be the ball of radius $k$ centered at $x$. By adding appropriate constants to $\t{u}_{n}$ we can arrange that $(\t{u}_{n})_{B_{1}} = 0$ for each $n$. For each $n,k \in \N$ we then set $\t{u}_{n,k} = \t{u}_{n}-(\t{u}_{n})_{B_{k}}$. Since $\{\t{u}_{n,k}\}$ also defines a Cauchy sequence in $D^{1,p}(B_{k})$, the argument in the first part of the proof then shows that for each $k \in \N$ there is a function $v_{k} \in \t{N}^{1,p}(B_{k})$ such that $\|(\t{u}_{n,k} - v_{k})|_{B_{k}}\|_{N^{1,p}(B_{k})} \rightarrow 0$ as $n \rightarrow \infty$. 

Restricting $\t{u}_{n,k}$ to $B_{j}$ for some $j < k$ shows that 
\[
\|(\t{u}_{n,k}-v_{k})|_{B_{j}}\|_{N^{1,p}(B_{j})} = \|(\t{u}_{n}- (\t{u}_{n})_{B_{k}}-v_{k})|_{B_{j}}\|_{N^{1,p}(B_{j})}\rightarrow 0.
\]
Since $\|(\t{u}_{n,j} - v_{j})|_{B_{j}}\|_{N^{1,p}(B_{j})} \rightarrow 0$ as well, it follows by the triangle inequality that 
\begin{equation}\label{j and k}
\lim_{n \rightarrow \infty} \|((\t{u}_{n})_{B_{k}}-(\t{u}_{n})_{B_{j}}-v_{k}+v_{j})|_{B_{j}}\|_{N^{1,p}(B_{j})} = 0.
\end{equation}
Applying this to the special case $j = 1$, we obtain that 
\[
\lim_{n \rightarrow \infty} \|((\t{u}_{n})_{B_{k}}-v_{k}+v_{1})|_{B_{1}}\|_{N^{1,p}(B_{1})} = 0.
\]
It follows that for each $k \in \N$ there is a constant $c_{k} \in \R$ such that $v_{k}-v_{1} = c_{k}$ a.e.~ on $B_{1}$ and $c_{k} = \lim_{n \rightarrow \infty} (\t{u}_{n})_{B_{k}}$. Note that $c_{1} = 0$. Applying this to \eqref{j and k}, we conclude that $v_{k}-v_{j} = c_{k}-c_{j}$ a.e.~ on $B_{j}$. We define $\t{u}: Y \rightarrow \R$ by setting $\t{u} = v_{k} - c_{k}$ on $B_{k}$; since for $j \leq k$ we have $v_{k}-c_{k} = v_{j}-c_{j}$ a.e.~ on $B_{j} \subset B_{k}$ it follows that $\t{u}$ is a well-defined measurable function on $Y$. 

It remains to show that $\t{u} \in \t{D}^{1,p}(Y)$ and that, denoting its projection to $D^{1,p}(Y)$ by $u$, we have $u_{n} \rightarrow u$ in $D^{1,p}(Y)$. By construction we have that $u_{n} \rightarrow u$ in $D^{1,p}(B_{k})$ for each $k \in \N$ since $u_{n}|_{B_{k}} \sim \t{u}_{n,k}$ and $\t{u}|_{B_{k}} \sim v_{k}$ in $D^{1,p}(B_{k})$ for each $k$. Letting $g_{k}$ denote a minimal $p$-weak upper gradient of $\t{u}|_{B_{k}}$ in $D^{1,p}(B_{k})$, it follows that we will have 
\[
\|g_{k}\|_{L^{p}(B_{k})} \leq \limsup_{n \rightarrow \infty}\|u_{n}\|_{D^{1,p}(B_{k})} \leq \limsup_{n \rightarrow \infty}\|u_{n}\|_{D^{1,p}(Y)} < \infty,
\]
with the final inequality following from the assumption that $\{u_{n}\}$ is a Cauchy sequence in $D^{1,p}(Y)$. By uniqueness of minimal $p$-weak upper gradients we have that $g_{j} = g_{k}$ a.e.~ on $B_{j}$ for each $j < k$. We then extend $g_{k}$ to $Y$ by setting $g_{k}(x) = 0$ for $x \notin B_{k}$ and define a Borel function $g: Y \rightarrow [0,\infty]$ by $g(x) = \sup_{k \in \N} g_{k}(x)$. Then $g$ defines a $p$-weak upper gradient for $\t{u}$ on $Y$. Furthermore we have $g = g_{k}$ a.e.~ on $B_{k}$ for each $k$. By the monotone convergence theorem applied to the sequence of functions $g_{(k)}(x) = \sup_{1 \leq j \leq k} g_{j}(x)$, we then conclude that
\[
\|g\|_{L^{p}(Y)} \leq \limsup_{n \rightarrow \infty}\|u_{n}\|_{D^{1,p}(Y)} < \infty
\] 
Thus $\t{u} \in \t{D}^{1,p}(Y)$. 

We thus conclude that $\t{u}-\t{u}_{n} \in \t{D}^{1,p}(Y)$ for each $n \in \N$. We let $f_{n}$ be a minimal $p$-integrable $p$-weak upper gradient of $\t{u}-\t{u}_{n}$ on $Y$ for each $n$. To show that $u_{n}  \rightarrow u$ in $D^{1,p}(Y)$ it suffices to show that $\|f_{n}\|_{L^{p}(Y)} \rightarrow 0$. For each $k \in \N$ we have that $f_{n}|_{B_{k}}$ is a minimal $p$-weak upper gradient of $(\t{u}_{n,k}-v_{k})|_{B_{k}}$.  It follows that 
\begin{align*}
\|f_{n}\|_{L^{p}(B_{k})} &= \|(\t{u}_{n,k}-v_{k})|_{B_{k}}\|_{D^{1,p}(B_{k})} \\
&\leq \limsup_{m \rightarrow \infty} \|(\t{u}_{n,k}-\t{u}_{m,k})|_{B_{k}}\|_{D^{1,p}(B_{k})} \\
&= \limsup_{m \rightarrow \infty} \|(\t{u}_{n}-\t{u}_{m})|_{B_{k}}\|_{D^{1,p}(B_{k})} \\
&\leq \limsup_{m \rightarrow \infty} \|\t{u}_{n}-\t{u}_{m}\|_{D^{1,p}(Y)}.
\end{align*}
By the monotone convergence theorem we conclude that
\[
\|f_{n}\|_{L^{p}(Y)} \leq \limsup_{m \rightarrow \infty} \|\t{u}_{n}-\t{u}_{m}\|_{D^{1,p}(Y)}.
\]
The right side converges to $0$ as $n \rightarrow \infty$ since $\{u_{n}\}$ is a Cauchy sequence in $D^{1,p}(Y)$. We conclude that $\|f_{n}\|_{L^{p}(Y)} \rightarrow 0$ as $n \rightarrow \infty$, which implies that $u_{n} \rightarrow u$ in $D^{1,p}(Y)$. 
\end{proof}

We next discuss quasicontinuity.  A function $u$ on $Y$ is \emph{$C_{p}^{Y}$-quasicontinuous} if for each $\eta > 0$ there is an open set $U \subset Y$ such that $C_{p}^{Y}(U) < \eta$ and $u|_{Y \backslash U}$ is continuous. We will later make use of the following theorem of Bj\"orn-Bj\"orn-Shanmugalingam.

\begin{thm}\label{outer capacity}\cite{BBS08}
Let $(Y,d,\mu)$ be a metric measure space that is complete, doubling, and supports a $p$-Poincar\'e inequality. Then every $u \in \t{N}^{1,p}_{\loc}(Y)$ is $C_{p}^{Y}$-quasicontinuous.  Moreover $C_{p}^{Y}$ is an outer capacity, i.e., for any subset $G \subset Y$, 
\[
C_{p}^{Y}(G) = \inf_{U} C_{p}^{Y}(U),
\]
with the infimum being taken over all open subsets $U$ with $G \subset U$. 
\end{thm}

We remark that, since $\t{D}^{1,p}(Y) \subset N^{1,p}_{\loc}(Y)$ under our hypotheses on $Y$ by Proposition \ref{integrable over balls}, we can freely apply Theorem \ref{outer capacity} to any function $u \in \t{D}^{1,p}(Y)$ if we further assume that $Y$ is complete. In particular all functions in $\t{D}^{1,p}(Y)$ are $C_{p}^{Y}$-quasicontinuous.

%\begin{rem}\label{localize measurability}
%Let $u: Y \rightarrow \R$ be a function with $u \in \t{D}^{1,p}_{\loc}(Y)$. For each integer $m \geq 1$ we put $u_{m} = \max\{-m,\min\{u,m\}\}$ so that $|u_{m}| \leq m$. Choose a maximal $1$-separated subset $\{x_{n}\}_{n \in J}$ of $Y$ and let $\{\psi_{n}\}_{n \in J}$ be an associated Lipschitz partition of unity (see Proposition \ref{partition of unity} below). Let $L$ be an upper bound for the Lipschitz constants of each function $\psi_{n}$ and set $B_{n} = B(x_{n},1)$. Then a straightforward calculation shows for each $n \in J$ that if $g$ is a $p$-integrable upper gradient for $u$ on $2B_{n}$ then $(g+Lm)\chi_{2B_{n}}$ is a $p$-integrable upper gradient for $\psi_{n} u_{m}$ on $Y$ and therefore $\psi_{n} u_{m}$ is measurable. Since $\sum_{n \in J} \psi_{n} u_{m} = u_{m}$ it follows that $u_{m}$ is measurable for each $m$, and since $u = \lim_{m \rightarrow \infty} u_{m}$ it follows that $u$ is measurable. Thus the measurability condition for functions $u \in \t{D}^{1,p}_{\loc}(Y)$ can be removed as well. 
%\end{rem}

\subsection{Besov spaces}\label{subsec:Besov} In this section we will be dropping the requirement that our space supports a Poincar\'e inequality, so we will be switching to different symbols $Z$ for the metric space and $\nu$ for the measure to reflect this. We will assume that $(Z,d,\nu)$ is a doubling metric measure space. Similarly to before we write $C_{\nu}$ for the doubling constant of the measure $\nu$ appearing in the doubling inequality \ref{doubling inequality}.  For a given $p \geq 1$ and $\theta > 0$ we recall  for a function $u \in L^{p}_{\loc}(Z)$ the definitions \eqref{Besov norm} and \eqref{check Besov norm} of the Besov norms $\|u\|_{B^{\theta}_{p}(Z)}$ and $\|u\|_{\ch{B}^{\theta}_{p}(Z)}$ of $u$. The space $\t{B}^{\theta}_{p}(Z)$ is the subspace of $L^{p}_{\loc}(Z)$ for which we have $\|u\|_{B^{\theta}_{p}(Z)} < \infty$ and the space $\ch{B}^{\theta}_{p}(Z) = L^{p}(Z) \cap \t{B}^{\theta}_{p}(Z)$ is characterized by the finiteness of the norm  $\|u\|_{\ch{B}^{\theta}_{p}(Z)}$. 

We let $B^{\theta}_{p}(Z) = \t{B}_{p}^{\theta}(Z)/\sim$ be the quotient of $\t{B}_{p}^{\theta}(Z)$ by the equivalence relation $u \sim v$ if $\|u-v\|_{B^{\theta}_{p}(Z)} = 0$. It is easy to see from the definition \eqref{Besov norm} that $\|u\|_{B^{\theta}_{p}(Z)} = 0$ if and only if there is a constant $c \in \R$ such that $u \equiv c$ a.e.; hence we can equivalently characterize $B^{\theta}_{p}(Z)$ as being the quotient of $\t{B}_{p}^{\theta}(Z)/\sim$ by the equivalence relation $u \sim v$ if $u-v$ is constant a.e.~ on $Z$. The space $B^{\theta}_{p}(Z)$ is a Banach space as well; we will deduce this from Proposition \ref{Dirichlet Banach} in Proposition \ref{Besov Banach} toward the end of the paper. This fact has been obtained previously in other treatments of the topic, see e.g. \cite{S16}.

We define the Besov capacity $C_{\ch{B}^{\theta}_{p}}^{Z}$ analogously to the $C_{p}^{Y}$-capacity for a subset $G\subset Z$,
\[
C_{\ch{B}^{\theta}_{p}}^{Z}(G) = \inf_{u} \|u\|_{\ch{B}^{\theta}_{p}(Z)}^{p},
\]
with the infimum being taken over all $u \in \ch{B}^{\theta}_{p}(Z)$ such that $u \geq 1$ $\nu$-a.e.~ on a neighborhood of $G$. We have to take a neighborhood of $G$ due to the lack of pointwise control over functions $u \in \ch{B}^{\theta}_{p}(Z)$. A function $u \in \t{B}^{\theta}_{p}(Z)$ is \emph{$C_{\ch{B}^{\theta}_{p}}^{Z}$-quasicontinuous} if for each $\eta > 0$ there is an open set $U \subset Z$ such that $C_{\ch{B}^{\theta}_{p}}^{Z}(U) < \eta$ and $u|_{Z \backslash U}$ is continuous. 

In connection with the Besov capacity the following truncation lemma is useful. 

\begin{lem}\label{truncate Besov}
Let $f \in \ch{B}^{\theta}_{p}(Z)$ be given and define
\[
\hat{f} = \max\{0,\min\{1,f\}\}.
\] 
Then $\hat{f} \in \ch{B}^{\theta}_{p}(Z)$ with $\|\hat{f}\|_{\ch{B}^{\theta}_{p}(Z)} \leq \|f\|_{\ch{B}^{\theta}_{p}(Z)}$. 
\end{lem}

\begin{proof}
For $\nu$-a.e.~ $x,y \in Z$ we have
\[
|\hat{f}(x)-\hat{f}(y)| \leq |f(x)-f(y)|.
\]
This immediately implies that $\|\hat{f}\|_{B^{\theta}_{p}(Z)} \leq \|f\|_{B^{\theta}_{p}(Z)}$. Since we also have $|\hat{f}(x)| \leq |f(x)|$ for $\nu$-a.e.~ $x \in Z$, we conclude that $\|\hat{f}\|_{L^{p}(Z)} \leq \|f\|_{L^{p}(Z)}$ and therefore $\|\hat{f}\|_{\ch{B}^{\theta}_{p}(Z)} \leq \|f\|_{\ch{B}^{\theta}_{p}(Z)}$.
\end{proof}

We have the following key estimate, which follows from the estimates of \cite[Theorem 5.2]{GKS10} with $\alpha > 1$ in place of $2$; see also \cite[Lemma 9.9]{BBS21} which gives the analogous estimate on a bounded space.

\begin{lem}\label{Besov estimate}
For any $\alpha > 1$ and $u \in B^{\theta}_{p}(Z)$ we have
\[
\|u\|_{B^{\theta}_{p}(Z)}^{p} \asymp_{C} \sum_{n\in \Z} \int_{Z} \dashint_{B_{Z}(x,\alpha^{-n})}\frac{|u(x)-u(y)|^{p}}{\alpha^{-n \theta p}}\,d\nu(y)d\nu(x),
\]
with $C = C(\alpha,p,\theta,C_{\nu})$.
\end{lem}

For a function $\psi: Z \rightarrow \R$ we write
\[
\mathrm{supp}(\psi):=\overline{\{z\in Z:\psi(z) \neq 0\}},
\]
for the support of $\psi$, defined as the closure of the points on which $\psi$ is nonzero. We say that $\psi$ has \emph{bounded support} if $\supp(\psi)$ is a bounded subset of $Z$, and we say that $\psi$ is \emph{compactly supported} if $\supp(\psi)$ is compact. These two notions coincide when $Z$ is proper. Via a somewhat lengthy computation one can compute using the estimate of Lemma \ref{Besov estimate} that Lipschitz functions on $Z$ with bounded support belong to $\ch{B}^{\theta}_{p}(Z)$ for each $0 < \theta < 1$. This shows in particular that these Besov spaces are always nontrivial (i.e., contain functions that are not constant $\nu$-a.e.) whenever $Z$ contains more than one point. 

\section{Hyperbolic fillings}\label{sec:fillings}

In this section we will review the results that we need from \cite{Bu20} on hyperbolic fillings of metric spaces and prove some additional results that we will need in subsequent sections.

\subsection{Gromov hyperbolic spaces}\label{subsec:hyp} 
We start by recalling some general results regarding Gromov hyperbolic spaces, for which a good reference is \cite{BS07}. Let $(X,d)$ be a geodesic metric space. A \emph{geodesic triangle} $\Delta$ in $X$ consists of three points $x,y,z \in X$ together with geodesics joining these points to one another. Writing $\Delta = \gamma_{1} \cup \gamma_{2} \cup \gamma_{3}$ as a union of its edges, we say that $\Delta$ is \emph{$\delta$-thin} for a given $\delta \geq 0$ if for each point $p \in \gamma_{i}$, $i =1,2,3$, there is a point $q \in \gamma_{j}$ with $d(p,q) \leq \delta$ and $i \neq j$. A geodesic metric space $X$ is \emph{Gromov hyperbolic} if there is a $\delta \geq 0$ such that all geodesic triangles in $X$ are $\delta$-thin; in this case we will also say that $X$ is \emph{$\delta$-hyperbolic}. When considering Gromov hyperbolic spaces $X$ we will usually use the generic distance notation $|xy|:=d(x,y)$ for the distance between $x$ and $y$ in $X$ and the generic notation $xy$ for a geodesic connecting two points $x,y \in X$, even when this geodesic is not unique.  

The Gromov boundary $\p X$ of a proper geodesic $\delta$-hyperbolic space $X$ is defined to be the collection of all geodesic rays $\gamma: [0,\infty) \rightarrow X$ up to the equivalence relation of two rays being equivalent if they are at a bounded distance from one another. We will often refer to the point $\omega \in \p X$ corresponding to a geodesic ray $\gamma$ as the \emph{endpoint} of $\gamma$. Using the Arzela-Ascoli theorem it is easy to see in a proper geodesic $\delta$-hyperbolic space that for any points $x,y \in X \cup \p X$ there is a geodesic $\gamma$ joining $x$ to $y$. We will continue to write $xy$ for any such choice of geodesic joining $x$ to $y$. 

%We will allow our geodesic triangles $\Delta$ to have vertices on $\p X$, in which case we will still write $\Delta = xyz$ if $\Delta$ has vertices $x,y,z$. We remark that geodesic triangles with vertices in $X \cup \p X$ are $10\delta$-thin by \cite[Lemma 2.2]{Bu20}. 

%As in our previous work \cite{Bu20}, we will use the notation $\p X$ for the Gromov boundary of $X$ even though it conflicts with the notation $\p \Omega = \bar{\Omega}\backslash \Omega$ for the metric boundary of a metric space $(\Omega,d)$ inside its completion $\bar{\Omega}$. Since we always assume that $X$ is proper we will always have $\bar{X} = X$, so the metric boundary of $X$ will always be trivial. Thus there will be no ambiguity in using $\p X$ for the Gromov boundary as well.  

For a geodesic ray $\gamma:[0,\infty) \rightarrow X$, the \emph{Busemann function} $b_{\gamma}: X \rightarrow \R$ associated to $\gamma$ is defined by the limit
\begin{equation}\label{first busemann definition}
b_{\gamma}(x) = \lim_{t \rightarrow \infty} |\gamma(t)x|-t. 
\end{equation} 
We then define
\begin{equation}\label{extension busemann definition}
\mathcal{B}(X) = \{b_{\gamma}+s: \text{$\gamma$ a geodesic ray in $X$, $s \in \R$}\},
\end{equation}
and refer to any function $b \in \mathcal{B}(X)$ as a Busemann function on $X$. See \cite[(1.4-1.5)]{Bu20} for further details on these definitions. The Busemann functions $b \in \mathcal{B}(X)$ are all $1$-Lipschitz functions on $X$. For a Busemann function $b$ of the form $b = b_{\gamma} + s$ for some $s \in \R$, we define the endpoint $\omega \in \p X$ of $\gamma$ to be the \emph{basepoint} of $b$ and say that $b$ is \emph{based at $\omega$}.

Fix a Busemann function $b \in \mathcal{B}(X)$ with basepoint $\omega$. The \emph{Gromov product} of $x,y \in X$ based at $b$ is defined by 
\begin{equation}\label{Gromov Busemann product}
(x|y)_{b} = \frac{1}{2}(b(x) + b(y) - |xy|). 
\end{equation}
We note that Gromov products are usually defined with basepoints in $X$, however we will only need to consider Gromov products based at Busemann functions in this paper. The following statements briefly summarize a more extensive discussion of Gromov products in \cite[Section 2]{Bu20}, which is based on \cite[Chapter 3]{BS07}. A sequence $\{x_{n}\}$ \emph{converges to infinity with respect to $\omega$} if $(x_{m}|x_{n})_{b} \rightarrow \infty$ as $m,n \rightarrow \infty$, and two sequences $\{x_{n}\}$ and $\{y_{n}\}$ are \emph{equivalent with respect to $\omega$} if $(x_{n}|y_{n})_{b} \rightarrow \infty$ as $n \rightarrow \infty$. These definitions do not depend on the choice of Busemann function based at $\omega$ by \cite[Lemma 2.5]{Bu20}. The \emph{Gromov boundary relative to $\omega$} is defined to be the set $\p_{\omega}X$ of all equivalence classes of sequences converging to infinity with respect to $\omega$.

By \cite[Proposition 3.4.1]{BS07} we have a canonical identification of $\p_{\omega}X$ with the complement $\p X \backslash\{\omega\}$ of $\omega$ in the Gromov boundary $\p X$. We will thus use the notation $\p_{\omega}X = \p X \backslash \{\omega\}$ throughout the rest of the paper. For $\xi \in \p_{\omega} X$ and a sequence $\{x_{n}\}$ that converges to infinity with respect $\omega$ we will write $\{x_{n}\} \in \xi$ if $\{x_{n}\}$ belongs to the equivalence class of $\xi$, and in this case we will also say that $\{x_{n}\}$ \emph{converges to $\xi$}.

Gromov products based at Busemann functions $b \in \mathcal{B}(X)$ can be extended to points of $\p_{\omega} X$ by defining the Gromov product of equivalence classes $\xi$, $\zeta \in \p_{\omega} X$ based at $b$ to be
\begin{equation}\label{full extended definition}
(\xi |\zeta)_{b} = \inf \liminf_{n \rightarrow \infty}(x_{n}|y_{n})_{b},
\end{equation}
with the infimum taken over all sequences $\{x_{n}\} \in \xi$, $\{y_{n}\} \in \zeta$; we leave this expression undefined when $\xi = \zeta = \omega$. As a consequence of \cite[Lemma 2.2.2]{BS07}, \cite[Lemma 3.2.4]{BS07}, and the discussion in \cite[Section 2.2]{Bu20}, for any choices of sequences  $\{x_{n}\} \in  \xi$ and  $\{y_{n}\} \in \zeta$ we have
\begin{equation}\label{sequence approximation}
(\xi |\zeta)_{b}  \leq \liminf_{n \rightarrow \infty}(x_{n}|y_{n})_{b} \leq \limsup_{n \rightarrow \infty}(x_{n}|y_{n})_{b} \leq (\xi |\zeta)_{b} + c(\delta),
\end{equation}
with the constant $c(\delta)$ depending only on $\delta$. One may take $c(\delta)  = 600\delta$. For $x \in X$ and $\xi \in \p X$ the Gromov product based at $b$ is defined analogously as 
\begin{equation}\label{extended definition}
(x |\xi)_{b} = \inf \liminf_{n \rightarrow \infty}(x|x_{n})_{b},
\end{equation}
and the analogous inequality \eqref{sequence approximation} holds with the same constants. By \cite[(2.10)]{Bu20}, for all $x,y \in X \cup \p X$ with $(x,y) \neq (\omega,\omega)$ we have
\begin{equation}\label{both busemann boundary}
(x|y)_{b} \leq \min\{b(x),b(y)\}+c(\delta), 
\end{equation}
where we set $b(\xi) = \infty$ for $\xi \in \p_{\omega} X$. We may also take $c(\delta) = 600\delta$ here. 

Gromov products based at Busemann functions $b \in \mathcal{B}(X)$ can be used to define visual metrics on the Gromov boundary $\p_{\omega} X$ relative to $\omega$. We refer to \cite[Chapters 2-3]{BS07} as well as \cite[Section 2.3]{Bu20} for precise details on this topic. We will summarize the results we need here below. For $\e > 0$ we define for $\xi$, $\zeta \in \p_{\omega} X$,
\begin{equation}\label{visual quasi}
\vartheta_{\e,b}(\xi,\zeta) = e^{-\e (\xi|\zeta)_{b}}.
\end{equation}
This expression may not define a metric on $\p_{\omega} X$, since the triangle inequality may not hold. However there is always $\e_{*} = \e_{*}(\delta) > 0$ depending only on $\delta$ such that for $0 < \e \leq \e_{*}$ the function $\vartheta_{\e,b}$ is $4$-biLipschitz to a metric $\vartheta$ on $\p_{\omega} X$. We refer to any metric $\vartheta$ on $\p_{\omega}X$ that is biLipschitz to $\vartheta_{\e,b}$ as a \emph{visual metric} on $\p_{\omega} X$ based at $b$ and refer to $\e$ as the \emph{parameter} of $\vartheta$. We give $\p_{\omega} X$ the topology associated to a visual metric based at $b$ for any Busemann function $b$ based at $\omega$. When equipped with a visual metric $\p_{\omega}X$ is a locally compact metric space. 

\subsection{Structure of hyperbolic fillings}Let $(Z,d)$ be a metric space and let $\alpha,\tau > 1$ be given parameters. We will assume these parameters satisfy
\begin{equation}\label{tau requirement}
\tau > \left\{3,\frac{\alpha}{\alpha-1}\right\}.
\end{equation} 
We write $B_{Z}(z,r)$ for the ball of radius $r$ centered at $z$ in $Z$. We construct a hyperbolic filling $X$ of $Z$ associated to these parameters as in \cite{Bu20}. A subset $S \subset Z$ is \emph{$r$-separated} for a given $r > 0$ if for each $x,y \in S$ we have $d(x,y) \geq r$. For each $n \in \Z$ we select a maximal $\alpha^{-n}$-separated subset $S_{n}$ of $Z$. Then for each $n \in \Z$ the balls $B_{Z}(z,\alpha^{-n})$, $z \in S_{n}$, cover $Z$.

The vertex set of $X$ has the form 
\[
V = \bigcup_{n \in \Z} V_{n}, \;\;\; V_{n} = \{(z,n):z \in S_{n}\}.
\]
To each vertex $v = (z,n)$ we associate the dilated ball $B(v) := B_{Z}(z,\tau \alpha^{-n})$. We define a projection $\pi: V \rightarrow Z$ by setting $\pi(z,n) = z$ and define the \emph{height function} $h: V \rightarrow \Z$ by $h(z,n) = n$. For a vertex $v \in V$ we will sometimes write $v$ in place of $\pi(v)$ and consider $v$ both as a point of $Z$ and a vertex of $X$, except in places where this could cause confusion. 

We place an edge between two distinct vertices $v$ and $w$ if $|h(v)-h(w)| \leq 1$ and $B(v) \cap B(w) \neq \emptyset$. For vertices $v,w \in V$ we will write $v \sim w$ if there is an edge joining $v$ to $w$. We write $E$ for the set of edges in $X$.  The resulting graph $X$ is connected by \cite[Proposition 5.5]{Bu20}. We give $X$ the geodesic metric in which all edges have unit length. We extend the height function piecewise linearly to the edges of $X$ to define a $1$-Lipschitz function $h: X \rightarrow \R$. By \cite[Proposition 5.9]{Bu20} we have that the hyperbolic filling $X$ is $\delta$-hyperbolic with $\delta = \delta(\alpha,\tau)$.

We say that an edge $e$ is \emph{vertical} if it connects two vertices of different heights and \emph{horizontal} if it connects two vertices of the same height. A \emph{vertical edge path} in $X$ is a sequence of edges joining a sequence of vertices $\{v_{k}\}$ with $h(v_{k+1}) = h(v_{k}) +1$ for each $k$ or $h(v_{k+1}) = h(v_{k}) -1$ for each $k$. In the first case we say that the edge path is \emph{ascending} and in the second case we say that the edge path is \emph{descending}. We observe that vertical edge paths are always geodesics in $X$ since edges in $X$ can only join two vertices of the same or adjacent heights. Thus we will also refer to vertical edge paths as \emph{vertical geodesics}. The following is an easy observation from the construction of $X$. 

\begin{lem}\label{height connection}\cite[Lemma 5.2]{Bu20}
Let $v,w \in V$ with $h(v) \neq h(w)$ and $B(v) \cap B(w) \neq \emptyset$. Then there is a vertical geodesic connecting $v$ to $w$.
\end{lem}

The next lemma concerns finding vertices $v \in V$ for which the associated ball $B(v)$ contains a given ball $B \subset Z$. 

\begin{lem}\label{large inclusion}
Let $B \subset Z$ be a ball in $Z$, let $k \in \Z$ be such that $\alpha^{-k} \geq r(B)$, and suppose that $B_{Z}(\pi(v),\alpha^{-k}) \cap B \neq \emptyset$ for some $v \in V_{k}$. Then $B \subset B(v)$. 

Consequently if $\alpha^{-k} \geq r(B)$ then there is always some vertex $v \in V_{k}$ such that $B \subset B(v)$. 
\end{lem}

\begin{proof}
Let $z \in B \cap B_{Z}(\pi(v),\alpha^{-k})$ and let $x \in B$ be any other point in $B$. Then $d(x,z) \leq \alpha^{-k}$ by assumption. Thus 
\[
d(x,\pi(v)) \leq d(x,z) + d(z,\pi(v)) < 2\alpha^{-k} < \tau \alpha^{-k},
\]
since we assume that $\tau > 3$. It follows that $x \in B(v)$. We conclude that $B \subset B(v)$. The final assertion follows by observing that the balls $B_{Z}(\pi(v),\alpha^{-k})$ for $v \in V_{k}$ cover $Z$ for each $k \in \Z$.
\end{proof}

The metric space $Z$ is \emph{doubling} if there is an integer $D \geq 1$ such that for each $z \in Z$ and $r > 0$ any $r$-separated subset of $B_{Z}(z,2r)$ has at most $D$ points. The graph $X$ has \emph{bounded degree} if there is an integer $N \geq 1$ such that each vertex is connected by an edge to at most $N$ other vertices. Note that if $X$ has bounded degree then it is proper, i.e., closed balls in $X$ are compact. These two concepts are closely linked to one another, as the following proposition shows.

\begin{prop}\label{doubling degree}
The metric space $Z$ is doubling if and only if the hyperbolic filling $X$ has bounded degree, and this equivalence is quantitative in the doubling constant, vertex degree, $\alpha$, and $\tau$. 
\end{prop}

\begin{proof}
We give a proof that closely follows the proof of an analogous proposition in the work of Buyalo-Schroeder \cite[Proposition 8.3.3]{BS07}. We first show that $X$ has bounded degree if $Z$ is doubling. It's an easy standard fact that if $Z$ is doubling with constant $D$ then there is a control function $\Lambda: [1,\infty) \rightarrow \N$, quantitative in $D$, such that for each $r > 0$ every ball of radius $sr$ contains at most $\Lambda(s)$ points that are $r$-separated (see \cite[Exercise 8.3.1]{BS07} ). This follows by simply applying the doubling condition repeatedly across multiple scales. 

Let $v \in V_{n}$ be any vertex and consider the associated ball $B(v)$. For each $m \in \Z$ we let $S_{m}(v) \subset S_{m}$ denote the set of points $z \in S_{m}$ such that the associated vertex $w \in V_{m}$ satisfies $v \sim w$. Then we can only have $S_{m}(v) \neq \emptyset$ when $|m-n| \leq 1$. If $w \in S_{m}(v)$ then $B(v) \cap B(w) \neq \emptyset$ and $m \leq n-1$.  This implies that $d(\pi(v),\pi(w)) < 2\tau \alpha^{-n+1}$. Thus $S_{n-1}(v)$, $S_{n}(v)$, and $S_{n+1}(v)$ each form an $\alpha^{-n-1}$-separated set inside of the ball $B(\pi(v),2\tau \alpha^{-n+1})$ and thus each have cardinality bounded by $\Lambda(2\tau \alpha^{2})$. We conclude that $X$ has bounded degree with degree bound $N = 3\Lambda(2\tau \alpha^{2})$. 

Let's now assume that $X$ has bounded degree, so that each vertex of $X$ is connected to at most $N$ other vertices. Consider a ball $B_{Z}(z,2r)$ in $Z$. Let $k \in \Z$ be such that $\alpha^{-k-1} < 2r \leq \alpha^{-k}$. Let $l \in \N$ be the minimal integer such that $\alpha^{-l+1} \leq \frac{1}{2}$. Then any $r$-separated subset of $B_{Z}(z,2r)$ is also $\alpha^{-k-l}$-separated. Let $v \in V_{k}$ be such that $d(z,\pi(v)) < \alpha^{-k}$. Then for $y \in B_{Z}(z,2r)$ we have
\[
d(y,\pi(v)) \leq d(y,z) + d(z,\pi(v)) < 2\alpha^{-k} < \tau \alpha^{-k},
\]
since $\tau > 3$, which implies that $B_{Z}(z,2r) \subseteq B(v)$. It thus suffices to show that there is a uniform bound on the size of an $\alpha^{-k-l}$-separated subset of $B(v)$, quantitative in $N$, $\alpha$, $l$, and $\tau$ (note $l$ is quantitative in $\alpha$). 

If $\{z_{n}\}$ is an $\alpha^{-k-l}$-separated subset of $B(v)$ then the balls $B_{Z}(z_{n},\alpha^{-k-2l})$ are all disjoint since $\alpha^{-l} < \frac{1}{2}$.  We select for each $z_{n}$ a corresponding vertex $v_{n} \in V_{k+2l}$ such that $\pi(v_{n}) \in B_{Z}(z_{n},\alpha^{-k-2l})$; these vertices are all distinct since these balls are disjoint. We then have 
\[
B_{Z}(z_{n},\alpha^{-k-2l}) \subset B_{Z}(\pi(v_{n}),2\alpha^{-k-2l}) \subset B(v_{n}),
\]
since $\tau > 3$, which implies that the vertex $v_{n}$ then satisfies $B(v_{n}) \cap B(v) \neq \emptyset$. It thus suffices to produce a uniform bound on the number of vertices $w \in V_{k+2l}$ such that $B(w) \cap  B(v) \neq \emptyset$. Given such a vertex $w$, since $B(w) \cap B(v) \neq \emptyset$ and $l \geq 1$, by Lemma \ref{height connection} we can find a vertical geodesic from $w$ to $v$ of length $2l$.  We conclude that any vertex $w \in V_{k+2l}$ with $B(w) \cap B(v)$ is joined to $v$ by a vertical geodesic of length $2l$. Since the number of  vertices joined to $v$ by a vertical geodesic of length $2l$ is at most $N^{2l}$, this produces our desired bound.
\end{proof}

We will assume for the rest of this section that $Z$ is a complete doubling metric space, from which it follows that $X$ is a proper geodesic metric space. Note that the doubling condition on $Z$ implies that bounded subsets of $Z$ are totally bounded, from which it follows that any closed and bounded subset of $Z$ is compact. We conclude in particular that $Z$ is proper. The following consequence of the doubling condition will be used frequently in subsequent sections. 

\begin{lem}\label{doubling overlap}
Let $Z$ be doubling with doubling constant $D$. Let $S$ be an $r$-separated subset of $Z$ for a given $r > 0$. Then for any $\tau \geq 1$ and $z \in Z$ we have that $z \in B_{Z}(x,\tau r)$ for at most $D^{l+1}$ points $x \in S$, where $l$ is the minimal integer such that $2^{l} \geq \tau$.  
\end{lem}

\begin{proof}
Suppose that $z \in B_{Z}(x,\tau r)$ for some $x \in S$. Then $y \in B_{Z}(x,2\tau r)$ for any other $y \in S$ such that $z \in B_{Z}(y,\tau r)$. Thus $S \cap B_{Z}(x,2\tau r)$ defines an $r$-separated subset of the ball $B_{Z}(x,2\tau r)$, which is also an $r$-separated subset of the ball $B_{Z}(x,2^{l+1}r)$. The doubling property then implies that the cardinality of $S \cap B_{Z}(x,2\tau r)$ is bounded above by $D^{l+1}$. 
\end{proof}

A vertical geodesic $\gamma$ in $X$ is \emph{anchored} at a point $z \in Z$ if for each vertex $v$ on $\gamma$ we have that $z \in B_{Z}(\pi(v),\alpha^{-h(v)})$. Note this implies that $z \in B(v)$. If we do not need to mention the point $z$ then we will just say that $\gamma$ is \emph{anchored}. The following simple lemma shows that for any $z \in Z$ we can find a geodesic line anchored at $z$. 

\begin{lem}\label{second height connection} 
For any $z \in Z$ there is a vertical geodesic line $\gamma: \R \rightarrow X$ anchored at $z$ . If $v \in V$ satisfies $z \in B_{Z}(\pi(v),\alpha^{-h(v)})$ then we can choose $\gamma$ such that $v \in \gamma$.  
\end{lem}

\begin{proof}
For each $n \in \Z$ the balls $B_{Z}(\pi(v),\alpha^{-n})$ for $v \in V_{n}$ cover $Z$. Thus for each $n$ we can choose a vertex $v_{n} \in \Z$ such that $z \in B_{Z}(\pi(v_{n}),\alpha^{-n})$. Then $z \in B(v_{n})$ as well. It follows that $B(v_{n}) \cap B(v_{n+1}) \neq \emptyset$ for each $n \in \Z$, i.e., $v_{n}\sim v_{n+1}$. We can then find a vertical geodesic $\gamma: \R \rightarrow X$ such that $v_{n} \in \gamma$ for each $n$. If $v \in V$ is a given vertex such that $z \in B_{Z}(\pi(v),\alpha^{-h(v)})$ then we can choose $v_{h(v)} = v$ in this construction to guarantee that $v \in \gamma$. 
\end{proof}

\begin{rem}\label{anchored remark} The definition of anchored geodesics in \cite{Bu20} is less restrictive, requiring only that $z \in B_{Z}(\pi(v),\frac{\tau}{3}\alpha^{-h(v)})$ for each $v \in \gamma$ instead. This was because we were not assuming that $Z$ was complete in that paper. We will use the more restrictive inclusion $z \in B_{Z}(\pi(v),\alpha^{-h(v)})$ here instead, with the understanding that all of the claims proved regarding anchored geodesics in \cite{Bu20} hold for these geodesics in particular.  
\end{rem}

By \cite[Lemma 5.11]{Bu20} all descending anchored geodesic rays in $X$ are at a bounded distance from one another and therefore define a common point $\omega \in \p X$. By \cite[Proposition 5.13]{Bu20} we have a canonical identification $\p_{\omega}X \cong Z$ that can be realized by identifying a point $z \in Z$ with the collection of ascending geodesic rays anchored at $z$. Under this identification the metric $d$ on $Z$ defines a visual metric on $\p_{\omega}X$ with parameter $\e = \log \alpha$.  

By \cite[Lemma 5.12]{Bu20} there is a Busemann function $b$ based at $\omega$ such that the height function $h$ satisfies $h \doteq_{3} b$. Thus the height function $h$ can be thought of as a Busemann function on $X$ based at $\omega$, up to an additive error of $3$. We define the Gromov product based at $h$ by
\[
(x|y)_{h} = \frac{1}{2}(h(x)+h(y)-2|xy|),
\]
for $x,y \in X$. We then observe that $(x|y)_{h} \doteq_{3} (x|y)_{b}$ as well. Since $h$ is $1$-Lipschitz we also have
\[
(x|y)_{h} \leq \min\{h(x),h(y)\},
\] 
for $x,y \in X$. The Gromov product based at $h$ is then extended to $\p_{\omega} X$ by the same formulas \eqref{full extended definition} and \eqref{extended definition} as were used for $b$, and the same estimates \eqref{sequence approximation} and \eqref{both busemann boundary} hold for this extension with $h$ replacing $b$, at the cost of increasing the constant $c(\delta)$ by $6$. 

\subsection{Uniformizing the hyperbolic filling} For this section we will remain in the same setting as the previous section. Throughout this section all implied constants are considered to only depend on $\alpha$ and $\tau$. We recall that the hyperbolic filling $X$ is $\delta$-hyperbolic with $\delta = \delta(\alpha,\tau)$. We also recall that $\omega \in \p X$ denotes the equivalence class of all anchored descending geodesic rays in $X$. 

We first recall the definition of a uniform metric space. We consider an incomplete metric space $(\Omega,d)$ and write $\p \Omega = \bar{\Omega}\backslash \Omega$. We write $d_{\Omega}(x) :=\dist(x,\p \Omega)$ for the distance of a point $x \in \Omega$ to the boundary $\p \Omega$. A basic observation that we will use without comment in what follows is that $d_{\Omega}$ defines a $1$-Lipschitz function on $\Omega$, i.e., for $x,y \in \Omega$ we have 
\[
|d_{\Omega}(x)-d_{\Omega}(y)| \leq d(x,y). 
\]
For a curve $\gamma: I \rightarrow \Omega$ we write $\ell(\gamma)$ for the length of $\gamma$. For an interval $I \subset \R$ we write $I_{\leq t} = \{s \in I: s\leq t\}$ and $I_{\geq t} = \{s \in I: s\geq t\}$. 

\begin{defn}\label{def:uniform}A curve $\gamma: I \rightarrow \Omega$ joining two points $x,y \in \Omega$ is \emph{$A$-uniform} for a constant $A \geq 1$ if 
\[
\ell(\gamma) \leq Ad(x,y),
\]
and if for every $t \in I$ we have
\[
\min\{\ell(\gamma|_{I_{\leq t}}),\ell(\gamma|_{I_{\geq t}})\} \leq A d_{\Omega}(\gamma(t)). 
\]
The metric space $\Omega$ is \emph{$A$-uniform} if any two points in $\Omega$ can be joined by an $A$-uniform curve. 
\end{defn}

We define a density $\rho: X \rightarrow (0,\infty)$ by $\rho(x) = \alpha^{-h(x)}$. For a curve $\gamma$ in $X$ we write
\[
\ell_{\rho}(\gamma) = \int_{\gamma} \rho \, ds,
\]
for the line integral of $\rho$ along $\gamma$. We define a new metric $d_{\rho}$ on $X$ by setting for $x,y \in X$,
\begin{equation}\label{rho formula}
d_{\rho}(x,y) = \inf  \ell_{\rho}(\gamma),
\end{equation}
with the infimum taken over all curves $\gamma$ joining $x$ to $y$. We write $X_{\rho} = (X,d_{\rho})$ for the resulting metric space, which we refer to as the \emph{conformal deformation of $X$ with conformal factor $\rho$}. We write $\bar{X}_{\rho}$ for the completion of $X_{\rho}$ and write $\p X_{\rho} = \bar{X}_{\rho} \backslash X_{\rho}$ for the complement of $X_{\rho}$ inside its completion. We will continue to write $d_{\rho}$ for the canonical extension of this metric to the completion $\bar{X}_{\rho}$. For $x \in X_{\rho}$ we write $d_{\rho}(x) = d_{X_{\rho}}(x)$ for the distance to the boundary $\p X_{\rho}$. For $x \in \bar{X}_{\rho}$ and $r > 0$ we will write $B_{\rho}(x,r)$ for the ball of radius $r$ centered at $x$ in $\bar{X}_{\rho}$ in the metric $d_{\rho}$. 

By \cite[Theorem 1.12]{Bu20} the metric space $X_{\rho}$ is $A$-uniform with $A = A(\alpha,\tau)$ depending only on $\alpha$ and $\tau$. The theorem in fact shows for any $x,y \in X$ that any geodesic $\gamma$ joining $x$ to $y$ in $X$ is an $A$-uniform curve in $X_{\rho}$. The theorem also shows that we have a canonical identification of $\p X_{\rho}$ with $\p_{\omega}X$, which is given by showing that a sequence $\{x_{n}\}$ in $X$ is a Cauchy sequence in $X_{\rho}$ if and only if it converges to infinity with respect to $\omega$ in $X$. Composing the identification $\p X_{\rho} \cong \p_{\omega}X$ with the identification $\p_{\omega}X \cong Z$ then gives a canonical $L$-biLipschitz identification of $\p X_{\rho}$ with $Z$, with $L = L(\alpha,\tau)$. Thus after a biLipschitz change of metric on $Z$ (quantitative in $\alpha$ and $\tau$) we can assume that $Z$ is isometrically identified with $\p X_{\rho}$. We will make this biLipschitz change of metric and thus consider $Z$ as an isometrically embedded subset of $\bar{X}_{\rho}$ with $Z = \p X_{\rho}$. We will then use the notation $Z$ and $\p X_{\rho}$ interchangeably for $Z$, with the choice of notation depending on the context. 

The local compactness of $X_{\rho}$ implies by the Arzela-Ascoli theorem that, for a given $x,y \in X$, a minimizing curve $\gamma$ for the right side of \eqref{rho formula} always exists. It is easy to see that such a curve must be a geodesic in $X_{\rho}$, from which we conclude that $X_{\rho}$ is geodesic. Since $X_{\rho}$ is a uniform metric space, by \cite[Proposition 2.20]{BHK} the completion $\bar{X}_{\rho}$ of $X_{\rho}$ is proper, and in particular is also locally compact. A second application of Arzela-Ascoli then shows that $\bar{X}_{\rho}$ is also geodesic.

We have the following two key lemmas, in which the implied constants depend only on $\alpha$ and $\tau$. Below we define $|xy| = \infty$ if $x \neq y$ and either $x \in \p_{\omega} X$ or $y \in \p_{\omega} X$, and set $|xy|=0$ if $x = y \in \p_{\omega} X$. We recall that we have canonically identified $\p_{\omega} X$ with $\p X_{\rho}$. 

\begin{lem}\label{filling estimate both}\cite[Lemma 4.13]{Bu20}
Let $x,y \in X \cup \p_{\omega} X$. Then we have
\[
d_{\rho}(x,y) \asymp \alpha^{-(x|y)_{h}}\min\{1,|xy|\}.
\]
\end{lem}

\begin{lem}\label{filling compute distance}\cite[Lemma 4.15]{Bu20}
For $x \in X$ we have
\[
d_{\rho}(x) \asymp  \alpha^{-h(x)}.
\]
\end{lem}

\begin{rem}\label{rem:transition}
Lemmas \ref{filling estimate both} and \ref{filling compute distance} are stated in a slightly different context in \cite{Bu20}. Letting $b$ denote a Busemann function on $X$ such that $b \doteq_{3} h$ and letting $\e = \log \alpha$, the results in that work show that these lemmas hold (with constants depending only on $\alpha$ and $\tau$) if we replace $\rho$ by 
\[
\rho_{\e}(x) = e^{-\e b(x)} = \alpha^{-b(x)}. 
\]
The rough equality $b \doteq_{3} h$ implies that the conformal deformation $X_{\rho_{\e}}$ of $X$ with conformal factor $\rho_{\e}$ is $\alpha^{3}$-biLipschitz to $X_{\rho}$ by the identity map on $X$. From this observation we can then deduce that Lemmas \ref{filling estimate both} and \ref{filling compute distance} hold as stated above. 
\end{rem}

The following observation will be useful in conjunction with Lemma \ref{filling compute distance}.

\begin{lem}\label{vertex to base}
For $v \in V$ we have
\[
d_{\rho}(v,\pi(v)) \asymp d_{\rho}(v) \asymp \alpha^{-h(v)}.
\]
\end{lem}

\begin{proof}
We trivially have $d_{\rho}(v,\pi(v)) \geq d_{\rho}(v)$. On the other hand, by Lemma \ref{second height connection} we can find an ascending vertical geodesic ray $\gamma: [0,\infty) \rightarrow X$ starting at $v$ and anchored at $\pi(v)$. A straightforward computation shows that
\[
\ell_{\rho}(\gamma) = \int_{0}^{\infty} \alpha^{-t-h(v)}\, dt \ls \alpha^{-h(v)},
\]
from which it follows that $d_{\rho}(v,\pi(v)) \ls \alpha^{-h(v)}$. The conclusion of the lemma then follows from Lemma \ref{filling compute distance}.
\end{proof}

%The following claim will also be useful. 

%\begin{lem}\label{product vertex endpoint}
%For any vertex $v \in V$ we have
%\[
%(v|\pi(v))_{h} \doteq h(v). 
%\]
%\end{lem}

%\begin{proof}
%By Lemma \ref{second height connection} we can find an ascending geodesic ray $\gamma:[0,\infty) \rightarrow X$ starting from $v$ that is anchored at $\pi(v)$. Then a straightforward calculation shows that we have $(\gamma(t)|v)_{h} = h(v)$ for all $t \geq 0$. Since the sequence $\{\gamma(n)\}_{n = 0}^{\infty}$ converges to $\pi(v) \in \p_{\omega} X \cong Z$, the conclusion then follows from inequality \eqref{sequence approximation}. 
%\end{proof}

We now introduce a concept inspired by work of Lindquist \cite[Definition 4.10]{L17}. Our definition is somewhat different from the one given there. For this definition we recall that balls are always considered to have a fixed center and radius, even if a particular subset can be described as a ball in multiple different ways. 

\begin{defn}\label{defn:hull}
Let $B = B_{Z}(z,r)$ be any ball in $Z$. The \emph{hull} $H^{B} \subset X_{\rho}$ of $B$ in $X_{\rho}$ is the union $H^{B} = \bigcup \bar{B}_{X}(v,\frac{1}{2})$ over all vertices $v \in V$ such that $\alpha^{-h(v)} \leq r$ and $B(v) \cap B \neq \emptyset$. We consider $H^{B}$ as being equipped with the uniformized metric $d_{\rho}$.  For $n \in \Z$ we write $H_{n}^{B}= H^{B}\cap V_{n}$ for the set of vertices in $H^{B}$ at height $n$. 
\end{defn}

Since balls are considered to come assigned with a center and radius, it may be the case for two balls $B$ and $B'$ in $Z$ with different centers and radii that $B = B'$ as sets but $H^{B} \neq H^{B'}$. By construction each vertex $v \in H^{B}$ has the property that $\alpha^{-h(v)} \leq r$ and $B(v) \cap B \neq \emptyset$, and for each $x \in H^{B}$ there is a vertex $v \in H^{B}$ such that $|xv| \leq \frac{1}{2}$. An edge $e$ in $X$ satisfies $e \subset H^{B}$ if and only if the endpoints of $e$ both belong to $H^{B}$. For each $n \in \Z$ such that $\alpha^{-n} \leq r$ we have by construction that the balls $B(v)$ for $v \in H_{n}^{B}$ cover $B$. In this paper we will primarily be using hulls of balls in $Z$ as a convenient approximation to balls in $\bar{X}_{\rho}$ centered at points of $Z$. We first show that the metric boundary of $H^{B}$ coincides with the closure $\bar{B}$ of $B$ in $Z$. 

\begin{lem}\label{hull closure}
Let $B \subset Z$ be any ball and let $H^{B} \subset X_{\rho}$ be its hull. Then 
\[
\p H^{B} = \overline{H^{B}}\backslash H^{B} = \bar{B} \subset Z. 
\]
\end{lem}

\begin{proof}
Let $r = r(B)$. Let $\{x_{n}\} \subset H^{B}$ be any sequence for which there is some $y \in \bar{X}_{\rho} \backslash H^{B}$ such that $d_{\rho}(x_{n},y) \rightarrow 0$. If $y \in X_{\rho}$ then, since the metric $d_{\rho}$ is locally biLipschitz to the hyperbolic metric on $X$, we must have $|x_{n}y| \rightarrow 0$ as well. Since the closed balls $\bar{B}_{X}(v,\frac{1}{2})$ for $v \in V$ cover $X$, we can find a vertex $v$ such that $y \in \bar{B}_{X}(v,\frac{1}{2})$. If $y \in B_{X}(v,\frac{1}{2})$ then we must have $x_{n} \in B_{X}(v,\frac{1}{2})$ for $n$ sufficiently large. Since $x_{n} \in H^{B}$ this implies that $v \in H^{B}$, from which it follows that $y \in H^{B}$, contradicting our assumptions. If $|yv| = \frac{1}{2}$ then we let $w$ be the other vertex on the edge containing $y$. Then for $n$ sufficiently large we must have either $x_{n} \in \bar{B}_{X}(v,\frac{1}{2})$ or $x_{n} \in \bar{B}_{X}(w,\frac{1}{2})$. Thus we must have either $v \in H^{B}$ or $w \in H^{B}$, both of which imply that $y \in H^{B}$. This again contradicts our assumptions. 

Thus we must have $y \in Z = \p X_{\rho}$. Then $\{x_{n}\}$ converges to a point of $\p X_{\rho}$ and therefore $h(x_{n}) \rightarrow \infty$ as $n \rightarrow \infty$.  By the definition of $H^{B}$ we can find a sequence of vertices $\{v_{n}\} \subset H^{B}$ such that $|x_{n}v_{n}| \leq \frac{1}{2}$. By Lemma \ref{filling estimate both} it follows that $d_{\rho}(x_{n},v_{n}) \rightarrow 0$, so we also have $d_{\rho}(v_{n},y) \rightarrow 0$. Set $z_{n} = \pi(v_{n})$. Then $d_{\rho}(z_{n},v_{n}) \ls \alpha^{-h(v_{n})}$ by Lemma \ref{vertex to base} and therefore $d_{\rho}(z_{n},v_{n}) \rightarrow 0$. Thus $d_{\rho}(z_{n},y) \rightarrow 0$. But since $B \cap B(v_{n}) \neq \emptyset$, we can find points $y_{n} \in B \cap B(v_{n})$ for each $n$ that satisfy $d(y_{n},z_{n}) < \tau \alpha^{-h(v_{n})}$. This implies that $d(y_{n},y) \rightarrow 0$ as well, which implies that $y \in \bar{B}$ since $y_{n} \in B$ for each $n$. 

We conclude that $\p H^{B} \subset \bar{B}$. To obtain equality, let $y \in \bar{B}$ be any point and let $\gamma$ be an ascending vertical geodesic ray anchored at $y$ as constructed in Lemma \ref{second height connection} starting from a vertex $v_{0} \in V_{0}$. Let $\{v_{n}\}_{n \geq 0}$ be the sequence of vertices on $\gamma$ with $h(v_{n}) = n$. For each $n$ we then have $y \in B(v_{n})$ by construction and therefore $B(v_{n}) \cap B \neq \emptyset$ by the definition of the closure $\bar{B}$. For $n$ sufficiently large we will have that $\alpha^{-n} \leq r$ and therefore $v_{n} \in H^{B}$. Then $d_{\rho}(v_{n},y) \rightarrow 0$ since $\gamma$ has $y$ as its endpoint in $\p X_{\rho}$. It follows that $y \in \overline{H^{B}}$. Since $H^{B} \subset X_{\rho}$, we in fact have $y \in \p H^{B}$. Thus $\p H^{B} = \bar{B}$.  
\end{proof}

We remind the reader that in general $\bar{B} = \overline{B_{Z}(z,r)}$ may be a proper subset of the closed ball $\bar{B}_{Z}(z,r) = \{x \in Z:d(x,z) \leq r\}$. 

We next show that the closure of the hull of a ball $B \subset Z$ can be approximated by balls in $\bar{X}_{\rho}$ centered at the center of $B$. We then use this to show that balls in $\bar{X}_{\rho}$ centered at points of $Z$ can be approximated by the closures of hulls of balls in $Z$. For a ball $B = B_{Z}(z,r)$ in $Z$ we write $\hat{B} = B_{\rho}(z,r)$ for the corresponding ball centered at $z$ in $\bar{X}_{\rho}$.

\begin{lem}\label{hull approximation}
There is a constant $C = C(\alpha,\tau) \geq 1$ such that if $B$ is any ball in $Z$ then $C^{-1}\hat{B} \subset \overline{H^{B}} \subset C \hat{B}$. Consequently we have
\begin{equation}\label{ball to hull}
\overline{H^{C^{-1}B}} \subset \hat{B} \subset \overline{H^{CB}}.
\end{equation}
\end{lem}

\begin{proof}
Let $B = B_{Z}(z,r)$ be a given ball. Let $0 < \la_{0} < 1$ be a given parameter to be tuned in the proof and suppose that $v \in V$ is a vertex satisfying $d_{\rho}(v,z) <  \la_{0} r$. Then by Lemma \ref{filling estimate both} we have that
\[
\alpha^{-(v|z)_{h}} \asymp d_{\rho}(v,z) <  \la_{0} r.
\]
By inequality \eqref{both busemann boundary} it follows that
\begin{equation}\label{kappa inequality}
\alpha^{-h(v)} \lesssim  \la_{0} r,
\end{equation}
as well. We have by inequality \eqref{kappa inequality} and Lemma \ref{vertex to base},
\[
d(\pi(v),z) \leq d_{\rho}(\pi(v),v) + d_{\rho}(v,z) \ls \alpha^{-h(v)} + \la_{0} r \ls  \la_{0} r. 
\]
Thus $d(\pi(v),z) < C_{0} \la_{0} r$ and $\alpha^{-h(v)} \leq C_{0} \la_{0}r$ for a constant $C_{0} = C_{0}(\alpha,\tau) \geq 1$ depending only on $\alpha$ and $\tau$. We set $\la_{0} = C^{-1}_{0}$.  It then follows that we have $d(\pi(v),z) < r$ and $\alpha^{-h(v)} \leq r$.  This implies that $\pi(v) \in B$, which means that $B(v) \cap B \neq \emptyset$. Since $\alpha^{-h(v)} \leq r$ it then follows that $v \in H^{B}$. 

We conclude that all vertices $v \in  \la_{0}\hat{B}$ satisfy $v \in H^{B}$. Let $0 < \la <  \la_{0}$ be another given parameter. Let $x \in \la \hat{B} \cap X_{\rho}$ be any given point and let $v$ be a vertex satisfying $|xv| \leq \frac{1}{2}$. Then by Lemma \ref{filling estimate both} and inequality \eqref{kappa inequality} (for $\la$ instead of $\la_{0}$) we have 
\[
d_{\rho}(v,x) \ls \alpha^{-(v|x)_{h}} \ls \alpha^{-h(v)} \ls \la r,
\]
from which it follows that 
\[
d_{\rho}(v,z) \leq d_{\rho}(v,x) + d_{\rho}(x,z) \ls \la r. 
\]
Thus there is a constant $C = C(\alpha,\tau)$ such that $d_{\rho}(v,z) < C\la r$ for any vertex $v$ such that there is a point $x \in \la \hat{B}$ with $|xv| \leq \frac{1}{2}$. We set $\la = C^{-1}\la_{0}$. Then it follows that $v \in \la_{0} \hat{B}$ and therefore $v \in H^{B}$. By the definition of the hull we then conclude that $x \in H^{B}$. Thus $\la \hat{B} \cap X_{\rho} \subset H^{B}$. Finally, since $\la < 1$ and $\p X_{\rho} = Z$, if $x \in \la \hat{B} \cap \p X_{\rho} = \la B$ then $x \in B$ and therefore $x \in \overline{H^{B}}$ by Lemma \ref{hull closure}. This proves the inclusion $\la\hat{B} \subset \overline{H^{B}}$. 

Now let $v \in H^{B}$ be any vertex. Let $y \in B(v) \cap B$ be a point in this intersection. Then
\[
d(\pi(v),z) \leq d(\pi(v),y) + d(y,z) < \tau \alpha^{-h(v)} + r \lesssim r, 
\]
since $\alpha^{-h(v)} \leq r$. Then by Lemma \ref{vertex to base}, 
\[
d_{\rho}(v,z) \leq d_{\rho}(v,\pi(v)) + d(\pi(v),z) \ls \alpha^{-h(v)} + r \ls r. 
\]
Thus there is a constant $C = C(\alpha,\tau) \geq 1$ such that $d_{\rho}(v,z) \leq Cr$. It follows that $v \in C\hat{B}$.

If $x \in H^{B}$ is an arbitrary point then we can find a vertex $v \in H^{B}$ such that $|xv| \leq \frac{1}{2}$. Then $v \in C\hat{B}$ by our prior calculations. By Lemma \ref{filling estimate both} we have
\[
d_{\rho}(x,v) \ls \alpha^{-(x|v)_{h}} \ls \alpha^{-h(v)} \leq r. 
\]
Thus
\[
d_{\rho}(x,z) \leq d_{\rho}(x,v) + d_{\rho}(v,z) \ls r. 
\]
It follows that $x \in C\hat{B}$ as well, for a possibly larger constant $C = C(\alpha,\tau)$. Finally since $\p H^{B} = \bar{B} \subset 2 \hat{B}$ by Lemma \ref{hull closure}, we conclude that there is a constant $C = C(\alpha,\tau)$ such that $\overline{H^{B}} \subset C\hat{B}$. This completes the proof of the first assertion. 

It remains to obtain the chain of inclusions \eqref{ball to hull}. Let $B$ be a given ball in $Z$ and let $C = C(\alpha,\tau) \geq 1$ be the constant obtained above. Applying the conclusion of the first part to the ball $CB$ gives that $\hat{B} \subset \overline{H^{CB}}$ and applying the conclusion of the first part to the ball $C^{-1}B$ gives $\overline{H^{C^{-1}B}} \subset \hat{B}$. This establishes the inclusions \eqref{ball to hull}.
\end{proof}

Lastly we make the useful observation that the hulls $H^{B(v)}$ for $v \in V_{n}$ have bounded overlap in $X$. 

\begin{lem}\label{hull overlap}
There is a constant $M = M(\tau,D)$ depending only on $\tau$ and the doubling constant $D$ for $Z$ such that for any fixed $n \in \Z$ a point $x \in X$ can belong to at most $M$ of the sets $H^{B(v)}$ for $v \in V_{n}$. 
\end{lem}

\begin{proof}
Let $n \in \Z$ and $x \in X$ be given. For the purposes of the lemma we can suppose that $x$ belongs to at least one hull $H^{B(v)}$ for some $v \in V_{n}$. We fix this vertex $v$ in what follows. The point $x$ can only belong to $H^{B(v)}$ if at least one of the vertices on the edge containing $x$ also belongs to $H^{B(v)}$. We let $x_{v}$ be a vertex on the edge containing $x$ such that $x_{v} \in H^{B(v)}$. Then we must have $B(x_{v}) \cap B(v) \neq \emptyset$ and $\alpha^{-h(x_{v})} \leq \tau \alpha^{-n}$ by the definition of $H^{B(v)}$. It follows that, for a point $z \in B(x_{v}) \cap B(v)$ in this intersection, 
\[
d(\pi(v),\pi(x_{v})) \leq d(\pi(v),z) + d(z,\pi(x_{v})) \leq 2\tau \alpha^{-n}. 
\]
Let $w \in V_{n}$ be any other vertex such that $x \in H^{B(w)}$. We similarly let $x_{w}$ be a vertex on the edge containing $x$ such that $x_{w} \in H^{B(w)}$. If $x_{v} = x_{w}$ then 
\[
d(\pi(v),\pi(w)) \leq d(\pi(v),\pi(x_{v})) + d(\pi(x_{v}),\pi(w)) \leq 4\tau \alpha^{-n}. 
\]
If $x_{v} \neq x_{w}$ then we note that $B(x_{v}) \cap B(x_{w}) \neq \emptyset$ and $\max\{\alpha^{-h(x_{v})},\alpha^{-h(x_{w})}\} \leq \tau \alpha^{-n}$. Thus 
\[
d(\pi(x_{v}),\pi(x_{w})) \leq 2\tau^{2}\alpha^{-n}.
\]
It follows by the triangle inequality that
\[
d(\pi(v),\pi(w)) \leq 6\tau^{2}\alpha^{-n}. 
\]
Thus the vertices $w \in V_{n}$ such that $x \in H^{B(w)}$ form an $\alpha^{-n}$-separated subset of the ball $B_{Z}(\pi(v),6\tau^{2} \alpha^{-n})$ in $Z$. Hence the number of such vertices is bounded above by $D^{2l+3}$, where $l$ is the minimal integer such that $2^{l} \geq \tau$.
\end{proof}

\section{Lifting doubling measures}\label{sec:lift}

In this section we start with a complete doubling metric measure space $(Z,d,\nu)$. We will write $C_{\nu}$ for the doubling constant of $\nu$ in the doubling inequality \eqref{doubling inequality}. It's easy to see for a doubling metric measure space $(Z,d,\nu)$ that the underlying metric space $(Z,d)$ is doubling with constant $D = D(C_{\nu})$, see for instance \cite[Chapter 4.1]{HKST}. Thus by Proposition \ref{doubling degree} any hyperbolic filling $X$ of $Z$ with parameters $\alpha,\tau > 1$ satisfying \eqref{tau requirement} will have vertex degree bounded above by $N = N(\alpha,\tau,C_{\nu})$.

We fix parameters $\alpha,\tau > 1$ satisfying \eqref{tau requirement} and let $X$ be a hyperbolic filling of $Z$ with these parameters as constructed in the previous section. We carry over all concepts and notation from the previous section. In particular we let $X_{\rho}$ denote the conformal deformation of $X$ with conformal factor $\rho(x) = \alpha^{-h(x)}$ and isometrically identify $Z$ with $\p X_{\rho}$ via a biLipschitz change of metric on $Z$. Throughout the first part of this section (until Lemma \ref{hull measure}) all implied constants will depend only on $\alpha$, $\tau$, and the doubling constant $C_{\nu}$ for $\nu$. 

Our first task in this section will be to lift the doubling measure $\nu$ to a uniformly locally doubling measure $\mu$ on $X$ that supports a uniformly local 1-Poincar\'e inequality. The terminology here is explained in the statement of Proposition \ref{local hyp}. To this end we adapt the construction in \cite[Section 10]{BBS21}. As before we write $V  = \bigcup_{n \in \Z} V_{n}$ for the vertices of $X$ and write $E$ for the set of edges of $X$. We let $\mathcal{L}$ denote the Borel measure on $X$ given by Lebesgue measure on each edge of $X$, recalling that each edge of $X$ has unit length. The measure $\mathcal{L}$ can equivalently be thought of as the $1$-dimensional Hausdorff measure on $X$. 

We define a measure $\hat{\mu}$ on $V$ by setting for each $v \in V$, 
\begin{equation}\label{hat measure definition}
\hat{\mu}(\{v\}) = \nu(B(v)),
\end{equation}
where we recall that if $v = (z,n)$ then $B(v) = B_{Z}(z,\tau \alpha^{-n})$. To simplify notation we will write $\hat{\mu}(v) :=\hat{\mu}(\{v\})$. Our first lemma shows that adjacent vertices have comparable $\hat{\mu}$-measure. 

\begin{lem}\label{doubling vertex comparison}
Let $v,w \in V$ satisfy $v \sim w$. Then
\[
\hat{\mu}(v) \asymp \hat{\mu}(w).
\]
\end{lem}

\begin{proof}
By symmetry it suffices to verify the upper bound $\hat{\mu}(v) \ls \hat{\mu}(w)$. Since $v \sim w$ we must have $|h(v)-h(w)| \leq 1$, which implies that $h(v) \geq h(w)-1$. Since $B(v) \cap B(w) \neq \emptyset$, we must then have $d(\pi(v),\pi(w)) < 2\tau \alpha^{-h(w)+1}$.  Thus if $z \in B(v)$ then 
\[
d(z,\pi(w)) \leq d(\pi(v),\pi(w)) + d(z,\pi(v)) < 3\tau \alpha^{-h(w)+1}. 
\]
Thus $B(v) \subset B_{Z}(\pi(w),3\tau \alpha^{-h(w)+1})$. Writing $3\tau \alpha^{-h(w)+1} = 3\alpha(\tau \alpha^{-h(w)})$, the doubling condition on $\nu$ implies that 
\[
\nu(B(w)) \asymp \nu(B_{Z}(\pi(w),3\tau \alpha^{-h(w)+1})) \geq \nu(B(v)).
\]
This implies that $\hat{\mu}(v) \ls \hat{\mu}(w)$. 
\end{proof}

We next smear out $\hat{\mu}$ to a measure $\mu$ on $X$ by setting, for a Borel set $A \subset X$, 
\begin{equation}\label{lift definition}
\mu(A) = \sum_{v \in V} \sum_{w \sim v}(\hat{\mu}(v) + \hat{\mu}(w))\mathcal{L}(A \cap vw).
\end{equation}
Here $vw$ denotes the edge connecting $v$ to $w$. By Lemma \ref{doubling vertex comparison} and the fact that $X$ has vertex degree bounded by $N = N(\alpha,\tau,C_{\nu})$, we obtain the useful comparison for any vertex $v \in V$ and any $w \sim v$,
\begin{equation}\label{measure to edge}
\nu(B(v)) = \hat{\mu}(v) \asymp \mu(vw).
\end{equation}

We can now apply \cite[Theorem 10.2]{BBS21} to directly obtain uniformly local doubling of $\mu$ and a uniformly local $1$-Poincar\'e inequality on $X$. 

\begin{prop}\label{local hyp}
For each $R_{0} > 0$ there is a constant $C_{0} = C_{0}(\alpha,\tau,C_{\nu},R_{0}) \geq 1$ such that for all balls $B = B_{X}(x,r)$ in $X$ with $0 < r \leq R_{0}$ and every integrable function $u$ on $B$ with upper gradient $g$ on $B$ we have
\begin{equation}\label{local hyp doubling}
\mu(2B) \leq C_{0}\mu(B),
\end{equation}
and 
\begin{equation}\label{local hyp Poincare}
\dashint_{B} |u-u_{B}|\, d\mu \leq C_{0}r\dashint_{B} g\, d\mu. 
\end{equation}
\end{prop}

We let $\beta > 0$ be given and define a measure $\mu_{\beta}$ on $X_{\rho}$ by, for $x \in X$,  
\begin{equation}\label{beta def}
d\mu_{\beta}(x) = \alpha^{-\beta h(x)}d\mu(x). 
\end{equation}
We extend $\mu_{\beta}$ to a measure on $\bar{X}_{\rho}$ by setting $\mu_{\beta}(\p X_{\rho}) = 0$. We will be applying the results of \cite{Bu22} to establish that the metric measure spaces $(X_{\rho},d_{\rho},\mu_{\beta})$ and $(\bar{X}_{\rho},d_{\rho},\mu_{\beta})$ are each doubling and support a $1$-Poincar\'e inequality. As in Remark \ref{rem:transition}, the setting we considered in \cite{Bu22} is slightly different than the setting we are considering here, as we used a Busemann function $b$ on $X$ in place of the height function $h$. By using \cite[Lemma 5.12]{Bu20} to choose a Busemann function $b$ on $X$ such that $b \doteq_{3} h$, we see that we can apply the results of \cite{Bu22} in our setting here, at the cost of an $\alpha^{3}$-biLipschitz change of metric on $X_{\rho}$ and changing the $\mu_{\beta}$-measure of measurable sets by at most a factor of $\alpha^{3\beta}$. Both of these changes only impact the doubling property and $1$-Poincar\'e inequality for these metric measure spaces up to a constant factor depending only on $\alpha$ and $\beta$. Hence we can freely apply the results of \cite{Bu22} to the setting considered here. We also remark that we have made a slight change of notation from \cite{Bu22}: our definition of $\mu_{\beta}$ in this work corresponds to  $\mu_{\beta \e}$ in \cite{Bu22} with $\e = \log \alpha$. 

\begin{rem}\label{rem:metric measure}
Strictly speaking we do not yet know that $(X_{\rho},d_{\rho},\mu_{\beta})$ and $(\bar{X}_{\rho},d_{\rho},\mu_{\beta})$ are metric measure spaces in the sense defined in the introduction, since we have not yet shown that balls in these spaces have finite $\mu_{\beta}$-measure. The finiteness of the measure of balls is a consequence of Lemma \ref{hull measure} below. Proposition \ref{ball infinite measure} shows that this finiteness property does not hold if we replace $\mu_{\beta}$ with $\mu$; in particular the comparison constant in \eqref{ball compute} must go to infinity as $\beta \rightarrow 0$. 
\end{rem}

We start by estimating the measure of the hull of a ball in $Z$ and use this to estimate the measure of balls centered at the boundary. We recall for a ball $B = B_{Z}(z,r)$ in $Z$ that we write $\hat{B} = B_{\rho}(z,r)$ for the corresponding ball in $\bar{X}_{\rho}$. Throughout the rest of this section all implied constants will depend only on $\alpha$, $\tau$, $C_{\nu}$, and $\beta$. By the estimate  \eqref{measure to edge} and the fact that $h$ is $1$-Lipschitz, we obtain for any vertex $v \in V$ and any edge $e$ with $v \in e$, 
\begin{equation}\label{transformed}
\mu_{\beta}(e) \asymp \alpha^{-\beta h(v)}\nu(B(v)).
\end{equation}
A \emph{half-edge} $e_{*}$ in $X$ is a geodesic segment in $X$ of length $\frac{1}{2}$ starting from a vertex $v \in V$. By applying the definition \eqref{lift definition} and using Lemma \ref{doubling vertex comparison} we similarly obtain for any half-edge $e_{*} \subset X$,
\begin{equation}\label{half-edge transformed}
\mu_{\beta}(e_{*}) \asymp \alpha^{-\beta h(v)}\nu(B(v)).
\end{equation}
Since $X$ has vertex degree bounded by $N = N(\alpha,\tau,C_{\nu})$, we obtain the useful estimate for a vertex $v \in V$, 
\begin{equation}\label{vertex transformed}
\mu_{\beta}\left(\bar{B}_{X}\left(v,\frac{1}{2}\right)\right) = \sum_{v \in e_{*}} \mu_{\beta}(e_{*}) \asymp \alpha^{-\beta h(v)}\nu(B(v)),
\end{equation}
where the sum is taken over all half-edges $e_{*}$ starting from $v$.

\begin{lem}\label{hull measure}
Let $B $ be any ball in $Z$. Then we have
\begin{equation}\label{hull compute}
\mu_{\beta}(H^{B}) \asymp r^{\beta}\nu(B),
\end{equation}
and
\begin{equation}\label{ball compute}
\mu_{\beta}(\hat{B}) \asymp r^{\beta}\nu(B).
\end{equation}
\end{lem}

\begin{proof}
We first obtain the comparison \eqref{hull compute}. Let $B = B_{Z}(z,r)$ be the given ball.  Let $m$ be the minimal integer such that $\alpha^{-m} \leq r$, so that we have $\alpha^{-m} \asymp r$. We then sum the estimate \eqref{vertex transformed} over all vertices $v \in H^{B}$, noting that each point $x \in H^{B}$ satisfies $|xv| \leq \frac{1}{2}$ for at least one vertex $v \in H^{B}$ and at most two vertices. We then obtain that 
\begin{equation}\label{sum hull}
\mu_{\beta}(H^{B}) \asymp \sum_{n=m}^{\infty}\alpha^{-\beta n}\left(\sum_{v \in H_{n}^{B}} \nu(B(v))\right).
\end{equation}
To estimate the inner sum on the right, note by Lemma \ref{doubling overlap} that there is a constant $M = M(\alpha,\tau,C_{\nu})$ such that for any $x \in Z$ we have that $x$ belongs to at most $M$ balls $B(v)$ for $v \in V_{n}$. Furthermore each ball $B(v)$ for $v \in H^{B}$ has radius at most $\tau r$ by the definition of $H^{B}$, so for each $n \in \Z$ we must have that
\[
B \subset \bigcup_{v \in H_{n}^{B}} B(v) \subset 3\tau B.
\]
By the doubling property of $\nu$ and the bounded overlap of the balls $B(v)$ associated to vertices $v \in H_{n}^{B}$, we conclude that
\[
\sum_{v \in H_{n}^{B}} \nu(B(v)) \asymp \nu(B). 
\]
Applying this comparison, summing the resulting geometric series in \eqref{sum hull}, and then using $\alpha^{-m} \asymp r$ gives 
\[
\mu_{\beta}(H^{B}) \asymp r^{\beta} \nu(B),
\]
as desired. To obtain the corresponding result for $\hat{B}$, we use Lemma \ref{hull approximation} together with the fact that $\mu_{\beta}(\p X_{\rho}) = 0$ to obtain that 
\[
\mu_{\beta}(H^{C^{-1}B}) \leq \mu_{\beta}(\hat{B}) \leq \mu_{\beta}(H^{CB}),
\]
with $C = C(\alpha,\tau)$. By \eqref{hull compute} and the doubling property for $\nu$ we have $\mu_{\beta}(H^{C^{-1}B}) \asymp r^{\beta}\nu(B)$ and $\mu_{\beta}(H^{CB}) \asymp r^{\beta}\nu(B)$, so the comparison for $\mu_{\beta}(\hat{B})$ follows. 
\end{proof}

\begin{prop}\label{all beta doubling}
The metric measure spaces $(X_{\rho},d_{\rho},\mu_{\beta})$ and $(\bar{X}_{\rho},d_{\rho},\mu_{\beta})$ are each doubling and support a $1$-Poincar\'e inequality, with constants depending only on  $\alpha$, $\tau$, $C_{\nu}$ and $\beta$. 
\end{prop}

\begin{proof}
We start by proving that the metric measure space $(\bar{X}_{\rho},d_{\rho},\mu_{\beta})$ is doubling with constant $C_{\mu_{\beta}}$ depending only on $\alpha$, $\tau$, $C_{\nu}$, and $\beta$. Since $\mu_{\beta}(\p X_{\rho}) = 0$ it then follows that $(X_{\rho},d_{\rho},\mu_{\beta})$ is also doubling with the same doubling constant. Let $\beta > 0$ be given. We will verify the doubling property using the uniformly local doubling property from Proposition \ref{local hyp} with $R_{0} = 1$ and a criterion established in our previous work \cite[Proposition 3.3]{Bu22}. For this criterion we are given a cutoff $\kappa = \kappa(\alpha,\tau) > 0$ determined by the data used to uniformize $X$ with the density $\rho$ (the conclusions of \cite[Theorem 1.12]{Bu20} show that all of this data depends only on $\alpha$ and $\tau$) and we must show that there is a constant $C_{0}$ such that whenever $z \in Z \cong \p X_{\rho}$, $r > 0$, and $x \in X$ are such that $B_{\rho}(x,\kappa r) \subset B_{\rho}(z,r)$ and $d_{\rho}(x) \geq 2\kappa r$, we have the inequality
\begin{equation}\label{controlled verify}
\mu_{\beta}(B_{\rho}(z,r)) \leq C_{0}r^{\beta}\mu(B_{X}(x,1)).
\end{equation}
We recall that $\beta$ here corresponds to $\beta \e$ in \cite{Bu22} with $\e = \log \alpha$. We can then conclude from \cite[Proposition 3.3]{Bu22} that $\mu_{\beta}$ is doubling on $\bar{X}_{\rho}$ with doubling constant $C_{\mu_{\beta}}$ depending only on $\alpha$, $\tau$, and $C_{0}$; our calculations will show that we can take $C_{0} = C_{0}(\alpha,\tau,C_{\nu},\beta)$, so in fact we can take  $C_{\mu_{\beta}}$ to depend only on $\alpha$, $\tau$, $C_{\nu}$, and $\beta$ (note that the constants in Proposition \ref{local hyp} also depend only on these same parameters). 

Assume that $z \in Z$ and $x \in X$ are given as specified above. It then follows from the assumptions that $d_{\rho}(x) \asymp r$ since $d_{\rho}(x) \geq 2\kappa r$ and $x \in B_{\rho}(z,r)$. Viewing $B_{\rho}(z,r)$ as $\hat{B}$ for the ball $B = B_{Z}(z,r)$ in $Z$, we conclude from Lemma \ref{hull measure} that it suffices to produce a bound
\[
\nu(B) \ls \mu(B_{X}(x,1)),
\]
with implied constant depending only on $\alpha$, $\tau$, $C_{\nu}$, and $\beta$. Let $e$ be an edge in $X$ containing $x$. Since $\bar{X}_{\rho}$ is geodesic, we can find a vertex $v$ on $e$ such that $v \in \hat{B}$ since $x \in \hat{B}$. Then $e \subset B_{X}(x,1)$ and $\mu(e) \asymp \nu(B(v))$ by \eqref{measure to edge}. Hence it suffices to show that
\begin{equation}\label{target doubling}
\nu(B) \ls \nu(B(v)). 
\end{equation}
To prove this bound we note by Lemma \ref{hull approximation} that there is a constant $C = C(\alpha,\tau)$ such that $\hat{B} \subset H^{CB}$, so that in particular we have $v \in H^{CB}$. Thus $B(v) \cap CB \neq \emptyset$. Since $d_{\rho}(x) \asymp r$ we have by Lemma \ref{filling compute distance} that $\alpha^{-h(x)} \asymp r$, which implies that $\alpha^{-h(v)} \asymp r$. Thus $r(B(v)) \asymp r$. Since $r = r(B)$ and $B(v) \cap CB \neq \emptyset$, we conclude that $B \subset C'B(v)$ for a constant $C' = C'(\alpha,\tau)$. The desired inequality \eqref{target doubling} then follows from the doubling property for $\nu$.   

We thus obtain that $\mu_{\beta}$ is doubling on $\bar{X}_{\rho}$ for all $\beta > 0$ with constant $C_{\mu_{\beta}}$ depending only on $\alpha$, $\tau$, $C_{\nu}$, and $\beta$. The $1$-Poincar\'e inequality for the metric measure spaces $(X_{\rho},d_{\rho},\mu_{\beta})$ and $(\bar{X}_{\rho},d_{\rho},\mu_{\beta})$ then follows by combining the uniformly local $1$-Poincar\'e inequality with $R_{0} = 1$ from Proposition \ref{local hyp} with \cite[Theorem 1.3]{Bu22}. 
\end{proof}

We note the following corollary of Lemma \ref{hull measure} and Proposition \ref{all beta doubling}, which provides an estimate for the $\mu_{\beta}$-measure of balls centered at any point of $\bar{X}_{\rho}$. Closeness below is taken with respect to the metric $d_{\rho}$. 

\begin{cor}\label{interior ball measure}
Let $x \in \bar{X}_{\rho}$ and $r > 0$ be given. If $r \geq d_{\rho}(x)$ then we let $z \in Z$ be a point closest to $x$, while if $r \leq d_{\rho}(x)$ then we let $v \in V$ be a vertex of $X_{\rho}$ nearest to $x$. Then in the case $r \geq d_{\rho}(x)$ we have
\begin{equation}\label{big interior}
\mu_{\beta}(B_{\rho}(x,r)) \asymp r^{\beta}\nu(B_{Z}(z,r)), 
\end{equation}
while in the case  $r \leq d_{\rho}(x)$ we have
\begin{equation}\label{small interior}
\mu_{\beta}(B_{\rho}(x,r)) \asymp rd_{\rho}(x)^{\beta-1}\nu(B(v)).
\end{equation}
\end{cor}

\begin{proof}
We first consider the case $r \geq d_{\rho}(x)$. In this case we have that $d_{\rho}(x,z) \leq r$ and therefore
\[
B_{\rho}(z,r) \subset B_{\rho}(x,r) \subset B_{\rho}(z,3r).
\]
The desired estimate then follows from Lemma \ref{hull measure} and the fact that $\nu$ is doubling. 

We now consider the case $r \leq d_{\rho}(x)$. We can then use \cite[Lemma 2.4]{Bu22} and Lemma \ref{filling compute distance} to obtain for some constants $C_{*} = C_{*}(\alpha,\tau)$ and $C_{0} = C_{0}(\alpha,\tau)$ that
\[
B_{\rho}\left(x,\frac{r}{2}\right) \subset B_{X}\left(x,\frac{C_{*}r}{d_{\rho}(x)}\right) \subset B_{X}(x,C_{0}). 
\]
We remark that the hypotheses of \cite[Lemma 2.4]{Bu22} are satisfied with parameters depending only on $\alpha$ and $\tau$ by \cite[Theorem 1.12]{Bu20}. Let $G$ be a minimal subgraph of $X$ such that $B_{X}(x,C_{0}) \subset G$. Since $X$ has vertex degree bounded by $N = N(\alpha,\tau,C_{\nu})$, $G$ has a number of edges $M = M(\alpha,\tau,C_{\nu})$ uniformly bounded in terms of $\alpha$, $\tau$, and $C_{\nu}$. Furthermore by Lemma \ref{doubling vertex comparison} the measure $\mu_{\beta}$ restricted to $G$ is uniformly comparable to the measure $\alpha^{-\beta h(v)}\hat{\mu}(v)\mathcal{L}|_{G}$ with comparison constants depending only on $\alpha$, $\tau$, $C_{\nu}$, and $\beta$, recalling that $\mathcal{L}$ denotes the measure on $X$ that restricts to Lebesgue measure on each edge of $X$. The ball $B_{\rho}\left(x,\frac{r}{2}\right)$ can be written as a union of at most $M$ $d_{\rho}$-geodesics starting from $x$. These geodesic segments have length comparable to $r \alpha^{h(v)}$ in $X$ since $h(y) \doteq h(v)$ for $y \in G$. Applying $\mu_{\beta}$ to these segments gives the upper bound
\[
B_{\rho}\left(x,\frac{r}{2}\right) \ls r \alpha^{(1-\beta) h(v)}\hat{\mu}(v). 
\]
On the other hand the ball $B_{\rho}(x,\frac{r}{2})$ must contain at least one $d_{\rho}$-geodesic segment $\sigma$ of length $\frac{1}{2}r$ (with respect to $d_{\rho}$). Evaluating $\mu_{\beta}$ on $\sigma$ gives the lower bound, 
\[
B_{\rho}\left(x,\frac{r}{2}\right) \gs r \alpha^{(1-\beta)h(v)}\hat{\mu}(v).
\]
Applying Lemma \ref{filling compute distance} again, we obtain $\alpha^{(1-\beta)h(v)} \asymp \alpha^{(1-\beta)h(x)} \asymp d_{\rho}(x)^{\beta -1}$. Combining this with the equality $\hat{\mu}(v) = \nu(B(v))$ and using the doubling property of $\mu_{\beta}$ gives the second estimate \eqref{small interior}. 
\end{proof}

For any doubling measure $\nu$ on $Z$ there is always some $Q > 0$ such that we have for any $z \in Z$ and $0 < r' \leq r$, 
\begin{equation}\label{lower volume}
\frac{\nu(B(z,r'))}{\nu(B(z,r))} \geq C_{\mathrm{low}}^{-1} \left(\frac{r'}{r}\right)^{Q},
\end{equation}
for some constant $C_{\mathrm{low}} \geq 1$. See \cite[Lemma 8.1.13]{HKST}. One may always take $Q = \log_{2} C_{\nu}$, but it is possible that \eqref{lower volume} holds for smaller values of $Q$.  We say that $\nu$ has \emph{relative lower volume decay of order $Q$} if inequality \eqref{lower volume} holds for all $z \in Z$, all $0 < r' \leq r$, and some implied constant. The exponent $Q$ functions as a kind of dimension for $\nu$, especially for the purpose of embedding theorems \cite[Section 5]{HK00}. 

In the next lemma we obtain an estimate on the relative lower volume decay exponent for $\mu_{\beta}$ in terms of the corresponding exponent for $\nu$. Compare \cite[Lemma 10.6]{BBS21}. 

\begin{lem}\label{doubling dimension}
Suppose that $\nu$ has relative lower volume decay of order $Q > 0$. Then $\mu_{\beta}$ has relative lower volume decay of order $Q_{\beta} = \max\{1,Q+\beta\}$ on $\bar{X}_{\rho}$ with constant $C_{\mathrm{low}}'$ depending only on $\alpha$, $\tau$, $\beta$, $Q$, and the constant $C_{\mathrm{low}}$ in \eqref{lower volume}.
\end{lem}

\begin{proof}
This is a direct consequence of the estimates of Corollary \ref{interior ball measure}. Throughout this proof implied constants are allowed to also depend on $Q$ and the constant $C_{\mathrm{low}}$ in \eqref{lower volume}. Let $x \in \bar{X}_{\rho}$ and $0 < r' \leq r$ be given. If $r \leq d_{\rho}(x)$ then $r' \leq d_{\rho}(x)$ as well and applying the estimate \eqref{small interior} for both of them gives
\[
\frac{\mu_{\beta}(B_{\rho}(x,r'))}{\mu_{\beta}(B_{\rho}(x,r))} \asymp \frac{r'}{r}.
\]
Thus in this case \eqref{lower volume} holds with exponent $1$.

Similarly if $r' \geq d_{\rho}(x)$ then $r \geq d_{\rho}(x)$ and we can apply the estimate \eqref{big interior} to both of them. This gives, for a nearest point $z \in Z$ to $x$ with respect to the metric $d_{\rho}$,
\[
\frac{\mu_{\beta}(B_{\rho}(x,r'))}{\mu_{\beta}(B_{\rho}(x,r))} \asymp \left(\frac{r'}{r}\right)^{\beta}\frac{\nu(B(z,r'))}{\nu(B(z,r))} \gs \left(\frac{r'}{r}\right)^{Q+\beta}.
\]
Thus \eqref{lower volume} holds with exponent $Q+\beta$. 

Finally we consider the case that $r' \leq d_{\rho}(x)$ and $r \geq d_{\rho}(x)$. We set $r'' = d_{\rho}(x)$ and write 
\[
\frac{\mu_{\beta}(B_{\rho}(x,r'))}{\mu_{\beta}(B_{\rho}(x,r))} = \frac{\mu_{\beta}(B_{\rho}(x,r''))}{\mu_{\beta}(B_{\rho}(x,r))}\frac{\mu_{\beta}(B_{\rho}(x,r))}{\mu_{\beta}(B_{\rho}(x,r''))}.
\]
The first case can be applied to the first ratio and the second case can be applied to the second ratio. This gives
\[
\frac{\mu_{\beta}(B_{\rho}(x,r'))}{\mu_{\beta}(B_{\rho}(x,r))} \gs \frac{r'}{r''}\left(\frac{r''}{r}\right)^{Q+\beta}.
\]
If $Q+\beta \geq 1$ then since $r' \leq r''$ this implies that 
\[
\frac{r'}{r''}\left(\frac{r''}{r}\right)^{Q+\beta} \geq \left(\frac{r'}{r}\right)^{Q+\beta}.
\]
Thus \eqref{lower volume} holds with exponent $Q+\beta$ in this subcase. If $Q+\beta\leq 1$ then we instead use the fact that $r'' \leq r$ to obtain
\[
\frac{r'}{r''}\left(\frac{r''}{r}\right)^{Q+\beta} \geq \frac{r'}{r}.
\]
Thus \eqref{lower volume} holds with exponent $1$ in this subcase. 
\end{proof}

We conclude this section with two interesting propositions on the properties of $\mu_{\beta}$ and $\mu$. The first proposition shows that we always have $\mu_{\beta}(X_{\rho}) = \infty$, even when $\nu(Z) < \infty$. For each $n \in \Z$ we let $E_{n}$ denote the set of all edges in $V$ which have at least one vertex in $V_{n}$. We note that our definition of a metric measure space forces $0 < \nu(Z) \leq \infty$, since any ball $B \subset Z$ must satisfy $0 < \nu(B) < \infty$.

\begin{prop}\label{infinite measure}
We have $\mu_{\beta}(E_{n}) \asymp \alpha^{-\beta n}\nu(Z)$ for each $n \in \Z$. Consequently $\mu_{\beta}(X_{\rho}) = \infty$. 
\end{prop}

\begin{proof}
By the estimate \eqref{transformed} we have for any edge $e \in E_{n}$ that $\mu_{\beta}(e) \asymp \alpha^{-\beta n}\nu(B(v))$ for a vertex $v \in e \cap V_{n}$. Since $X$ has vertex degree bounded by $N = N(\alpha,\tau,C_{\nu})$, since each vertex $v \in V_{n}$ is attached to at least one edge, and since the balls $B(v)$ for $v \in V_{n}$ cover $Z$ and have bounded overlap by Lemma \ref{doubling overlap}, we conclude that
\[
\sum_{e \in E_{n}}\mu_{\beta}(e) \asymp \sum_{v \in V_{n}}\alpha^{-\beta n}\nu(B(v)) \asymp \alpha^{-\beta n}\nu(Z). 
\]
This proves the main estimate. The conclusion $\mu_{\beta}(X_{\rho}) = \infty$ follows by letting $n \rightarrow -\infty$. 
\end{proof}

The second proposition shows that the conclusions of Proposition \ref{all beta doubling} do not extend to the case $\beta = 0$. In fact every ball in $\bar{X}_{\rho}$ centered at a point of $\p X_{\rho}$ has infinite measure with respect to $\mu$. We recall that for a ball $B = B_{Z}(z,r) \subset Z$ we write $\hat{B} = B_{\rho}(z,r) \subset \bar{X}_{\rho}$ for the corresponding ball with the same center and radius in $\bar{X}_{\rho}$. 

\begin{prop}\label{ball infinite measure}
For each ball $B \subset Z$ we have $\mu(H^{B}) = \infty$. Consequently for each ball $B \subset Z$ we have $\mu(\hat{B}) = \infty$. 
\end{prop}

\begin{proof}
Let $B \subset Z$ be a given ball of radius $r > 0$. The estimate \eqref{vertex transformed} also holds for $\beta = 0$ with $\mu$ in place of $\mu_{\beta}$ by Proposition \ref{doubling vertex comparison} and the definition of $\mu$, with implied constant depending only on $\alpha$, $\tau$, and $C_{\nu}$. Thus for each $n \in \Z$ with $\alpha^{-n} \leq r$ we have
\[
\mu\left(\bigcup_{v \in H_{n}^{B}} \bar{B}_{X}\left(v,\frac{1}{2}\right)\right) = \sum_{v \in H_{n}^{B}}\mu\left(\bar{B}_{X}\left(v,\frac{1}{2}\right)\right) \gs \sum_{v \in H_{n}^{B}} \nu(B(v)) \gs \nu(B),
\] 
since the balls $B(v)$ for $v \in H_{n}^{B}$ cover $B$ with bounded overlap. Summing this estimate over the infinitely many $n$ such that $\alpha^{-n} \leq r$ gives $\mu(H^{B}) = \infty$ since $\nu(B) > 0$. By Lemma \ref{ball to hull} we can find a constant $C \geq 1$ such that $H^{C^{-1}B} \subset \hat{B}$. Since $\mu(H^{C^{-1}B}) = \infty$ by the above calculation, we conclude that $\mu(\hat{B}) = \infty$ as well. 
\end{proof}

\section{Trace theorems}\label{sec:trace}

In this section we carry over the concepts and notation from Section \ref{sec:lift}; we refer the reader back to the start of Section \ref{sec:lift} for an overview of the setting and notation. Since we showed in Proposition \ref{all beta doubling} that the geodesic doubling metric measure spaces $(X_{\rho},d_{\rho},\mu_{\beta})$ and $(\bar{X}_{\rho},d_{\rho},\mu_{\beta})$ each support a $1$-Poincar\'e inequality, they also support a $p$-Poincar\'e inequality for each $p \geq 1$. Hence the results of Section \ref{sec:capacities} apply to the metric measure spaces $(X_{\rho},d_{\rho},\mu_{\beta})$ and $(\bar{X}_{\rho},d_{\rho},\mu_{\beta})$ for each fixed $p \geq 1$. 

We will generalize the trace results of \cite[Section 11]{BBS21} to the case of a potentially unbounded complete doubling metric space $Z$. Our treatment of the trace results in this context will be slightly different than that of \cite{BBS21}, falling closer in spirit to the work of Bonk-Saksman-Soto \cite{BSS18} in that we will express the trace as a  $\nu$-a.e.~ limit of Lipschitz functions on $Z$. Throughout this section our parameter $\beta$ corresponds to the ratio $\beta/\e$ in \cite[Section 11]{BBS21} with $\e = \log \alpha$. We will always assume that $p > \beta$ wherever $p$ and $\beta$ appear together.

\begin{rem}\label{graph capacity}
The statements from \cite[Remark 9.5]{BBS21} carry through without modification to our setting, as they only rely on the fact that the metric measure space $(X_{\rho},d_{\rho},\mu_{\beta})$ is a metric graph with $\mu_{\beta}$ being comparable to a multiple of Lebesgue measure on each edge. Given $p \geq 1$, the only family of nonconstant compact rectifiable curves in $X_{\rho}$ with zero $p$-modulus (with respect to $\mu_{\beta}$) is the empty family. Thus any $p$-weak upper gradient for $u$ on $X_{\rho}$ is an upper gradient for $u$.   Since functions in $N^{1,p}_{\loc}(X_{\rho})$ are absolutely continuous along $p$-a.e.~ curve \cite[Proposition 6.3.2]{HKST} this also implies that any function $u \in \t{N}^{1,p}_{\loc}(X_{\rho})$ is continuous on $X_{\rho}$ and absolutely continuous along each edge of $X_{\rho}$. Restricted to each edge the minimal upper gradient $g_{u}$ of $u$ is given by $g_{u} = \left|\frac{du}{ds_{\rho}}\right|$, with $ds_{\rho}$ denoting arclength with respect to the distance $d_{\rho}$ and $\frac{du}{ds_{\rho}}$ denoting the \emph{metric differential} of $u$ on this edge given by the absolute continuity of $u$ on this edge (see \cite[Theorem 4.4.8]{HKST}). All points in $X_{\rho}$ have positive $p$-capacity, and each equivalence class in $\t{N}^{1,p}(X_{\rho}) = N^{1,p}(X_{\rho})$ consists of a single function that is continuous. Lastly all of these statements remain true of any metric subgraph of $X_{\rho}$, equipped with the restriction of the measure $\mu_{\beta}$ to this subgraph. 
\end{rem}

The following invaluable inequality is an immediate consequence of the convexity of the function $t \rightarrow t^{p}$ for $p \geq 1$ on $[0,\infty)$. We will use it throughout this section largely without comment. 

\begin{lem}\label{convexity lemma}
Let $\{x_{i}\}_{i=1}^{k}$ be nonnegative real numbers. Then for any $p \geq 1$ we have
\[
\left(\sum_{i=1}^{k}x_{i}\right)^{p} \leq k^{p-1}\sum_{i=1}^{k}x_{i}^{p}. 
\]
\end{lem}

In the construction of our trace operator we will need to make use of the following proposition. 

\begin{prop}\label{partition of unity}\cite[Lemma B.7.4]{S99}
Let $(Z,d)$ be a doubling metric space with doubling constant $D$. Let $r > 0$ be given. Let $\{z_{n}\}_{n \in J}$ be a maximal $r$-separated subset of $Z$ indexed by $J \subset \N$. Then there is a corresponding collection $\{\psi_{n}\}_{n \in J}$ of functions $\psi_{n}: Z \rightarrow [0,1]$ such that for each $n \in J$ we have
\begin{equation}\label{partition support}
\mathrm{supp}(\psi_{n}) \subset B(z_{n},2r),
\end{equation}
and for each $z \in Z$, 
\begin{equation}\label{partition sum}
\sum_{n \in J} \psi_{n}(z) = 1,
\end{equation}
and $\psi_{n}$ is $C r^{-1}$-Lipschitz for each $n \in S$ with $C = C(D)$ depending only on the doubling constant $D$. 
\end{prop}

The collection of functions  $\{\psi_{n}\}_{n \in J}$ will be referred to as a \emph{Lipschitz partition of unity}. As noted after \cite[(10)]{BS18} the condition on $\mathrm{supp}(\psi_{n})$ can be obtained by a slight modification of the proof in the reference. We remark that by Lemma \ref{doubling overlap} the sum \eqref{partition sum} always has only finitely many nonzero terms for each fixed choice of $z \in Z$. 

%We also note that if $Z$ is unbounded then we can always choose the index set $J$ to satisfy $J = \N$ by renumbering the indices. 

Using Proposition \ref{partition of unity} we fix, for each $n \in \Z$, a Lipschitz partition of unity $\{\psi_{v}\}_{v \in V_{n}}$ associated to the $\alpha^{-n}$-separated subset $V_{n}$ of $Z$, considering these vertices of $X$ as points of $Z$.  Since we require $\tau > 3$, the condition \eqref{partition support} implies that $\supp(\psi_{v}) \subset B(v)$ for all $v \in V$. Since $0 \leq \psi_{v} \leq 1$ we have the bound 
\begin{equation}\label{partition bound}
\|\psi_{v}\|_{L^{1}(Z)} \leq \nu(B(v)).
\end{equation}

We make some additional definitions here for future reference. For $n \in \Z$ we write $X_{\geq n} = X \cap h^{-1}([n,\infty))$ for the set of all points in $X$ of height at least $n$ and write $X_{\leq n} = X \cap h^{-1}((-\infty,n])$ for the set of all points in $X$ of height at most $n$. We consider each of these subsets as being equipped with the metric $d_{\rho}$. For a ball $B \subset Z$, its hull $H^{B} \subset X_{\rho}$, and any integer $n \in \Z$ we then set $H^{B}_{\geq n} = H^{B} \cap X_{\geq n}$ and $H^{B}_{\leq n} = H^{B} \cap X_{\leq n}$. 

All implied constants throughout this section will depend only on $\alpha$, $\tau$, $C_{\nu}$, $\beta$, and the exponent $p > \beta$. By Remark \ref{graph capacity} we have $\t{N}^{1,p}_{\loc}(X_{\rho}) = N^{1,p}_{\loc}(X_{\rho})$ and consequently $\t{N}^{1,p}(X_{\rho}) = N^{1,p}(X_{\rho})$. 

Our first proposition constructs the trace of a Newtonian function defined on the hull $H^{B} \subset X_{\rho}$ of a ball $B \subset Z$. We recall from Lemma \ref{hull closure} that for a ball $B \subset Z$ we have $\p H^{B} = \bar{B}$. We will need the following lemma. 

\begin{lem}\label{arc length lemma}
Let $e \in E$ be any edge of $X$ and let $v \in e$ denote either vertex on $e$. Then restricted to $e$ we have the comparison
\begin{equation}\label{arc length measure}
ds_{\rho}|_{e} \asymp \frac{\alpha^{(\beta-1)h(v)}d\mu_{\beta}|_{e}}{\nu(B(v))},
\end{equation}
where $ds_{\rho}$ denotes arc length with respect to the metric $d_{\rho}$. 
\end{lem}

\begin{proof}
Let $e \in E$ be a given edge and let $v \in e$ denote a vertex on $e$. By the definition of $X_{\rho}$ we have on $e$, 
\begin{equation}\label{1 arc length}
ds_{\rho}|_{e} \asymp \alpha^{-h(v)} d\mathcal{L}|_{e},
\end{equation}
where we recall that $\mathcal{L}$ is the measure on $X$ given by Lebesgue measure on each unit length edge of $X$. On $e$ we also have by the definition \eqref{lift definition} of the measure $\mu$ and by Lemma \ref{doubling vertex comparison},
\begin{equation}\label{2 arc length}
d\mathcal{L}|_{e} \asymp \frac{d\mu|_{e}}{\hat{\mu}(v)} = \frac{d\mu|_{e}}{\nu(B(v))} \asymp \frac{d\mu_{\beta}|_{e}}{\alpha^{-\beta h(v)}\nu(B(v))}.
\end{equation}
The comparison \eqref{arc length measure} follows by combining the comparisons \eqref{1 arc length} and \eqref{2 arc length}. 
\end{proof}

For defining $T_{n}u$ in \eqref{trace formula} below we recall that functions in Newtonian spaces are required to be defined pointwise everywhere; in this particular case the functions are actually continuous by Remark \ref{graph capacity}. We consider $H^{B}$ as being equipped with the restriction $\mu_{\beta}|_{H^{B}}$ of the measure $\mu_{\beta}$ to this subset. 

\begin{prop}\label{Lp ball trace}
Let $B \subset Z$ be any ball in $Z$ of radius $r > 0$ and let $u \in N^{1,p}(H^{B})$ be given. Then $u$ has a trace $Tu \in L^{p}(B)$ given as follows: for each $n \in \Z$ such that $\alpha^{-n} \leq r$  and each $z \in B$ we set 
\begin{equation}\label{trace formula}
T_{n}u(z) = \sum_{v \in V_{n}}u(v)\psi_{v}(z) = \sum_{v \in H_{n}^{B}}u(v)\psi_{v}(z),
\end{equation}
then we have $T_{n}u \rightarrow Tu$ in $L^{p}(B)$. Furthermore, letting $k$ be the minimal integer such that $\alpha^{-k} \leq r$, we have the following estimate for any $p$-integrable upper gradient $g$ of $u$ on $H^{B}$ and any integer $n \geq k$, 
\begin{equation}\label{ball p trace}
\|Tu-T_{n}u\|_{L^{p}(B)} \ls \alpha^{(\beta/p-1) n}\|g\|_{L^{p}\left(H_{\geq n}^{B}\right)}.
\end{equation}
\end{prop}

We remark that the second equality in \eqref{trace formula} follows from the fact that $\psi_{v}(z) \neq 0$ implies that $z \in B(v)$ and therefore $B(v) \cap B \neq \emptyset$ since $z \in B$, which implies that $v \in H^{B}$ provided that $\alpha^{-h(v)} \leq r$.

\begin{proof}
 
Let $u \in \t{N}^{1,p}(H^{B})$ be a given function, which is continuous by Remark \ref{graph capacity}. We let $g \in L^{p}(H^{B})$ be a $p$-integrable upper gradient for $u$. For $z \in B$ we write $H_{n}^{B}(z)$ for the set of $v \in H_{n}^{B}$ such that $z \in B(v)$. We note that by the discussion in the previous paragraph we have for all $z \in B$ and $n \in \Z$ such that $\alpha^{-n} \leq r$, 
\[
\sum_{v \in H_{n}^{B}(z)}\psi_{v}(z) = 1.
\] 
We can then estimate
\begin{equation}\label{successor estimate}
|T_{n+1}u(z)-T_{n}u(z)| \leq \sum_{v' \in H_{n+1}^{B}(z)} \sum_{v \in H_{n}^{B}(z)}|u(v')-u(v)|\psi_{v'}(z)\psi_{v}(z).
\end{equation}
Observe that we can only have $\psi_{v'}(z) \psi_{v}(z) \neq 0$ if $B(v) \cap B(v') \neq \emptyset$, which implies that there is a vertical edge $e_{vv'}$ from $v$ to $v'$. Since $g$ is an upper gradient for $u$ on $H^{B}$, we thus conclude in this case that
\begin{equation}\label{upper gradient estimate}
|u(v')-u(v)| \leq \int_{e_{vv'}} g \, ds_{\rho} \leq \sum_{e \in \mathcal{U}(v)} \int_{e} g \, ds_{\rho},
\end{equation}
with the sum being taken over the set $\mathcal{U}(v)$ of all upward directed vertical edges $e$ starting from $v$, and with $ds_{\rho}$ denoting arclength with respect to the metric $d_{\rho}$. By Lemma \ref{arc length lemma} this implies that
\begin{equation}\label{sharpened}
|u(v')-u(v)| \ls \frac{\alpha^{(\beta-1)n}}{\nu(B(v))}\sum_{e \in \mathcal{U}(v)} \int_{e} g \, d\mu_{\beta}.
\end{equation}

Applying the estimate \eqref{sharpened} to the inequality \eqref{successor estimate} gives
\begin{equation}\label{refined successor estimate}
|T_{n+1}u(z)-T_{n}u(z)| \ls \alpha^{(\beta-1)n}\sum_{v \in H_{n}^{B}(z)}\left(\sum_{e \in \mathcal{U}(v)} \int_{e} g \, d\mu_{\beta}\right)\frac{\psi_{v}(z)}{\nu(B(v))}.
\end{equation}
Let $m > n$ be a given integer. Summing inequality \eqref{refined successor estimate} above from $n$ to $m-1$, we obtain that
\begin{align}\label{iterated successor estimate}
|T_{m}u(z)-T_{n}u(z)| &\ls \sum_{j=n}^{m-1}\alpha^{(\beta-1)j}\sum_{v \in H_{j}^{B}(z)}\left(\sum_{e \in \mathcal{U}(v)} \int_{e} g \, d\mu_{\beta}\right)\frac{\psi_{v}(z)}{\nu(B(v))} \\
&\leq \sum_{j=n}^{\infty}\alpha^{(\beta-1)j}\sum_{v \in H_{j}^{B}(z)}\left(\sum_{e \in \mathcal{U}(v)} \int_{e} g \, d\mu_{\beta}\right)\frac{\psi_{v}(z)}{\nu(B(v))}.
\end{align}
Let's assume for the moment that $p > 1$ and let $\la > 0$ be a given parameter. By applying H\"older's inequality for sequences on the right side of the above inequality with the conjugate exponents $p$ and $q = \frac{p}{p-1}$, we obtain for $m > n$ and $z \in B$, 
\begin{align*}
|T_{m}u(z)-T_{n}u(z)| &\lesssim \sum_{j=n}^{\infty}\sum_{v \in H_{j}^{B}(z)}\left(\alpha^{(\beta+\la-1)j}\sum_{e \in \mathcal{U}(v)} \int_{e} g \, d\mu_{\beta}\right)\frac{\psi_{v}(z)}{\nu(B(v))}\alpha^{-j \la} \\
&\leq \left(\sum_{j=n}^{\infty}\alpha^{p(\beta+\la-1)j}\left(\sum_{v \in H_{j}^{B}(z)}\left(\sum_{e \in \mathcal{U}(v)} \int_{e} g \, d\mu_{\beta}\right)\frac{\psi_{v}(z)}{\nu(B(v))}\right)^{p}\right)^{1/p}\left(\sum_{j=n}^{\infty}\alpha^{-q  \la j}\right)^{1/q} \\
&\lesssim \alpha^{-n \la} \left(\sum_{j=n}^{\infty}\alpha^{p(\beta+\la-1)j}\left(\sum_{v \in H_{j}^{B}(z)}\left(\sum_{e \in \mathcal{U}(v)} \int_{e} g \, d\mu_{\beta}\right)\frac{\psi_{v}(z)}{\nu(B(v))}\right)^{p}\right)^{1/p},
\end{align*}
where now the implied constants also depend on $\la$.

By Lemma \ref{doubling overlap} the sets $H_{j}^{B}(z)$ have a number of elements uniformly bounded in terms of $\alpha$, $\tau$, and $C_{\nu}$ for each $j \geq n$ and $z \in B$. Using Lemma \ref{convexity lemma}, it follows from this, and the fact that $X$ has vertex degree uniformly bounded in terms of this same data, that we have 
\begin{align*}
\left(\sum_{v \in H_{j}^{B}(z)}\left(\sum_{e \in \mathcal{U}(v)} \int_{e} g \, d\mu_{\beta}\right)\frac{\psi_{v}(z)}{\nu(B(v))}\right)^{p} &\lesssim \sum_{v \in H_{j}^{B}(z)}\left(\sum_{e \in \mathcal{U}(v)} \int_{e} g \, d\mu_{\beta}\right)^{p}\frac{\psi_{v}(z)^{p}}{\nu(B(v))^{p}} \\
&\leq \sum_{v \in H_{j}^{B}(z)}\sum_{e \in \mathcal{U}(v)} \left(\int_{e} g \, d\mu_{\beta}\right)^{p}\frac{\psi_{v}(z)}{\nu(B(v))^{p}},
\end{align*}
where we have used $0 \leq \psi_{v}(z) \leq 1$ so that $\psi_{v}(z)^{p} \leq \psi_{v}(z)$. By Jensen's inequality for integrals of convex functions on probability spaces we have for any edge $e \in X$, 
\begin{align}\label{Jensen}
\left(\int_{e} g \, d\mu_{\beta}\right)^{p} &= \mu_{\beta}(e)^{p}\left(\dashint_{e} g \, d\mu_{\beta}\right)^{p} \\
&\leq \mu_{\beta}(e)^{p}\dashint_{e} g^{p} \, d\mu_{\beta} \\
&= \mu_{\beta}(e)^{p-1}\int_{e} g^{p} \, d\mu_{\beta}
\end{align}
Using this estimate together with the comparison \eqref{transformed}, we conclude that we have the inequality 
\begin{equation}\label{estimate I}
|T_{m}u(z)-T_{n}u(z)|^{p} \lesssim \alpha^{-\la p n  }\sum_{j=n}^{\infty} \alpha^{(\beta + p\la-p) j}\sum_{v \in H_{j}^{B}(z)} \left(\sum_{e \in \mathcal{U}(v)} \int_{e} g^{p} \, d\mu_{\beta}\right)\frac{\psi_{v}(z)}{\nu(B(v))},
\end{equation}
again with the implied constant additionally depending on $\la$. For future reference we note that inequality \eqref{estimate I} also holds for $p = 1$ as it is directly implied by the inequality \eqref{iterated successor estimate}. 

We now set $\la = (p-\beta)/p$. This simplifies \eqref{estimate I} to
\begin{equation}\label{substituted estimate I}
|T_{m}u(z)-T_{n}u(z)|^{p} \lesssim \alpha^{(\beta-p) n}\sum_{j=n}^{\infty}\sum_{v \in H_{j}^{B}(z)} \left(\sum_{e \in \mathcal{U}(v)} \int_{e} g^{p} \, d\mu_{\beta}\right)\frac{\psi_{v}(z)}{\nu(B(v))}.
\end{equation}
Integrating each side over $B$ and using \eqref{partition bound}, we conclude that for all $m > n \geq k$, 
\begin{align*}
\|T_{m}u-T_{n}u\|^{p}_{L^{p}(B)} &\ls \alpha^{(\beta-p) n}\sum_{j=n}^{\infty}\sum_{v \in H_{j}^{B}} \left(\sum_{e \in \mathcal{U}(v)} \int_{e} g^{p} \, d\mu_{\beta}\right) \\
&\leq \alpha^{(\beta-p)n}\|g\|_{L^{p}(H_{\geq n}^{B})}^{p}.
\end{align*}
By taking the $p$th root of each side, we obtain that
\begin{equation}\label{p norm estimate}
\|T_{m}u-T_{n}u\|_{L^{p}(B)} \ls \alpha^{(\beta/p-1) n}\|g\|_{L^{p}(H_{\geq n}^{B})}.
\end{equation}
In particular we have
\[
\|T_{m}u-T_{n}u\|_{L^{p}(B)} \ls \alpha^{(\beta/p-1)  n}\|g\|_{L^{p}(H^{B}_{\geq n})}.
\]
The right side converges to $0$ as $n \rightarrow \infty$ since $p > \beta$. We conclude that $\{T_{n}u\}$ defines a Cauchy sequence in $L^{p}(B)$ which therefore converges in $L^{p}(B)$ to a function $Tu \in L^{p}(B)$. Letting $m \rightarrow \infty$ in \eqref{p norm estimate} then gives \eqref{ball p trace}.
\end{proof}

The proof of Proposition \ref{Lp ball trace} enables us to deduce a Poincar\'e-type inequality relating the average of the trace $Tu$ on the ball $B$ to the $L^p$ norm of the upper gradient on the hull in $X_{\rho}$.  

\begin{prop}\label{Poincare type inequality}
Let $B \subset Z$ be a ball in $Z$ of radius $r > 0$ and let $u \in N^{1,p}(H^{B})$ be given. Then the trace $Tu \in L^{p}(B)$ defined by Proposition \ref{Lp ball trace} has the property that for any $p$-integrable upper gradient $g$ of $u$ on $H^{B}$, 
\begin{equation}\label{equation Poincare type}
\dashint_{B} |Tu-(Tu)_{B}|^{p} \,d\nu \ls r^{p}\dashint_{H^{B}}g^{p}\,d\mu_{\beta}.
\end{equation}
\end{prop}

\begin{proof}
Let $r  = r(B)$ be the radius of $B$ and let $k \in \Z$ be the minimal integer such that $\alpha^{-k} \leq r$, so that we have $\alpha^{-k} \asymp r$. Then by Lemma \ref{large inclusion} we can find a vertex $v_{0} \in V_{k}$ such that $B \subset B(v_{0})$. By the triangle inequality in $L^{p}(B)$ we have 
\[
\|Tu-u(v_{0})\|_{L^{p}(B)} \leq \|Tu-T_{k}u\|_{L^{p}(B)} + \|T_{k}u-u(v_{0})\|_{L^{p}(B)}. 
\]
The first term we estimate by \eqref{ball p trace}, noting that $\alpha^{(\beta/p-1)k} \leq r^{1-\beta/p}$ since $\alpha^{-k} \leq r$, 
\[
\|Tu-T_{k}u\|_{L^{p}(B)} \ls r^{1-\beta/p}\|g\|_{L^{p}\left(H^{B}\right)}.
\]
For the second term we observe that any $v \in V_{k}$ with the property that $\psi_{v}(z) \neq 0$ for some $z \in B$ must satisfy $B(v) \cap B(v_{0}) \neq \emptyset$ since $B \subset B(v_{0})$ and $\supp(\psi_{v}) \subset B(v)$.  Thus either $v = v_{0}$ or $v$ is joined to $v_{0}$ by a horizontal edge $e_{vv_{0}}$ in $X$. Furthermore in the second case we have that $v,v_{0} \in H^{B}$ since $B(v)$ and $B(v_{0})$ both contain $z \in B$. Using the fact that $g$ is an upper gradient for $u$ on $H^{B}$, we obtain that for each $z \in B$ we have the inequality
\[
|T_{k}u(z)-u(v_{0})| \leq \sum_{v \in V_{k}:\;|vv_{0}| = 1} |u(v)-u(v_{0})|\psi_{v}(z) \leq \sum_{v \in V_{k}:\;|vv_{0}| = 1} \left(\int_{e_{vv_{0}}}g\,ds_{\rho}\right)\psi_{v}(z). 
\]
By calculating using Lemma \ref{arc length lemma} as we did in the proof of Proposition \ref{Lp ball trace}, we obtain from the above the inequality for each $p \geq 1$ and $z \in B$ via an analogous computation, 
\[
|T_{k}u(z)-u(v_{0})|^{p} \ls \sum_{v \in V_{k}:\;|vv_{0}| = 1} \alpha^{(\beta-p)k}\left(\int_{e_{vv_{0}}}g^{p}\,d\mu_{\beta}\right)\frac{\psi_{v}(z)}{\nu(B(v))},
\]
Integrating this inequality over $B$, taking the $p$th root, noting again that $\alpha^{-k} \leq r$, and recalling that each edge $e_{vv_{0}}$ belongs to $H^{B}$, we conclude that we have a matching inequality 
\[
\|T_{k}u-u(v_{0})\|_{L^{p}(B)}\ls r^{1-\beta/p}\|g\|_{L^{p}\left(H^{B}\right)}.
\]
Thus
\[
\|Tu-u(v_{0})\|_{L^{p}(B)} \ls r^{1-\beta/p}\|g\|_{L^{p}\left(H^{B}\right)}.
\]
Taking $p$th powers and passing to averages over $B$ and $H^{B}$ then gives
\[
\dashint_{B} |Tu-u(v_{0})|^{p} \,d\nu \ls r^{p-\beta}\frac{\mu_{\beta}(H^{B})}{\nu(B)}\dashint_{H^{B}}g^{p}\,d\mu_{\beta}.
\]
By Lemma \ref{hull measure} we have $\frac{\mu_{\beta}(H^{B})}{\nu(B)} \asymp r^{\beta}$. Thus we conclude that 
\[
\dashint_{B} |Tu-u(v_{0})|^{p} \,d\nu \ls r^{p}\dashint_{H^{B}}g^{p}\,d\mu_{\beta}.
\]
Lastly we can replace $u(v_{0})$ with the average $(Tu)_{B}$ on the left via a standard inequality at the cost of a factor of $2^{p}$, see \cite[Lemma 4.17]{BB11}. 
\end{proof}

We explore other Poincar\'e type inequalities of the form \eqref{equation Poincare type} later in the paper, see for instance Propositions \ref{extension hyperbolic Poincare} and \ref{hyperbolic Poincare}.

The proposition below is an immediate consequence of Proposition \ref{Lp ball trace}. We recall that $\t{N}^{1,p}_{\loc}(X_{\rho}) = \t{D}^{1,p}_{\loc}(X_{\rho})$, so we will be formulating our theorems in terms of $\t{N}^{1,p}_{\loc}(X_{\rho})$.

\begin{prop}\label{Lp local trace}
Let $u \in \t{N}^{1,p}_{\loc}(X_{\rho})$ be given. Then $u$ has a trace $Tu \in L^{p}_{\loc}(Z)$ given as follows: for each $n \in \Z$ and each $z \in Z$ we set
\begin{equation}\label{local trace formula}
T_{n}u(z) = \sum_{v \in V_{n}}u(v)\psi_{v}(z).
\end{equation}
Then for each ball $B \subset Z$ we have $T_{n}u \rightarrow Tu$ in $L^{p}(B)$. Furthermore, for a given ball $B$ of radius $r > 0$ and $k$ the minimal integer such that $\alpha^{-k} \leq r$, we have the following estimate for any $p$-integrable upper gradient $g$ of $u$ on $H^{B}$ and any integer $n \geq k$,
\begin{equation}\label{local p trace}
\|Tu-T_{n}u\|_{L^{p}(B)} \ls \alpha^{(\beta/p-1)n}\|g\|_{L^{p}(H^{B}_{\geq n})}.
\end{equation}
If furthermore $u$ has a $p$-integrable upper gradient $g$ on $X_{\rho}$ then we have for each $n \in \Z$, 
\begin{equation}\label{global p trace}
\|Tu-T_{n}u\|_{L^{p}(Z)} \ls \alpha^{(\beta/p-1)n}\|g\|_{L^{p}(X_{\geq n})}.
\end{equation}
\end{prop}

\begin{proof}
Let $u \in \t{N}^{1,p}_{\loc}(X_{\rho})$ be given. If $B \subset Z$ is any ball of radius $r > 0$ then by Lemma \ref{hull approximation} we have that $H^{B} \subset C \hat{B}$ for some $C = C(\alpha,\tau) \geq 1$, where $\hat{B} = B_{\rho}(z,r) \subset \bar{X}_{\rho}$, which implies for any $x \in H^{B}$ that $H^{B} \subset B_{\rho}(x,2Cr) \cap X_{\rho}$. It follows that if $u \in \t{N}^{1,p}_{\loc}(X_{\rho})$ then $u|_{H^{B}} \in \t{N}^{1,p}(H^{B})$ for any ball $B \subset Z$. All of the claims of the proposition except for inequality \eqref{global p trace} then follow immediately from the corresponding claims of Proposition \ref{Lp ball trace} upon observing that the formulas \eqref{trace formula} and \eqref{local trace formula} defining $T_{n}u$ in each proposition are the same. 

We now assume that $u$ has a $p$-integrable upper gradient $g$ on $X_{\rho}$ and let $n \in \Z$ be given. For each vertex $v \in V_{n}$ we let $B(v)$ be the associated ball of radius $r(B(v)) = \tau \alpha^{-n}$, so that in particular we have $\alpha^{-n} \leq r(B(v))$. Then $g$ is a $p$-integrable upper gradient of $u$ on $H^{B(v)}$, so inequality \eqref{local p trace} implies that for each $v \in V_{n}$ we have
\[
\|(Tu-T_{n}u)\chi_{B(v)}\|_{L^{p}(Z)} \ls \alpha^{(\beta/p-1) n}\|g\|_{L^{p}(H^{B(v)}_{\geq n})}
\]
Since the balls $B(v)$ for $v \in V_{n}$ cover $Z$ with bounded overlap by Lemma \ref{doubling overlap}, we deduce from the triangle inequality in $L^{p}(Z)$ that
\[
\|Tu-T_{n}u\|_{L^{p}(Z)} \ls \alpha^{(\beta/p-1) n}\sum_{v \in V_{n}}\|g\|_{L^{p}(H^{B(v)}_{\geq n})}.
\]
Since the hulls $H^{B(v)}$ for $v \in V_{n}$ have bounded overlap by Lemma \ref{hull overlap}, we conclude from the above inequality that the estimate \eqref{global p trace} holds. 
\end{proof}

By Proposition \ref{Lp local trace} we have a linear \emph{trace operator} 
\[
T: \t{N}^{1,p}_{\loc}(X_{\rho}) \rightarrow L^{p}_{\loc}(Z),
\]
defined by $u \rightarrow Tu$. The domain of $T$ depends on both $p$ and $\beta$, however we will suppress this dependence in the notation.

We next show that $T$ restricts to a bounded linear operator $T: N^{1,p}(X_{\rho}) \rightarrow L^{p}(Z)$; we recall from Remark \ref{graph capacity} that each equivalence class in $\t{N}^{1,p}(X_{\rho})$ consists of a single continuous function, so we can consider $N^{1,p}(X_{\rho})$ to be canonically identified with $\t{N}^{1,p}(X_{\rho})$.

\begin{prop}\label{p integrable trace}
Let $u \in N^{1,p}(X_{\rho})$. Then $Tu \in L^{p}(Z)$ with the estimate
\[
\|Tu\|_{L^{p}(Z)} \ls \|u\|_{N^{1,p}(X_{\rho})}.
\]
\end{prop}

\begin{proof}
We define a function $\xi: X_{\rho} \rightarrow [0,1]$ by setting $\xi(x) = 1$ for $x \in X_{\geq 1}$, $\xi(x) = 0$ for $x \in X_{\leq 0}$, and linearly interpolating (with respect to the metric $d_{\rho}$) the values of $\xi$ on each vertical edge connecting $X_{\leq 0}$ to $X_{\geq 1}$. On a given vertical edge $e$ connecting $v$ to $w$ with $h(v) = 0$ and $h(w) = 1$ we have from Lemma \ref{filling estimate both} that $d_{\rho}(v,w) \asymp 1$. Thus there is a constant $L = L(\alpha,\tau)$ such that $\xi$ is $L$-Lipschitz on $X_{\rho}$. Since $\xi|_{X_{\leq 0}} = 0$, it follows that the scaled characteristic function $L\chi_{X_{\geq 0}}$ defines an upper gradient for $\xi$ on $X$. 

Now set $u_{*} = \xi u$.  Then $|u_{*}| \leq |u|$ and therefore $\|u_{*}\|_{L^{p}(X_{\rho})} \leq \|u\|_{L^{p}(X_{\rho})}$. Let $g_{u}$ be a minimal $p$-weak upper gradient for $u$ on $X_{\rho}$, which is an upper gradient for $u$ on $X_{\rho}$ by Remark \ref{graph capacity}. The product rule for upper gradients implies that
\[
g_{*}:= L|u| + g_{u} \geq L\chi_{X_{\geq 0}}|u| + \xi \cdot g_{u}, 
\]
is an upper gradient of $u_{*}$. It follows that
\[
\|g_{*}\|_{L^{p}(X_{\rho})} \ls \|u\|_{L^{p}(X_{\rho})} + \|g_{u}\|_{L^{p}(X_{\rho})} = \|u\|_{N^{1,p}(X_{\rho})}.  
\]
We conclude from the above that we have $\|u_{*}\|_{N^{1,p}(X_{\rho})} \ls \|u\|_{N^{1,p}(X_{\rho})}$. 

Since $u_{*}|_{X_{\geq 1}} = u|_{X_{\geq 1}}$, it follows immediately from the defining formula \eqref{local trace formula} for the trace that we have $Tu_{*} = Tu$. On the other hand we have by construction that $T_{0}u_{*} \equiv 0$ since $u_{*}(v) = 0$ for all $v \in V_{0}$. It then follows from \eqref{global p trace} applied in the case $n = 0$ that
\[
\|Tu\|_{L^{p}(Z)} \ls \|u_{*}\|_{N^{1,p}(X_{\rho})} \ls  \|u\|_{N^{1,p}(X_{\rho})}.
\]
\end{proof}

Lipschitz functions on $X_{\rho}$ belong to $\t{N}^{1,p}_{\loc}(X_{\rho})$ and have a canonical extension by continuity to $\p X_{\rho} = Z$. We show below that this extension agrees with the trace $T$. The equality $\hat{u}|_{Z} = Tu$ below should be understood as holding for the distinguished $L$-Lipschitz representative of $Tu$ in $L^{p}_{\loc}(Z)$. 

\begin{prop}\label{Lip trace}
Let $u: X_{\rho} \rightarrow \R$ be $L$-Lipschitz and let $\hat{u}: \bar{X}_{\rho} \rightarrow \R$ denote the canonical $L$-Lipschitz extension of $u$ to the completion $\bar{X}_{\rho}$ of $X_{\rho}$. Then $\hat{u}|_{Z} = Tu$. In particular $Tu$ is $L$-Lipschitz. 
\end{prop}

\begin{proof}
We define $T_{n}u$ for each $n \in \Z$ as in \eqref{trace formula}.  Let $z \in Z$ be a given point and let $\gamma_{z}: \R \rightarrow X$ be an ascending geodesic line anchored at $z$ as given by Lemma \ref{second height connection}. Let $\{v_{n}\}_{n \in \Z}$ be the sequence of vertices on $\gamma_{z}$ with $h(v_{n}) = n$. Then $d(\pi(v_{n}),z) < \alpha^{-n}$ for each $n \in \Z$. For each $n \in \Z$ we then have by Lemma \ref{vertex to base},
\[
d_{\rho}(v_{n},z) \leq d_{\rho}(v_{n},\pi(v_{n})) + d(\pi(v_{n}),z) \ls \alpha^{-n}.  
\]
Since $\hat{u}$ is $L$-Lipschitz on $\bar{X}_{\rho}$ it then follows that
\[
|u(v_{n})-\hat{u}(z)| \ls L\alpha^{-n}. 
\]
On the other hand we have
\[
|u(v_{n})-T_{n}u(z)| \leq \sum_{v \in V_{n}} |u(v_{n})-u(v)|\psi_{v}(z)
\]
The terms in the sum on the right are nonzero only when $z \in B(v)$. Since $z \in B(v_{n})$, this implies that $z \in B(v_{n}) \cap B(v)$. Thus there is a horizontal edge $e$ from $v_{n}$ to $v$. It then follows from Lemma \ref{filling estimate both} that $d_{\rho}(v_{n},v) \ls \alpha^{-n}$.  We thus conclude in this case that
\[
|u(v_{n})-u(v)| \leq Ld_{\rho}(v_{n},v) \ls L\alpha^{-n},
\]
which implies that
\[
|u(v_{n})-T_{n}u(z)| \ls \sum_{v \in V_{n}} L\alpha^{-n}\psi_{v}(z) = L\alpha^{-n}. 
\]
Thus for all $n \in \Z$ and $z \in Z$ we have
\[
|T_{n}u(z)-\hat{u}(z)| \ls L\alpha^{-n}. 
\]
Since $T_{n}u \rightarrow Tu$ in $L^{p}(B)$ for each ball $B \subset Z$, it follows in particular that $T_{n}u \rightarrow Tu$ pointwise a.e.~ on $Z$. We thus obtain that for any $z \in Z$ such that $\lim_{n \rightarrow \infty} T_{n}u(z) = Tu(z)$ we in fact have $Tu(z) = \hat{u}(z)$. Thus $Tu$ agrees $\nu$-a.e.~ on $Z$ with the $L$-Lipschitz function $\hat{u}|_{Z}$, as desired. 
\end{proof}

We now estimate the Besov norm of the trace $Tu$ for $u \in \t{D}^{1,p}(X_{\rho}) \subset \t{N}^{1,p}_{\loc}(X_{\rho})$. Throughout the rest of this section we set $\theta = 1-\beta/p$, observing that $0 < \theta < 1$ since $p > \beta > 0$.

\begin{prop}\label{Dirichlet Besov finite}
Let $u \in \t{D}^{1,p}(X_{\rho})$ be given. Then
\[
\|Tu\|_{B^{\theta}_{p}(Z)} \ls \|u\|_{D^{1,p}(X_{\rho})}.
\]
Consequently if $u \in N^{1,p}(X_{\rho})$ then $Tu \in \ch{B}^{\theta}_{p}(Z)$ with the estimate
\[
\|Tu\|_{\ch{B}^{\theta}_{p}(Z)} \ls \|u\|_{N^{1,p}(X_{\rho})}.
\]
\end{prop}

\begin{proof}
We start with $u \in \t{D}^{1,p}(X_{\rho})$ as specified. Since $u \in \t{D}^{1,p}_{\loc}(X_{\rho}) = \t{N}^{1,p}_{\loc}(X_{\rho})$ it follows from Proposition \ref{Lp local trace} that the trace $Tu$ given by formula \eqref{local trace formula} exists and satisfies $Tu \in L^{p}_{\loc}(Z)$. We have the estimate for any $n \in \Z$ and $\nu$-a.e.~ $x,y \in Z$,
\[
|Tu(x)-Tu(y)|^{p} \ls |Tu(x)-T_{n}u(x)|^{p} + |T_{n}u(x)-T_{n}u(y)|^{p} + |Tu(y)-T_{n}u(y)|^{p}.
\]
For $n \in \Z$ and $x \in Z$ we define 
\[
A_{n}(x) = \{y \in Z: \alpha^{-n-1} \leq d(x,y) < \alpha^{-n}\}. 
\]
We then have the following estimate for the Besov norm \eqref{Besov norm} of $Tu$, using the doubling property of $\nu$, 
\begin{align*}
\|Tu\|_{B^{\theta}_{p}(Z)}^{p} &\ls \int_{Z} \sum_{n \in \Z} \int_{A_{n}(x)}\frac{|Tu(x)-T_{n}u(x)|^{p}}{\alpha^{-n \theta p}}\frac{d\nu(y)d\nu(x)}{\nu(B(x,\alpha^{-n}))} \\
&+  \int_{Z} \sum_{n\in \Z} \int_{A_{n}(x)}\frac{|T_{n}u(x)-T_{n}u(y)|^{p}}{\alpha^{-n \theta p}}\frac{d\nu(y)d\nu(x)}{\nu(B(x,\alpha^{-n}))}\\
&+  \int_{Z} \sum_{n \in \Z} \int_{A_{n}(y)}\frac{|Tu(y)-T_{n}u(y)|^{p}}{\alpha^{-n \theta p}}\frac{d\nu(x)d\nu(y)}{\nu(B(y,\alpha^{-n}))}.
\end{align*}
Similarly to the proof of \cite[Theorem 11.1]{BBS21}, we label the three summands on the right sequentially as $(I)$, $(II)$, and $(III)$, and estimate each one separately. Since $(I)$ and $(III)$ are related by switching the roles of $x$ and $y$, it suffices to estimate $(I)$ and $(II)$. 

We begin with $(I)$. Since none of the terms depend on $y$, we can integrate with respect to this variable and use the fact that $\nu(A_{n}(x)) \leq \nu(B(x,\alpha^{-n}))$ (since $A_{n}(x) \subset B(x,\alpha^{-n})$) to obtain 
\[
(I) \ls \int_{Z} \sum_{n\in \Z} \frac{|Tu(x)-T_{n}u(x)|^{p}}{\alpha^{-n \theta p}}d\nu(x) 
\]
We let $g$ be a minimal $p$-weak upper gradient for $u$ on $X_{\rho}$, which is an upper gradient for $u$ by Remark \ref{graph capacity}. By letting $m \rightarrow \infty$ in inequality \eqref{estimate I} we obtain the estimate for a.e.~ $x \in Z$ and any choice of $\la > 0$, 
\begin{equation}\label{estimate I limit}
|Tu(x)-T_{n}u(x)|^{p} \lesssim \alpha^{-\la p n  }\sum_{j=n}^{\infty} \alpha^{(\beta + p\la-p) j}\sum_{v \in V_{j}} \left(\sum_{e \in \mathcal{U}(v)} \int_{e} g^{p} \, d\mu_{\beta}\right)\frac{\psi_{v}(x)}{\nu(B(v))},
\end{equation}
since $T_{n}u \rightarrow Tu$ pointwise a.e. Here the implied constant depends additionally on $\la$. Substituting this estimate into $(I)$ for each $n \in \Z$ gives 
\begin{equation}\label{first (I)}
(I) \ls \int_{Z}\sum_{n \in \Z} \alpha^{(\theta-\la)p n  }\sum_{j=n}^{\infty} \alpha^{(\beta+p\la-p) j}\sum_{v \in V_{j}} \left(\sum_{e \in \mathcal{U}(v)} \int_{e} g^{p} \, d\mu_{\beta}\right)\frac{\psi_{v}(x)}{\nu(B(v))} \, d\nu(x).
\end{equation}
Using the bound \eqref{partition bound}, we conclude from inequality \eqref{first (I)} that
\[
(I) \ls \sum_{n \in \Z} \alpha^{(\theta-\la)p n  }\sum_{j=n}^{\infty} \alpha^{(\beta+p\la-p) j}\sum_{v \in V_{j}} \left(\sum_{e \in \mathcal{U}(v)} \int_{e} g^{p} \, d\mu_{\beta}\right).
\]
Using Tonelli's theorem we can switch the order of summation to obtain
\[
(I) \ls \sum_{j \in \Z} \alpha^{(\beta+p\la-p) j}\sum_{v \in V_{j}} \left(\sum_{e \in \mathcal{U}(v)} \int_{e} g^{p} \, d\mu_{\beta}\right) \sum_{n=-\infty}^{j}  \alpha^{(\theta-\la)p n  }.
\]
We set $\la = \theta/2 = (p-\beta)/2p > 0$. Summing the geometric series on the far right above then gives
\[
(I) \ls \sum_{j \in \Z} \alpha^{(\beta+p\theta-p) j}\sum_{v \in V_{j}} \left(\sum_{e \in \mathcal{U}(v)} \int_{e} g^{p} \, d\mu_{\beta}\right).
\]
Since we assumed that $\theta = (p-\beta)/p$, this simplifies to the desired estimate $(I) \ls \|g\|_{L^{p}(X_{\rho})}$.   

We now estimate $(II)$. For this we observe that 
\begin{align*}
|T_{n}u(x)-T_{n}u(y)|^{p} &\ls \sum_{v \in V_{n}}\sum_{v' \in V_{n}}|u(v)-u(v')|^{p}\psi_{v}^{p}(x)\psi_{v'}^{p}(y) \\
&\leq \sum_{v \in V_{n}}\sum_{v' \in V_{n}}|u(v)-u(v')|^{p}\psi_{v}(x)\psi_{v'}(y)
\end{align*}
using $0 \leq \psi_{v} \leq 1$, since this sum contains a number of terms uniformly bounded in terms of the number of $v \in V_{n}$ such that $\psi_{v}(x) \neq 0$ and $\psi_{v}(y) \neq 0$, which is uniformly bounded in terms of $\alpha$, $\tau$, and $C_{\nu}$ by Lemma \ref{doubling overlap}. Now suppose in addition that $y \in B_{Z}(x,\alpha^{-n})$, which follows from $y \in A_{n}(x)$. If $\psi_{v}(x)\psi_{v'}(y) \neq 0$ for some $v,v' \in V_{n}$ then we must have $y \in B_{Z}(v',2\alpha^{-n})$ by the construction of the Lipschitz partition of unity in Proposition \ref{partition of unity}. Since $d(x,y) < \alpha^{-n}$ it then follows that $x \in B_{Z}(v',3\alpha^{-n}) \subset B(v')$ (since $\tau > 3$). Since we also have $x \in B(v)$, it then follows that $B(v) \cap B(v') \neq \emptyset$. This implies either that $v = v'$ (in which case $|u(v)-u(v')| = 0$) or that there is a horizontal edge in $X$ from $v$ to $v'$. Then, writing $\mathcal{H}(v)$ for the set of all horizontal edges having $v$ as a vertex, using the comparison \eqref{arc length measure},  using the Jensen inequality estimate \eqref{Jensen}, and then using the comparison \eqref{transformed}, we conclude that
\begin{align*}
|T_{n}u(x)-T_{n}u(y)|^{p} &\ls \sum_{v \in V_{n}}\sum_{v' \in V_{n}}\sum_{e \in \mathcal{H}(v)}\left(\int_{e}g\,ds_{\rho}\right)^{p}\psi_{v}(x)\psi_{v'}(y) \\
&= \sum_{v \in V_{n}}\sum_{e \in \mathcal{H}(v)}\left(\int_{e}g\,ds_{\rho}\right)^{p}\psi_{v}(x) \\
&\asymp \alpha^{p(\beta-1)n}\sum_{v \in V_{n}}\sum_{e \in \mathcal{H}(v)}\left(\int_{e}g\,d\mu_{\beta}\right)^{p}\frac{\psi_{v}(x)}{\nu(B(v))^{p}} \\
&\ls \alpha^{p(\beta-1)n}\left(\sum_{v \in V_{n}}\sum_{e \in \mathcal{H}(v)}\mu_{\beta}(e)^{p-1}\int_{e}g^{p}\,d\mu_{\beta}\right)\frac{\psi_{v}(x)}{\nu(B(v))^{p}} \\
&\asymp \alpha^{(\beta-p)n}\left(\sum_{v \in V_{n}}\sum_{e \in \mathcal{H}(v)}\int_{e}g^{p}\,d\mu_{\beta}\right)\frac{\psi_{v}(x)}{\nu(B(v))}.
\end{align*} 
Using this estimate in $(II)$ and using the fact that $p - \beta = p\theta$, we conclude that
\[
(II) \ls \int_{Z}\sum_{n \in \Z}\int_{A_{n}(x)} \sum_{v \in V_{n}}\sum_{e \in \mathcal{H}(v)}\int_{e}g^{p}\,d\mu_{\beta} \frac{\psi_{v}(x)}{\nu(B(v))}\, \frac{d\nu(y)d\nu(x)}{\nu(B(x,\alpha^{-n}))},
\]
which, upon integrating with respect to $y$ followed by $x$ and using $\nu(A_{n}(x)) \leq \nu(B_{Z}(x,\alpha^{-n}))$ and the bound \eqref{partition bound}, gives
\[
(II) \ls \sum_{v \in V_{n}}\sum_{e \in \mathcal{H}(v)}\int_{e}g^{p}\,d\mu_{\beta} \leq \|g\|_{L^{p}(X_{\rho})}.
\]
Since $\|g\|_{L^{p}(X_{\rho})} = \|u\|_{D^{1,p}(X_{\rho})}$, we conclude the desired estimate for $\|Tu\|_{B^{\theta}_{p}(Z)}$. The corresponding estimate for $\|Tu\|_{\ch{B}^{\theta}_{p}(Z)}$ then follows from Proposition \ref{p integrable trace}.
\end{proof}

We conclude from Proposition \ref{Dirichlet Besov finite} that the trace $T: \t{N}^{1,p}_{\loc}(X_{\rho}) \rightarrow L^{p}_{\loc}(Z)$ restricts to bounded linear operators $T: N^{1,p}(X_{\rho}) \rightarrow \ch{B}^{\theta}_{p}(Z)$ and $T: D^{1,p}(X_{\rho}) \rightarrow B^{\theta}_{p}(Z)$ when we set $\theta = 1-\beta/p$. In the next section we will show that these linear operators are surjective. 

%\begin{rem}\label{BP remark}
%Let's consider Proposition \ref{Dirichlet Besov finite} in the context of the Besov spaces originally considered by Bourdon-Pajot in \cite{BP03}. They assume that $Z$ is compact consider the assumption that $\nu$ is \emph{Ahlfors $Q$-regular} for some exponent $Q > 0$, meaning that there is some constant $C_{\nu} \geq 1$ such that $B_{Z}(z,r) \asymp_{C_{\nu}} r^{Q}$ for $0 < r \leq \diam(Z)$; note that this condition clearly implies that $\nu$ is doubling, quantitatively in $C_{\nu}$ and $Q$. 
%\end{rem}

We can now generalize \cite[Proposition 11.2]{BBS21} and \cite[Theorem 11.3]{BBS21} to our setting. The compactness of $Z$ and the uniformization $\bar{X}_{\rho}$ are only used to show that the metric measure spaces $(X_{\rho},d_{\rho},\mu_{\beta})$ and $(\bar{X}_{\rho},d_{\rho},\mu_{\beta})$ are doubling and support a 1-Poincar\'e inequaliy, and to then derive the corresponding estimates of Proposition \ref{Dirichlet Besov finite} in \cite{BBS21}.  Once all of these claims have also been verified in for the case of noncompact $Z$ and $\bar{X}_{\rho}$, the proofs for \cite[Proposition 11.2]{BBS21} and \cite[Theorem 11.3]{BBS21} carry over verbatim to the noncompact setting. Since the proofs are short, we reproduce abbreviated versions of them here for the convenience of the reader.

The first proposition shows that the boundary measure $\nu$ on $Z = \p X_{\rho}$ is absolutely continuous with respect to the $C_{p}^{\bar{X}_{\rho}}$-capacity. We recall that our parameter $\beta$ corresponds to $\beta/\epsilon$ in \cite{BBS21}. By Remark \ref{graph capacity} any subset $G \subset \bar{X}_{\rho}$ with $C_{p}^{\bar{X}_{\rho}}(G) = 0$ must satisfy $G \subset Z$. 

\begin{prop}\label{absolutely continuous capacity}\cite[Proposition 11.2]{BBS21}
Let $G \subset Z$. If $C_{p}^{\bar{X}_{\rho}}(G) = 0$ then $\nu(G) = 0$. 
\end{prop}

\begin{proof}
By Theorem \ref{outer capacity} we can find open sets $U_{n} \subset \bar{X}_{\rho}$ for each $n \in \N$ with $G \subset U_{n}$ and $C_{p}^{\bar{X}_{p}}(U_{n}) < 1/n$. The intersection $G' = \bigcap_{n = 1}^{\infty}U_{n}$ then defines a Borel subset of $\bar{X}_{\rho}$ with $C_{p}^{\bar{X}_{\rho}}(G') = 0$, hence $G' \subset Z$. Let $K \subset G'$ be compact. Since $(\bar{X}_{\rho},d_{\rho},\mu_{\beta})$ supports a $p$-Poincar\'e inequality, by \cite[Theorem 6.7(xi)]{BB11} we can find Lipschitz functions $u_{k}$ on $\bar{X}_{\rho}$ such that $u_{k} = 1$ on $K$ and $\|u_{k}\|_{N^{1,p}(\bar{X}_{\rho})} < 1/k$. By Proposition \ref{Lip trace} we have $Tu_{k} = u_{k}|_{Z}$, hence in particular $Tu_{k} = 1$ on $K$ as well. Thus by Proposition \ref{p integrable trace} we have for each $k$,
\[
\nu(K)^{1/p} \leq \|Tu_{k}\|_{L^{p}(Z)} \ls \|u_{k}\|_{N^{1,p}(\bar{X}_{\rho})} < \frac{1}{k}.
\]
By letting $k \rightarrow \infty$ we conclude that $\nu(K) = 0$. Since $G'$ is a Borel set and $\nu$ is a Borel regular measure on $Z$, we conclude that
\[
\nu(G) \leq \nu(G') = \sup_{K \subset G'} \nu(K) = 0,
\] 
where the supremum is taken over all compact subsets $K$ of $G'$. 
\end{proof}

Since $(X_{\rho},d_{\rho},\mu_{\beta})$ is a uniform geodesic metric measure space that is doubling and satisfies a $p$-Poincar\'e inequality, by work of J. Bj\"orn and Shanmugalingam \cite[Proposition 5.9]{BJS07} there is a bounded linear operator assigning to any $u \in 
N^{1,p}(X_{\rho})$ an extension $\hat{u} \in \t{N}^{1,p}(\bar{X}_{\rho})$. We can then consider the restriction $\hat{u}|_{Z}$ of this extension to $Z$. The next theorem gives a characterization of the trace $Tu \in \ch{B}^{\theta}_{p}(Z)$ for $u \in N^{1,p}(X_{\rho})$ and shows that $\hat{u}|_{Z}$ agrees $\nu$-a.e.~ with $Tu$. 

\begin{prop}\label{trace characterization}\cite[Theorem 11.3]{BBS21}
Let $u \in N^{1,p}(X_{\rho})$. Then $u$ has an extension $\hat{u} \in \t{N}^{1,p}(\bar{X}_{\rho})$. This extension satisfies $\hat{u}|_{Z} = Tu$ $\nu$-a.e.~. Consequently we have the estimates
\[
\|\hat{u}|_{Z}\|_{B^{\theta}_{p}(Z)} \ls \|u\|_{D^{1,p}(X_{\rho})},
\]
and 
\[
\|\hat{u}|_{Z}\|_{L^{p}(Z)} \ls \|u\|_{N^{1,p}(X_{\rho})}.
\]
Moreover, for $C_{p}^{\bar{X}_{\rho}}$-q.e.~ (and thus $\nu$-a.e.) $z \in Z$ we have
\begin{equation}\label{N characterization limit}
\lim_{r \rightarrow 0^{+}} \dashint_{X_{\rho} \cap B_{\rho}(z,r)}|u-\hat{u}(z)|^{p}\,d\mu_{\beta} = 0. 
\end{equation}
\end{prop}

The notation $r \rightarrow 0^{+}$ indicates taking the limit through positive values of $r$. The extension $\hat{u}$ is unique up to sets of zero $p$-capacity in $\bar{X}_{\rho}$: if $\hat{u}_{1}$ and $\hat{u}_{2}$ are two such extensions then $\hat{u}_{1} = \hat{u}_{2}$ $\mu_{\beta}$-a.e.~ on $\bar{X}_{\rho}$ since $\mu_{\beta}(\p X_{\rho}) = 0$, which implies that they are equal q.e.~ on $\bar{X}_{\rho}$ (see the discussion after \eqref{define capacity}). This implies by Proposition \ref{absolutely continuous capacity} that the restriction $\hat{u}|_{Z}$ is unique up to sets of $\nu$-measure zero. In particular the restriction $\hat{u}|_{Z}$ is well-defined as an element of $L^{p}(Z)$. 

For a metric measure space $(Y,d,\mu)$ and an exponent $p \geq 1$ we say that a point $x \in Y$ is an \emph{$L^{p}(Y)$-Lebesgue point} of a measurable function $u: Y \rightarrow [-\infty,\infty]$ if 
\begin{equation}\label{Lebesgue point}
\lim_{r \rightarrow 0^{+}} \dashint_{B_{Y}(x,r)} |u-u(x)|^{p}\,d\mu = 0. 
\end{equation}
If $x$ is an $L^{p}(Y)$-Lebesgue point of $u$ then by H\"older's inequality it is also an $L^{q}(Y)$-Lebesgue point of $u$ for each $1 \leq q \leq p$. Since $\mu_{\beta}(\p X_{\rho}) =  0$,  \eqref{N characterization limit} can be rephrased as saying that $C_{p}^{\bar{X}_{\rho}}$-q.e.~ point of $Z$ is an $L^{p}(\bar{X}_{\rho})$-Lebesgue point for $\hat{u}$. 

\begin{proof}
By \cite[Theorem 4.1 and Corollary 3.9]{S04} we can find a sequence of Lipschitz functions $u_{k} \in \t{N}^{1,p}(\bar{X}_{\rho})$ such that $\|u_{k}-\hat{u}\|_{N^{1,p}(\bar{X}_{\rho})} \rightarrow 0$ and $u_{k}(x) \rightarrow \hat{u}(x)$ for q.e.~ $x \in \bar{X}_{\rho}$ as $k \rightarrow \infty$. Let $\ch{u}_{k} = u_{k}|_{Z}$ and $\ch{u} = \hat{u}|_{Z}$. Then by Proposition \ref{absolutely continuous capacity} we have $\ch{u}_{k} \rightarrow \ch{u}$ $\nu$-a.e.~ on $Z$. 

By Proposition \ref{Lip trace} we have $T(u_{j}-u_{k}) = \ch{u}_{j}-\ch{u}_{k}$ for each $j,k \in \N$. It follows from Proposition \ref{Dirichlet Besov finite} that 
\[
\|\ch{u}_{j}-\ch{u}_{k}\|_{\ch{B}^{\theta}_{p}(Z)} \ls \|u_{j}-u_{k}\|_{N^{1,p}(X_{\rho})} \leq \|u_{j}-u_{k}\|_{N^{1,p}(\bar{X}_{\rho})}
\]
Thus $\{\ch{u}_{k}\}$ defines a Cauchy sequence in $\ch{B}^{\theta}_{p}(Z)$. Since  $\ch{B}^{\theta}_{p}(Z)$ is a Banach space by \cite[Remark 9.8]{BBS21}, we conclude that this sequence converges in $\ch{B}^{\theta}_{p}(Z)$ to a function $u'$. In particular we have $\ch{u}_{k} \rightarrow u'$ in $L^{p}(Z)$, which implies that $\ch{u}_{k} \rightarrow u'$ $\nu$-a.e.~ on $Z$. Since we also have $\ch{u}_{k} \rightarrow \ch{u}$ $\nu$-a.e.~ on $Z$, we conclude that $u' = \ch{u} = \hat{u}|_{Z}$ $\nu$-a.e. Since the trace $T$ defines a bounded linear operator $T: N^{1,p}(X_{\rho}) \rightarrow \ch{B}^{\theta}_{p}(Z)$ it follows that $Tu = \hat{u}|_{Z}$ in $\ch{B}^{\theta}_{p}(Z)$, i.e., $Tu = \hat{u}|_{Z}$ $\nu$-a.e. The estimates for $\|\hat{u}|_{Z}\|_{B^{\theta}_{p}(Z)}$ and  $\|\hat{u}|_{Z}\|_{L^{p}(Z)}$ then follow from the corresponding estimates for $Tu$ since $Tu = \hat{u}|_{Z}$ $\nu$-a.e.

For $p > 1$ the equality \eqref{N characterization limit} follows from the fact that $C_{p}^{Y}$-q.e.~ point of $\bar{X}_{\rho}$ is an $L^{p}(\bar{X}_{\rho})$-Lebesgue point of $\hat{u}$ \cite[Theorem 9.2.8]{HKST}. The same claim also holds for $p = 1$ \cite[Theorem 4.1 and Remark 4.7]{KKST08}; note in our case that $\mu_{\beta}(X_{\rho}) = \infty$ by Proposition \ref{infinite measure}. This proves \eqref{N characterization limit}.
\end{proof}

Proposition \ref{trace characterization} has a natural generalization to functions $u \in \t{D}^{1,p}(X_{\rho})$. 

\begin{prop}\label{Dirichlet characterization}
Let $u \in \t{D}^{1,p}(X_{\rho})$. Then $u$ has an extension $\hat{u} \in \t{D}^{1,p}(\bar{X}_{\rho})$. This extension satisfies $\hat{u}|_{Z} = Tu$ $\nu$-a.e.~. Consequently we have the estimate
\begin{equation}\label{D bound}
\|\hat{u}|_{Z}\|_{B^{\theta}_{p}(Z)} \ls \|u\|_{D^{1,p}(X_{\rho})},
\end{equation}
Moreover, for $C_{p}^{\bar{X}_{\rho}}$-q.e.~ (and thus $\nu$-a.e.) $z \in Z$ we have
\begin{equation}\label{D characterization limit}
\lim_{r \rightarrow 0^{+}} \dashint_{X_{\rho} \cap B_{\rho}(z,r)}|u-\hat{u}(z)|^{p}\,d\mu_{\beta} = 0. 
\end{equation}
\end{prop}

\begin{proof}
Let $u \in \t{D}^{1,p}(X_{\rho})$ be given. We extend $u$ to $\hat{u}: \bar{X}_{\rho} \rightarrow [-\infty,\infty]$ by setting for $z \in Z$,
\[
\hat{u}(z) = \limsup_{r \rightarrow 0^{+}}\dashint_{X_{\rho} \cap B_{\rho}(z,r)}u\,d\mu_{\beta}.
\]
The integrability of $u$ over $X_{\rho} \cap B_{\rho}(z,r)$ follows from the fact that this is a bounded subset of $X_{\rho}$ and that functions $u \in \t{D}^{1,p}(X_{\rho})$ are integrable on balls by Proposition \ref{integrable over balls}. 

We fix $z_{0} \in Z$ and for each $n \in \N$ let $\zeta_{n}:[0,\infty) \rightarrow [0,1]$ be a piecewise linear function such that $\zeta_{n}(t) = 1$ for $0 \leq t \leq n$, $\zeta_{n}(t) = 2-n^{-1}t$ for $n \leq t \leq 2n$,  and $\zeta_{n}(t) = 0$ for $t \geq 2n$. We then define $\kappa_{n}: \bar{X}_{\rho} \rightarrow [0,1]$ by $\kappa_{n}(x) = \zeta_{n}(d_{\rho}(x,z_{0}))$. We note that $\kappa_{n}$ is $n^{-1}$-Lipschitz for each $n$. The rescaled characteristic function $n^{-1}\chi_{B_{\rho}(z_{0},2n)}$ defines an upper gradient for $\kappa_{n}$ on $\bar{X}_{\rho}$. 

For each $n \in \N$ we set $f_{n} = \kappa_{n}u$. The product rule for upper gradients implies that if $g$ is any upper gradient for $u$ on $X_{\rho}$ then for each $n \in \N$,
\begin{equation}\label{cutoff gradient}
g_{n} := n^{-1}\chi_{B_{\rho}(z_{0},2n)} u + g,
\end{equation}
defines an upper gradient of $f_{n}$ on $X_{\rho}$. We claim that if $g$ is $p$-integrable then $g_{n}$ is $p$-integrable as well. For this it suffices to show that the function $n^{-1}\chi_{B_{\rho}(z_{0},2n)} u$ is $p$-integrable on $X_{\rho}$ for each $n \in \Z$. But since for any point $x \in B_{\rho}(z_{0},2n) \cap X_{\rho}$ we have $B_{\rho}(z_{0},2n)  \subset B_{\rho}(x,4n)$, this claim follows from the fact that functions in $\t{D}^{1,p}(X_{\rho})$ are $p$-integrable on balls by Proposition \ref{integrable over balls}. Thus $f_{n} \in \t{N}^{1,p}(X_{\rho})$ for each $n \in \N$. 

We can therefore apply Proposition \ref{trace characterization} to $f_{n}$ for each $n \in \N$. We thus obtain an extension $\hat{f}_{n} \in \t{N}^{1,p}(\bar{X}_{\rho})$ of $f_{n}$ for each $n$ that satisfies for  $C_{p}^{\bar{X}_{\rho}}$-q.e.~ (in particular $\nu$-a.e.) $z \in Z$,
\begin{equation}\label{characterization limit}
\lim_{r \rightarrow 0^{+}} \dashint_{X_{\rho} \cap B_{\rho}(z,r)}|f_{n}-\hat{f}_{n}(z)|^{p}\,d\mu_{\beta} = 0.
\end{equation}
We let $G \subset Z$ be the set of points such that \eqref{characterization limit} holds for each $n \in \N$, which satisfies $C_{p}^{\bar{X}_{\rho}}(Z \backslash G) = 0$ by the countable subadditivity of the capacity. If $z \in G$ then for sufficiently large $n$ we will have from \eqref{characterization limit} and the fact that $f_{n} = u$ on $B_{\rho}(z_{0},n)$, 
\[
\lim_{r \rightarrow 0^{+}} \dashint_{X_{\rho} \cap B_{\rho}(z,r)}|u-\hat{f}_{n}(z)|^{p}\,d\mu_{\beta}  = 0.
\] 
Applying H\"older's inequality then gives
\[
\lim_{r \rightarrow 0^{+}} \dashint_{X_{\rho} \cap B_{\rho}(z,r)}|u-\hat{f}_{n}(z)|\,d\mu_{\beta} = 0,
\]
from which we conclude that $\hat{f}_{n}(z) = \hat{u}(z)$. We conclude in particular that the equality \eqref{D characterization limit} holds. We also conclude that $\lim_{n \rightarrow \infty} \hat{f}_{n}(z) = \hat{u}(z)$ for $C_{p}^{\bar{X}_{\rho}}$-q.e.~ $z \in Z$, hence also for $\nu$-a.e.~ $z \in Z$.

It is clear from the defining formula \eqref{local p trace} for the trace $Tu$ that we have $Tf_{n} = Tu$ $\nu$-a.e.~ on $B_{Z}(z_{0},n/2)$ since $f_{n} = u$ on $B_{\rho}(z_{0},n)$. Thus we also have $\lim_{n \rightarrow \infty} Tf_{n}(z) = Tu(z)$ for $\nu$-a.e.~ $z \in Z$. Since $\hat{f}_{n}|_{Z} = Tf_{n}$ $\nu$-a.e.~ by Proposition \ref{trace characterization}, we conclude that $\hat{u}|_{Z} = Tu$ $\nu$-a.e. The estimate \eqref{D bound} immediately follows from this equality. 

Lastly we need to show that $\hat{u} \in \t{D}^{1,p}(\bar{X}_{\rho})$. Let $g_{u}$ be a minimal $p$-weak upper gradient of $u$ on $X_{\rho}$. For each $n$ we let $\bar{g}_{n}$ denote a minimal $p$-weak upper gradient of $\hat{f}_{n}$ on $\bar{X}_{\rho}$.  We then define a Borel function $\bar{g}: \bar{X}_{\rho} \rightarrow [0,\infty]$ by setting $\bar{g}(x) = \sup_{k \geq n} \bar{g}_{n}(x)$ for $x \in B_{\rho}(z_{0},n/2) \backslash B_{\rho}(z_{0},(n-1)/2) $ and $n \in \N$; here by definition $B_{\rho}(z_{0},0) = \emptyset$.

We claim that $\bar{g}$ is a $p$-integrable $p$-weak upper gradient for $\hat{u}$ on $\bar{X}_{\rho}$. Let $G \subset Z = \p X_{\rho}$ be the set we constructed earlier on which \eqref{characterization limit} holds for each $n \in \N$. We then let $\hat{G} = G \cup X_{\rho}$. Since $C^{\bar{X}_{\rho}}_{p}(\bar{X}_{\rho} \backslash \hat{G}) = C^{\bar{X}_{\rho}}_{p}(Z \backslash G)  = 0$ we have that $\bar{X}_{\rho} \backslash \hat{G}$ is $p$-exceptional, i.e., we have that $p$-a.e.~ curve in $\bar{X}_{\rho}$ belongs entirely to $\hat{G}$.  Now let $x,y \in \hat{G}$ be given and let $\gamma$ be a curve joining them in $\hat{G}$. We can then choose $n$ large enough that $\gamma \subset B_{\rho}(z_{0},n/2)$. Restricted to $B_{\rho}(z_{0},n/2) \cap \hat{G}$ we have $\hat{f}_{n} = \hat{u}$ by the construction of $\hat{f}_{n}$ and $\hat{u}$. Thus since $\gamma \subset  B_{\rho}(z_{0},n/2) \cap \hat{G}$ and $\bar{g}_{n} \leq \bar{g}$ on $B_{\rho}(z_{0},n/2)$, 
\[
|\hat{u}(x)-\hat{u}(y)| =  |\hat{f}_{n}(x)-\hat{f}_{n}(y)| \leq \int_{\gamma} \bar{g}_{n}\,ds  \leq \int_{\gamma} \bar{g} \,ds.
\]
It follows that $\bar{g}$ is a $p$-weak upper gradient for $\hat{u}$ on $\bar{X}_{\rho}$. 

It remains to show that $\bar{g}$ is $p$-integrable on $\bar{X}_{\rho}$. Since $\mu_{\beta}(\p X_{\rho}) = 0$ it suffices to show that $\bar{g}$ is $p$-integrable on $X_{\rho}$. Since $X_{\rho}$ is open in $\bar{X}_{\rho}$, it follows from \cite[Proposition 6.3.22]{HKST} that for each $n \in \N$  the minimal $p$-weak upper gradient $\bar{g}_{n}$ for $\hat{f}_{n}$ on $\bar{X}_{\rho}$ coincides $\mu_{\beta}$-a.e.~ on $X_{\rho}$ with the minimal $p$-weak upper gradient $\ch{g}_{n}$ of $f_{n}$ on $X_{\rho}$. On the open set $B_{\rho}(z_{0},n/2) \cap X_{\rho}$ we have that $f_{n} = u$ by construction. By using \cite[Proposition 6.3.22]{HKST} again it follows that the minimal $p$-weak upper gradient $\ch{g}_{n}$ for $f_{n}$ on $X_{\rho}$  coincides $\mu_{\beta}$-a.e.~ with the minimal $p$-weak upper gradient $g_{u}$ for $u$ on $X_{\rho}$ when each of these upper gradients are restricted to $B_{\rho}(z_{0},n/2) \cap X_{\rho}$. It follows that on $B_{\rho}(z_{0},n/2) \cap X_{\rho}$ we have $\bar{g}_{n} = g_{u}$ $\mu_{\beta}$-a.e. By the defining formula for $\bar{g}$ we conclude that we have $\bar{g} = g_{u}$ $\mu_{\beta}$-a.e.~ on $X_{\rho}$. It follows that $\bar{g}$ is $p$-integrable on $X_{\rho}$ and therefore on $\bar{X}_{\rho}$. 
\end{proof}

%We end this section by remarking that in the case of the Besov spaces considered by Bourdon-Pajot \cite{BP03} the Dirichlet space $D^{1,p}(X_{\rho})$ actually coincides with the Dirichlet space $D^{1,p}(X)$ when we consider $X$ as equipped with the $1$-dimensional Hausdorff measure $\mathcal{L}$ that restricts to the standard Lebesgue measure on each unit length edge of $X$. For this proposition we will assume in addition that $Z$ is unbounded and that $\nu$ is \emph{Ahlfors $Q$-regular} on $Z$ for some $Q > 0$, meaning there is a constant $C_{\nu}$ such that
%\begin{equation}\label{Ahlfors regular}
%\nu(B_{Z}(z,r)) \asymp_{C_{\nu}} r^{Q},
%\end{equation}
%for all $z \in Z$ and $r > 0$. Then $\nu$ is clearly doubling with doubling constant depending only on $C_{\nu}$ and $Q$. In the following proposition and its proof all implied constants will be allowed to additionally depend on $C_{\nu}$ and $Q$. We will be setting $\beta = p-Q$ for $p > Q$, which corresponds to setting $\theta = Q/p$.

%\begin{prop}\label{conformal scale}
%Suppose in addition that $Z$ is Ahlfors $Q$-regular and unbounded. Let $p > Q$ be given and set $\beta = p-Q$. Then for any $p > Q$ and measurable function $u: X \rightarrow [-\infty,\infty]$ we have
%\[
%\|u\|_{D^{1,p}(X_{\rho})} \asymp \|u\|_{D^{1,p}(X)}.
%\]
%Consequently $D^{1,p}(X_{\rho}) = D^{1,p}(X)$. 
%\end{prop}

%\begin{proof}
%Let $u \in D^{1,p}(X_{\rho})$ be given. By a standard calculation (see \cite[(6.3)]{BBS20})
%\end{proof}

\section{Extension theorems}\label{sec:extend}
In this section we establish analogues of the extension theorems in \cite[Section 11]{BBS21} in the case of noncompact $Z$.  We carry over all conventions, notation, and hypotheses from the previous section. Throughout this section we set $\theta = 1-\beta/p$, so that $0 < \theta < 1$ since $p > \beta > 0$ by assumption.

Following Bonk-Saksman \cite{BS18}, for any function $f \in L^{1}_{\loc}(Z)$ we define the \emph{Poisson extension} $Pf: X_{\rho} \rightarrow \R$ by setting for any vertex $v \in V$, 
\[
Pf(v) = \dashint_{B(v)}f\,d\nu,
\]
and then extending $Pf$ to the edges of $X_{\rho}$ by linearly interpolating (with respect to arclength for the metric $d_{\rho}$) the values of $Pf$ on the vertices of each edge. Then $Pf$ defines a continuous function on $X_{\rho}$. We extend $Pf$ to $\bar{X}_{\rho}$ by defining for $z \in Z$, 
\begin{equation}\label{extend extension}
Pf(z) = \limsup_{r \rightarrow 0^{+}}\dashint_{X_{\rho} \cap B_{\rho}(z,r)}Pf\,d\mu_{\beta}. 
\end{equation}
Then $Pf: \bar{X}_{\rho} \rightarrow [-\infty,\infty]$.  Thus $Pf$ defines a linear operator on $L^{1}_{\loc}(Z)$ which takes values in measurable functions on $\bar{X}_{\rho}$. This operator is similar to the one used in \cite[Section 11]{BBS21}, however we use a different notation to avoid conflict with the notation $E$ for the set of edges in $X$. 

A straightforward computation similar to the one done in Proposition \ref{infinite measure} shows for a nonnegative function $f \geq 0$ that we can only have $Pf \in L^{p}(\bar{X}_{\rho})$ for some $p \geq 1$ if $f \equiv 0$ $\nu$-a.e.~ on $Z$. We rectify this issue by defining \emph{truncations} $P_{n}f$ of $Pf$ for each $n \in \Z$ as follows: we let $\xi_{n}: \bar{X}_{\rho} \rightarrow [0,1]$ be the function defined by setting $\xi(x) = 1$ for $x \in X_{\geq n+1}$ and $x \in Z$, $\xi(x) = 0$ for $x \in X_{\leq n}$, and linearly interpolating (with respect to the metric $d_{\rho}$) the values of $\xi_{n}$ on each vertical edge connecting $X_{\leq n}$ to $X_{\geq n+1}$. We then define $P_{n}f = \xi_{n}Pf$. The operators $f \rightarrow P_{n}f$ on $L^{1}_{\loc}(Z)$ are also linear for each $n \in \Z$. We have $P_{n}f|_{X_{\leq n}} = 0$ and $P_{n}f|_{X_{\geq n+1}} = Pf|_{X_{\geq n+1}}$ as well as $P_{n}f|_{Z} = Pf|_{Z}$.

In our first proposition of this section we analyze the $L^p$ norm of $Pf$ on the subsets $X_{\geq n}$ for $n \in \Z$. This estimate does not require $p > \beta$. 

\begin{prop}\label{extension p norm} 
Let $f \in L^{p}(Z)$ be given. Then for each $n \in \Z$, 
\[
\|Pf\|_{L^{p}(X_{\geq n})} \ls \alpha^{-(\beta/p) n}\|f\|_{L^{p}(Z)}. 
\]
Consequently,
\[
\|P_{n}f\|_{L^{p}(X_{\rho})} \ls \alpha^{-(\beta/p) n}\|f\|_{L^{p}(Z)}
\]
\end{prop}

\begin{proof}
If $v$ and $w$ are two distinct vertices of the same edge $e$ then, using Jensen's inequality twice, 
\begin{align*}
\int_{e}|Pf|^{p}\,d\mu_{\beta} &\leq \mu_{\beta}(e)(|Pf(v)|^{p} + |Pf(w)|^{p}) \\
&\leq \mu_{\beta}(e)\left(\dashint_{B(v)}|f|^{p}\,d\nu+\dashint_{B(w)}|f|^{p}\,d\nu\right). 
\end{align*}
Since $v \sim w$, we have $B(v) \cap B(w) \neq \emptyset$ and $r(B(v)) \asymp_{\alpha} r(B(w))$. Thus $B(w) \subset 4\alpha B(v)$. It then follows from the doubling property of $\nu$ and the comparison \eqref{transformed} that 
\[
\int_{e}|Pf|^{p}\,d\mu_{\beta} \ls \mu_{\beta}(e)\dashint_{4\alpha B(v)}|f|^{p}\,d\nu \ls \alpha^{-\beta h(v)}\int_{4\alpha B(v)}|f|^{p}\,d\nu.
\]
As in Proposition \ref{infinite measure} we let $E_{j}$ be the set of all edges in $X$ with at least one vertex in $V_{j}$. The balls $4\alpha B(v)$ for $v$ a vertex of some $e \in E_{j}$ then cover $Z$ with bounded overlap by Lemma \ref{doubling overlap}. Using the doubling of $\nu$ again, it follows from summing over all such vertices $v$ that for any $j \in \Z$, 
\[
\int_{E_{j}}|Pf|^{p}\,d\mu_{\beta} \ls \alpha^{-\beta j}\int_{Z}|f|^{p}\,d\nu.
\]
Summing the geometric series over all integers $j \geq n$ and observing that each edge $e \in X$ belongs to at most two of the sets $E_{n}$ for $n \in \Z$ then gives the first desired estimate. The estimate for $P_{n}f$ follows by observing that $|P_{n}f| \leq |Pf|$ and $P_{n}f|_{X \leq n} = 0$. 
\end{proof}

As a consequence of Proposition \ref{extension p norm} we obtain that $P$ defines a linear operator $P:L^{p}_{\loc}(Z) \rightarrow L^{p}_{\loc}(\bar{X}_{\rho})$ for each $p \geq 1$ since $\mu_{\beta}(\p X_{\rho}) = 0$. We next show that $Pf$ belongs to $\t{D}^{1,p}(\bar{X}_{\rho})$ when $f \in \t{B}^{\theta}_{p}(Z)$, where $\theta = 1-\beta/p$.  We deduce from this that for $f \in \ch{B}^{\theta}_{p}(Z)$  the truncations $P_{n}f$ belong to $\t{N}^{1,p}(\bar{X}_{\rho})$ for each $n \in \Z$.  

\begin{prop}\label{extension Besov}
If $f \in \t{B}^{\theta}_{p}(Z)$  then 
\begin{equation}\label{extension Besov bound}
\|Pf\|_{D^{1,p}(\bar{X}_{\rho})} \ls \|f\|_{B^{\theta}_{p}(Z)}. 
\end{equation}
If $f \in \ch{B}^{\theta}_{p}(Z)$ then we further have for each $n \in \Z$,
\begin{equation}\label{extension truncate Besov}
\|P_{n}f\|_{D^{1,p}(\bar{X}_{\rho})} \ls \alpha^{\theta n}\|f\|_{L^{p}(Z)} + \|f\|_{B^{\theta}_{p}(Z)},
\end{equation}
and 
\begin{equation}\label{extension truncate bound}
\|P_{n}f\|_{N^{1,p}(\bar{X}_{\rho})} \ls (\alpha^{(\theta-1)n} + \alpha^{\theta n})\|f\|_{L^{p}(Z)} + \|f\|_{B^{\theta}_{p}(Z)}.
\end{equation}
\end{prop}

\begin{proof}
For each edge $e$ of $X$ with vertices $v$ and $w$ we set 
\begin{equation}\label{edge gradient}
g_{e} = \frac{|Pf(v)-Pf(w)|}{d_{\rho}(v,w)}.
\end{equation}
The function $g(x) = g_{e}$ for $x \in e$ belonging to the interior of $e$ defines an upper gradient for $Pf$ on the edge $e$ by the construction of $Pf$. We define $g: X_{\rho} \rightarrow \R$ by setting $g(v) = 0$ for each $v \in V$ and setting $g(x) = g_{e}$ for each point $x \in X$ belonging to the interior of an edge $e$. Since the vertices $V$ of $X$ have measure zero with respect to the 1-dimensional Hausdorff measure on $X_{\rho}$, it follows that $g$ defines an upper gradient for $Pf$ on $X_{\rho}$. We note that $g$ is constant on the interior of each edge of $X$. 

As noted in the proof of Proposition \ref{extension p norm}, if $e$ is an edge of $X$ with vertices $v$ and $w$ then $B(w) \subset 4\alpha B(v)$.  Since $|vw| = 1$, we also have from Lemma \ref{filling estimate both}, 
\[
d_{\rho}(v,w) \asymp \alpha^{-(v|w)_{h}} \asymp \alpha^{-h(v)}.
\]
Consequently by the doubling of $\nu$,
\[
g_{e} \ls \alpha^{h(v)} \dashint_{4\alpha B(v)} \dashint_{4\alpha B(v)} |f(x)-f(y)|\,d\nu(y)d\nu(x). 
\]
By Jensen's inequality we then have
\[
g_{e}^{p} \ls \alpha^{p h(v)} \dashint_{4\alpha B(v)} \dashint_{4\alpha B(v)} |f(x)-f(y)|^{p}\,d\nu(y)d\nu(x)
\]
Multiplying and dividing by $\alpha^{-p \theta h(v)}$ and noting that $p - p\theta = \beta$, we conclude that
\[
g_{e}^{p} \ls \alpha^{\beta h(v)} \dashint_{4\alpha B(v)} \dashint_{4\alpha B(v)} \frac{|f(x)-f(y)|^{p}}{\alpha^{-p \theta h(v)}}\,d\nu(y)d\nu(x).
\]
Since $g_{e}$ is constant on the interior of the edge $e$, we thus have from \eqref{transformed} that for any edge $e$ in $X$ and either vertex $v$ on $e$, 
\[
\int_{e} g_{e}^{p}\,d\mu_{\beta} \ls  \int_{4\alpha B(v)} \dashint_{4\alpha B(v)} \frac{|f(x)-f(y)|^{p}}{\alpha^{-p \theta h(v)}}\,d\nu(y)d\nu(x).
\]
Now suppose that $v \in V_{n}$. If $x,y \in 4\alpha B(v)$ then $y \in B_{Z}(x,8\tau \alpha^{-n+1})$. Thus 
\[
\int_{e} g_{e}^{p}\,d\mu_{\beta} \ls  \int_{4\alpha B(v)} \dashint_{B_{Z}(x,8\tau \alpha^{-n+1})} \frac{|f(x)-f(y)|^{p}}{\alpha^{-p \theta n}}\,d\nu(y)d\nu(x).
\]
Summing this inequality over the set $E_{n}$ of edges in $X$ having at least one vertex in $V_{n}$ and using the fact that the balls $4\alpha B(v)$ for $v$ a vertex of some $e \in E_{n}$  cover $Z$ with bounded overlap by Lemma \ref{doubling overlap}, we obtain in a similar manner to what was done in Proposition \ref{extension p norm}, 
\begin{equation}\label{layer bound}
\int_{E_{n}} g_{e}^{p}\,d\mu_{\beta} \ls  \int_{Z} \dashint_{B_{Z}(x,8\tau \alpha^{-n+1})} \frac{|f(x)-f(y)|^{p}}{\alpha^{-p \theta n}}\,d\nu(y)d\nu(x).
\end{equation}
By summing this inequality over $n \in \Z$ and using the fact that each edge $e$ belongs to at most two sets $E_{n}$, we conclude  that
\[
\|g\|_{L^{p}(X_{\rho})}^{p} \ls  \sum_{n \in \Z} \int_{Z} \dashint_{B_{Z}(x,8\tau \alpha^{-n+1})} \frac{|f(x)-f(y)|^{p}}{\alpha^{-p \theta n}}\,d\nu(y)d\nu(x).
\]
Let $m = m(\alpha,\tau) \in \N$ be the minimal integer such that $\alpha^{m} \geq 8\tau \alpha$. Then by the doubling property for $\nu$ we have for each $n \in \Z$,  
\[
\dashint_{B_{Z}(x,8\tau \alpha^{-n+1})} \frac{|f(x)-f(y)|^{p}}{\alpha^{-p \theta n}}\,d\nu(x) \ls \dashint_{B_{Z}(x,\alpha^{-n+m})} \frac{|f(x)-f(y)|^{p}}{a^{p \theta (-n+m)}}\,d\nu(x).
\] 
It then follows by Lemma \ref{Besov estimate} that
\begin{equation}\label{pre extension bound}
\|g\|_{L^{p}(X_{\rho})}^{p} \ls  \sum_{n \in \Z} \int_{Z} \dashint_{B_{Z}(x,\alpha^{-n})} \frac{|f(x)-f(y)|^{p}}{\alpha^{-p \theta n}}\,d\nu(x)d\nu(y) \asymp \|f\|_{B^{\theta}_{p}(Z)}^{p}. 
\end{equation}
We conclude in particular that $Pf \in \t{D}^{1,p}(X_{\rho})$ with $\|Pf\|_{D^{1,p}(X_{\rho})} \ls \|f\|_{B^{\theta}_{p}(Z)}$. 

Let $u \in \t{D}^{1,p}(\bar{X}_{\rho})$ be an extension of $Pf$ to $\bar{X}_{\rho}$ given by Proposition \ref{Dirichlet characterization}. Then, by \eqref{D characterization limit} and the definition of $Pf$ on $Z$, we have that $u = Pf$ $C_{p}^{\bar{X}_{\rho}}$-q.e.~ on $Z$, from which it follows that $u = Pf$ $C_{p}^{\bar{X}_{\rho}}$-q.e.~ on $\bar{X}_{\rho}$. Let $g_{u}$ be a minimal $p$-weak upper gradient for $u$ on $\bar{X}_{\rho}$. Then $g_{u}$ must also be a $p$-weak upper gradient for $Pf$ on $\bar{X}_{\rho}$ since $p$-a.e.~ curve in $\bar{X}_{\rho}$ does not meet the $p$-exceptional set  $\{u(x) \neq Pf(x)\} \subset \bar{X}_{\rho}$. This shows that $Pf \in \t{D}^{1,p}(\bar{X}_{\rho})$. 

Now let $g_{Pf}$ be a minimal $p$-weak upper gradient for $Pf$ on $\bar{X}_{\rho}$. Then $g_{Pf}$ is also a $p$-integrable $p$-weak upper gradient of $Pf$ on $X_{\rho}$. Since $X_{\rho}$ is open in $\bar{X}_{\rho}$ we have by \cite[Proposition 6.3.22]{HKST} that $g_{Pf}$ agrees a.e.~ with a minimal $p$-weak upper gradient for $Pf$ on $X_{\rho}$. Thus $g_{Pf} \leq g$ a.e.~ on $X_{\rho}$, where $g$ is the upper gradient for $Pf$ in the bound \eqref{pre extension bound}. We thus conclude that \eqref{pre extension bound} holds with $g_{Pf}$ in place of $g$ and therefore, since $\mu_{\beta}(\p X_{\rho}) = 0$, 
\[
\|g_{Pf}\|_{L^{p}(\bar{X}_{\rho})} = \|g_{Pf}\|_{L^{p}(X_{\rho})} \ls \|f\|_{B^{\theta}_{p}(Z)}.
\]
This proves the first estimate \eqref{extension Besov bound}.

Let $n \in \Z$ be given and consider the truncated extension $P_{n}f = \xi_{n}Pf$. By Lemma \ref{filling estimate both} if we have vertices $v \in X_{\leq n}$ and $w \in X_{\geq n+1}$ connected by a vertical edge then $d_{\rho}(v,w) \asymp \alpha^{-n}$. Thus $\xi_{n}$ is locally $L\alpha^{n}$-Lipschitz on $\bar{X}_{\rho}$ for some constant $L = L(\alpha,\tau) \geq 1$, which implies that it is $L\alpha^{n}$-Lipschitz on $\bar{X}_{\rho}$ since $\bar{X}_{\rho}$ is geodesic. Since $\xi_{n}|_{X_{\leq n}} = 0$ and $\xi|_{X_{\geq n+1} \cup Z} = 1$, we conclude that $L\alpha^{n}\chi_{X\geq n \cup Z}$ defines an upper gradient for $\xi_{n}$ on $\bar{X}_{\rho}$. Let $g_{Pf}$ be a minimal $p$-weak upper gradient for $Pf$ on $\bar{X}_{\rho}$. By the product rule for upper gradients we then conclude from  $|\xi_{n}| \leq 1$ that we have that
\[
g_{n}:= L\alpha^{n}\chi_{X_{\geq n} \cup Z} |Pf| + g_{Pf},
\]
is an upper gradient for $P_{n}f$ on $\bar{X}_{\rho}$. Since $\mu_{\beta}(Z) = 0$, we conclude from Proposition \ref{extension p norm} and the bound \eqref{extension Besov bound} that
\begin{align*}
\|g_{n}\|_{L^{p}(\bar{X}_{\rho})} &\ls \alpha^{n}\|Pf\|_{L^{p}(X_{\geq n})} + \|g_{Pf}\|_{L^{p}(\bar{X}_{\rho})} \\
&\ls \alpha^{n}\|Pf\|_{L^{p}(X_{\geq n})} + \|Pf\|_{D^{1,p}(\bar{X}_{\rho})} \\
&\ls \alpha^{\theta n}\|f\|_{L^{p}(Z)} + \|f\|_{B^{\theta}_{p}(Z)},
\end{align*}
since $1-\beta/p= \theta$. The bound \eqref{extension truncate Besov} follows. The bound \eqref{extension truncate bound} is then a consequence of Proposition \ref{extension p norm}. 
\end{proof}

We have thus defined a bounded linear operator $P:\t{B}^{\theta}_{p}(Z) \rightarrow D^{1,p}(\bar{X}_{\rho})$ as well as bounded linear operators $P_{n}: \ch{B}^{\theta}_{p}(Z) \rightarrow \t{N}^{1,p}(\bar{X}_{\rho})$ for each $n \in \Z$. We will relate these operators to the trace operator $T$ of the previous section using Proposition \ref{q Lebesgue points} below. This  proposition shows, for each $q \geq 1$, that $L^{q}(Z)$-Lebesgue points for $f \in L^{1}_{\loc}(Z)$ are $L^{q}(\bar{X}_{\rho})$-Lebesgue points for $Pf$. 

\begin{prop}\label{q Lebesgue points}
Let $f \in L^{1}_{\loc}(Z)$. Let $z \in Z$ be an $L^{q}(Z)$-Lebesgue point for $f$ for a given $q \geq 1$. Then $z$ is an $L^{q}(\bar{X}_{\rho})$-Lebesgue point for $Pf$ and $Pf(z) = f(z)$. Consequently the same is true with $P_{n}f$ replacing $Pf$ for each $n \in \Z$. 
\end{prop}
\begin{proof}
Our proof is modeled on the final part of the proof of \cite[Theorem 12.1]{BBS21}. We will establish a more general inequality \eqref{ball integral estimate} in the process of proving the proposition that will be needed for the proof of Proposition \ref{extension hyperbolic Poincare} later on. Throughout the proof all implied constants will be allowed to depend additionally on the exponent $q$.

We let $z \in Z$ be an arbitrary point, fix a radius $r > 0$, let $N = N(r)$ denote the maximal integer such that $\alpha^{-N} \geq r$, and consider all $x \in X$ such that $d_{\rho}(x,z) < \alpha^{-N}$. If $x$ belongs to an edge $e$ with vertices $v$ and $w$ then either $v$ or $w$ must belong to $B_{\rho}(z,\alpha^{-N})$ since the metric on $\bar{X}_{\rho}$ is geodesic. Let $v$ be the vertex belonging to $B_{\rho}(z,\alpha^{-N})$. Then
\begin{equation}\label{height bound}
\alpha^{-h(v)} \asymp d_{\rho}(v) \leq d_{\rho}(v,z) < \alpha^{-N},
\end{equation}
so that we have $\alpha^{-h(v)} \ls \alpha^{-N}$. Then by Lemma \ref{vertex to base} we have
\[
d_{\rho}(\pi(v),z) \leq d_{\rho}(\pi(v),v) + d_{\rho}(v,z) \ls \alpha^{-h(v)} + \alpha^{-N} \ls \alpha^{-N}. 
\]
A similar estimate shows that the other vertex $w$ on $e$ also satisfies $d(\pi(w),z) \ls a^{N}$. 

Let $a \in \R$ be arbitrary. Since $Pf(x)$ is a convex combination of $Pf(v)$ and $Pf(w)$, we have by Jensen's inequality, 
\[
\int_{e}|Pf-a|^{q}\,d\mu_{\beta} \leq (|Pf(v)-a|^{q} + |Pf(w)-a|^{q})\mu_{\beta}(e).
\]
Using Jensen again gives 
\[
|Pf(v)-a|^{q} \leq \dashint_{B(v)}|f-a|^{q}\,d\nu,
\]
and the same with $w$ replacing $v$. Combining this with the comparison \eqref{transformed} and using the fact that $B(w) \subset 4\alpha B(v)$ gives
\[
\int_{e}|Pf-a|^{q}\,d\mu_{\beta} \ls \alpha^{-\beta h(v)}\int_{4\alpha B(v)}|f-a|^{q}\,d\nu. 
\]
Summing this over all edges $e$ such that at least one vertex of $e$ belongs to $B_{\rho}(z,\alpha^{-N})$ and then using the fact that $X$ has bounded degree and that $h(v) \geq N-c$ for some constant $c = c(\alpha,\tau)$ (by \eqref{height bound}),  we conclude that
\[
\int_{B_{\rho}(z,\alpha^{-N})}|Pf-a|^{q}\,d\mu_{\beta} \ls \sum_{n \geq N-c}\; \sum_{v \in V_{n} \cap B_{\rho}(z,\alpha^{-N})} \alpha^{-\beta n}\int_{4\alpha B(v)}|f-a|^{q}\,d\nu.
\]
Let $B = B_{Z}(z,\alpha^{-N})$ and let $C = C(\alpha,\tau) \geq 1$ be chosen large enough that Lemma \ref{hull approximation} implies that $B_{\rho}(z,\alpha^{-N}) \subset H^{CB}$. Then for a given $n \geq N-c$, the balls $4\alpha B(v)$ for $v \in H^{CB}_{n}$ have bounded overlap by Lemma \ref{doubling overlap} and will be contained in $C' B$ for a constant $C' = C'(\alpha,\tau)$. Thus
\begin{align*}
\int_{B_{\rho}(z,\alpha^{-N})}|Pf-a|^{q}\,d\mu_{\beta} &\ls \sum_{n \geq N-c} \alpha^{-\beta n}\int_{C'B}|f-a|^{q}\,d\nu \\
&\asymp \alpha^{-\beta N}\int_{C'B}|f-a|^{q}\,d\nu.
\end{align*}
By combining Lemma \ref{hull measure} and the doubling property of $\nu$ we conclude that
\[
\mu_{\beta}(B_{\rho}(z,\alpha^{-N})) \asymp \alpha^{-\beta N} \nu(B_{Z}(z,\alpha^{-N})) \asymp \alpha^{-\beta N} \nu(B_{Z}(z,C'\alpha^{-N})).
\]
Thus, using the fact that $r \asymp \alpha^{-N}$ and the doubling property of $\nu$ and $\mu_{\beta}$, we conclude for each $r > 0$ and $a \in \R$ that 
\begin{equation}\label{ball integral estimate}
\dashint_{B_{\rho}(z,r)}|Pf-a|^{q}\,d\mu_{\beta} \ls \dashint_{B_{Z}(z,C'r)}|f-a|^{q}\,d\nu,
\end{equation}
with the implied constant being independent of $r$ and $a$. 

Now assume that $z$ is an $L^{q}(Z)$-Lebesgue point for $f$ and set $a = f(z)$ in \eqref{ball integral estimate}. Then the right side of \eqref{ball integral estimate} converges to $0$ as $r \rightarrow 0$. We can then conclude from \eqref{ball integral estimate} that $z$ is also an $L^{q}(\bar{X}_{\rho})$-Lebesgue point for $Pf$ and that $Pf(z) = f(z)$ (by the definition \eqref{extend extension} of $Pf(z)$). The conclusions for $P_{n}f$ for each $n \in \Z$ then follow from the fact that $P_{n}f = Pf$ on $X_{\geq n+1} \cup Z$, which implies that the $L^{q}(\bar{X}_{\rho})$-Lebesgue points for $P_{n}f$ and $Pf$ on $Z$ are the same. 
\end{proof}

When specialized to Lipschitz functions, Propositions \ref{extension Besov} and \ref{q Lebesgue points} show that the extension $Pf$ is also Lipschitz and restricts to $f$ on $Z$. 

\begin{prop}\label{Lip extend}
Let $f: Z \rightarrow \R$ be $L$-Lipschitz. Then there is a constant $C = C(\alpha,\tau) \geq 1$ such that $Pf: \bar{X}_{\rho} \rightarrow \R$ is $CL$-Lipschitz. Furthermore we have $Pf|_{Z} = f$. 
\end{prop}

\begin{proof}
For any vertex $v \in V$ we have
\[
|Pf(v)-f(\pi(v))| \leq \dashint_{B(v)}|f-f(\pi(v))|\,d\nu \leq L \tau \alpha^{-h(v)}. 
\]
Let $e$ be an edge of $X$ with vertices $v$ and $w$. Then $B(v) \cap B(w) \neq \emptyset$, so we must have $d(\pi(v),\pi(w)) \leq 2\tau \alpha^{-h(v)+1}$ since $|h(v)-h(w)| \leq 1$.  We then conclude that
\begin{align*}
|Pf(v)-Pf(w)| &\leq |Pf(v)-f(\pi(v))| + |f(\pi(v))-f(\pi(w))| + |f(\pi(w))-Pf(w)|  \\
&\leq 4\tau L \alpha^{-h(v)+1}. 
\end{align*}
We have $d_{\rho}(v,w) \asymp \alpha^{-h(v)}$ by Lemma \ref{filling estimate both}, with comparison constant depending only on $\alpha$ and $\tau$. Thus if we define $g_{e}$ as in \eqref{edge gradient} then we conclude that $g_{e} \leq CL$ with $C = C(\alpha,\tau) \geq 1$. We conclude that $Pf$ is $CL$-Lipschitz in the metric $d_{\rho}$ on each edge of $X_{\rho}$. Since $X_{\rho}$ is geodesic it follows that $Pf$ is $CL$-Lipschitz on $X_{\rho}$. 

Since $f$ is Lipschitz we have that every point of $Z$ is an $L^{1}(Z)$-Lebesgue point for $f$. Proposition \ref{q Lebesgue points} then implies that every point of $Z$ is an $L^{1}(\bar{X}_{\rho})$-Lebesgue point for $Pf$. Fix a point $z \in Z$ and let $\gamma: [0,\infty) \rightarrow X$ be an ascending vertical geodesic ray anchored at $z$ as constructed in Lemma \ref{second height connection} with vertices $v_{n} = \gamma(n)$ on $\gamma$ for each $n \geq 0$ satisfying $h(v_{n}) = n$. Let $r > 0$ be given and choose $n$ large enough that $v_{n} \in B_{\rho}(z,r)$ and $\alpha^{-n} \leq r$. Since $Pf$ is $CL$-Lipschitz on $X_{\rho}$ and since $z \in B(v_{n})$,  we have
\begin{align*}
\dashint_{B_{\rho}(z,r)}|Pf-f(z)| \, d\mu_{\beta} &\leq 2CLr + |Pf(v_{n})-f(z)| \\
&\leq 2CLr + \dashint_{B(v_{n})}|f-f(z)|\, d\nu \\
&\leq 2CLr + CL\tau \alpha^{-n} \\
&\leq (2+\tau)CLr.
\end{align*}
We conclude that 
\[
\lim_{r \rightarrow 0^{+}} \dashint_{B_{\rho}(z,r)}|Pf-f(z)| \, d\mu_{\beta} = 0. 
\]
Since $z$ is an $L^{1}(\bar{X}_{\rho})$-Lebesgue point for $Pf$, it follows that $Pf(z) = f(z)$. Since this holds for any $z \in Z$ we conclude that $Pf|_{Z} = f$ and that every point of $\bar{X}_{\rho}$ is an $L^{1}(\bar{X}_{\rho})$-Lebesgue point for $Pf$. 

Let $\hat{u}$ denote the unique $CL$-Lipschitz extension of $Pf|_{X_{\rho}}$ to $\bar{X}_{\rho}$. Then every point of $\bar{X}_{\rho}$ is also an $L^{1}(\bar{X}_{\rho})$-Lebesgue point for $\hat{u}$. Since $\mu_{\beta}(Z) = 0$, we have $\hat{u} = Pf$ $\mu_{\beta}$-a.e.~ on $\bar{X}_{\rho}$. Since every point of $\bar{X}_{\rho}$ is an $L^{1}(\bar{X}_{\rho})$-Lebesgue point for both $\hat{u}$ and $Pf$, we conclude that $\hat{u} = Pf$. This implies in particular that $Pf$ is $CL$-Lipschitz on $\bar{X}_{\rho}$, as desired. 
\end{proof}

\begin{rem}\label{remark Lip extend}
For Lipschitz functions $f: Z \rightarrow \R$ one can define a simpler Lipschitz extension $\t{f}: \bar{X}_{\rho} \rightarrow \R$ that does not make use of the measure $\nu$ by setting $\t{f}(v) = f(\pi(v))$ for each vertex $v \in V$ and linearly interpolating the values of $\t{f}$ on the edges between vertices (with respect to the metric $d_{\rho}$). This extension will have the same properties as the extension $Pf$ of $f$ used in Proposition \ref{Lip extend}.
\end{rem}

We can now relate the trace and extension operators. 

\begin{prop}\label{extend to trace}
Let $f \in \t{B}_{p}^{\theta}(Z)$. Then $T(Pf) = f$ $\nu$-a.e. Consequently the induced trace operators $T:D^{1,p}(X_{\rho}) \rightarrow B^{\theta}_{p}(Z)$ and $T: N^{1,p}(X_{\rho}) \rightarrow \ch{B}^{\theta}_{p}(Z)$ are surjective. 
\end{prop}

\begin{proof}
Let $f \in \t{B}_{p}^{\theta}(Z)$ be given. Since $f \in L^{1}_{\loc}(Z)$ we have that $\nu$-a.e.~ point of $Z$ is an $L^{1}(Z)$-Lebesgue point of $f$ by the Lebesgue differentiation theorem \cite[Theorem 1.8]{Hein01}. By Proposition \ref{extension Besov} we have $Pf \in \t{D}^{1,p}(\bar{X}_{\rho})$, and by Proposition \ref{q Lebesgue points} we have that each $L^{1}(Z)$-Lebesgue point $z\in Z$ for $f$ is an $L^{1}(\bar{X}_{\rho})$-Lebesgue point for $Pf$ that satisfies $Pf(z) = f(z)$. We write $u = Pf|_{X_{\rho}}$ and let $\hat{u}$ denote the extension of $u$ to $\bar{X}_{\rho}$ given by Proposition \ref{Dirichlet characterization}. Then $\hat{u}|_{Z} = Tu$ $\nu$-a.e. On the other hand we have by \eqref{D characterization limit} and H\"older's inequality that each $L^{1}(\bar{X}_{\rho})$-Lebesgue point of $Pf$ on $Z$ (and therefore each $L^{1}(Z)$-Lebesgue point of $Z$) satisfies $Pf(z) = \hat{u}(z)$. We conclude that $Pf|_{Z} = \hat{u}|_{Z} = T(Pf)$ $\nu$-a.e. Since $Pf|_{Z} = f$ $\nu$-a.e., we conclude that $T(Pf) = f$ $\nu$-a.e., as desired. 

Let $f \in \t{B}_{p}^{\theta}(Z)$ be arbitrary. Then $u = Pf|_{X_{\rho}}$ defines an element of $\t{D}^{1,p}(X_{\rho})$ such that $Tu = f$ $\nu$-a.e. on $Z$ and therefore $Tu = f$ in $B_{p}^{\theta}(Z)$. We conclude that the trace operator $T:D^{1,p}(X_{\rho}) \rightarrow B^{\theta}_{p}(Z)$ is surjective. 

Now let $f \in \ch{B}^{\theta}_{p}(Z)$ be arbitrary. We set $u = P_{0}f|_{X_{\rho}}$. Then $u$ defines an element of $N^{1,p}(X_{\rho})$ by Proposition \ref{extension Besov}. Since $u|_{X \geq 1} = Pf|_{X\geq 1}$ we conclude that $Tu = f$ $\nu$-a.e.~ on $Z$, which implies that $Tu = f$ in $\ch{B}^{\theta}_{p}(Z)$. Thus the trace operator $T: N^{1,p}(X_{\rho}) \rightarrow \ch{B}^{\theta}_{p}(Z)$ is also surjective.
\end{proof}

Proposition \ref{extend to trace} completes the proof of Theorem \ref{thm:extendtrace}. The representative $Pf|_{Z}$ of $f$ in $\t{B}^{\theta}_{p}(Z)$ constructed in the proof of Proposition \ref{extend to trace} is better behaved than the original function $f$ in many respects. We elaborate on this in Section \ref{sec:properties}. Proposition \ref{extend to trace} shows that $B^{\theta}_{p}(Z)$ is the continuous image under the trace operator $T$ of the Banach space $D^{1,p}(X_{\rho})$; this allows us to immediately deduce as a corollary that the homogeneous Besov space $B^{\theta}_{p}(Z) = \t{B}^{\theta}_{p}(Z)/\sim$ is a Banach space. We recall that we showed that $D^{1,p}(X_{\rho})$ is a Banach space in Proposition \ref{Dirichlet Banach}.

\begin{cor}\label{Besov Banach}
$B^{\theta}_{p}(Z)$ is a Banach space for $p \geq 1$ and $0 < \theta < 1$. 
\end{cor}

%\begin{proof}
%Let $\{f_{n}\}$ be a Cauchy sequence in $B^{\theta}_{p}(Z)$, and let $\{\t{f}_{n}\}$ be a sequence of representatives of these functions in $\t{B}^{\theta}_{p}(Z)$. The sequence of functions $\{P\t{f}_{n}\}$ in $\t{D}^{1,p}(\bar{X}_{\rho})$ then defines a Cauchy sequence in $D^{1,p}(\bar{X}_{\rho})$ by Proposition \ref{extension Besov} and the linearity of the extension operator $P$. Since $D^{1,p}(\bar{X}_{\rho})$ is a Banach space by Proposition \ref{Dirichlet Banach}, we conclude that there is a function $u \in \t{D}^{1,p}(\bar{X}_{\rho})$ such that $\|P\t{f}_{n}-u\|_{D^{1,p}(\bar{X}_{\rho})} \rightarrow 0$. Then $u$ has a trace $Tu \in \t{B}^{\theta}_{p}(Z)$ by Proposition \ref{Dirichlet Besov finite}.  By Proposition \ref{Dirichlet Besov finite} we also conclude that 
%\[
%\|T(P\t{f}_{n})-Tu\|_{B^{\theta}_{p}(Z)} = \|T(P\t{f}_{n}-u)\|_{B^{\theta}_{p}(Z)} \rightarrow 0.
%\]
%Since $T(P\t{f}_{n}) = \t{f}_{n}$ $\nu$-a.e.~ on $Z$ by Proposition \ref{extend to trace}, this implies that $\|\t{f}_{n}-Tu\|_{B^{\theta}_{p}(Z)} \rightarrow 0$. Letting $f$ denote the projection of $Tu$ to $B^{\theta}_{p}(Z)$, we conclude that $f_{n} \rightarrow f$ in $B^{\theta}_{p}(Z)$. It follows that $B^{\theta}_{p}(Z)$ is a Banach space. 
%\end{proof}

By combining Proposition \ref{extend to trace} with Proposition \ref{Poincare type inequality} and inequality \eqref{ball integral estimate}, we obtain the following useful Poincar\'e type inequality for functions defined on balls in $Z$ such that  This inequality will later be used in the proof of Corollary \ref{cor:quasicontinuous} in Section \ref{sec:properties} to show that Besov functions have $L^{p}(Z)$-Lebesgue points quasieverywhere with respect to the Besov capacity.

\begin{prop}\label{extension hyperbolic Poincare}
There is a constant $C = C(\alpha,\tau) \geq 1$ such that if $B \subset Z$ is any ball in $Z$ of radius $r > 0$ and $f \in \t{B}^{\theta}_{p}(Z)$ is given then for any $p$-integrable upper gradient $g$ of the extension $Pf$ on $C\hat{B}$ we have 
\begin{equation}\label{extension equation type}
\dashint_{B} |f-Pf_{\hat{B}}|^{p} \,d\nu \ls r^{p}\dashint_{C\hat{B}}g^{p}\,d\mu_{\beta}.
\end{equation}
\end{prop}

\begin{proof}
By Jensen's inequality we have
\[
\dashint_{B} |f-Pf_{\hat{B}}|^{p} \,d\nu \leq \dashint_{B}\dashint_{\hat{B}}|f(z)-Pf(x)|^{p} \,d\mu_{\beta}(x) \, d\nu(z). 
\]
We can then apply inequality \eqref{ball integral estimate} with $a = f(z)$ to obtain for a constant $C' = C'(\alpha,\tau)$, 
\begin{align*}
\dashint_{B}\dashint_{\hat{B}}|f(z)-Pf(x)|^{p} \,d\mu_{\beta}(x) \, d\nu(z) &\ls \dashint_{B}\dashint_{C'B}|f(z)-f(y)|^{p} \,d\nu(y) \, d\nu(z) \\
&\ls \dashint_{C'B}\dashint_{C'B}|f(z)-f(y)|^{p} \,d\nu(y) \, d\nu(z),
\end{align*}
with the second inequality following from the doubling property for $\nu$. By using the inequality
\[
|f(z)-f(y)|^{p} \ls |f(z)-f_{C'B}|^{p} + |f(y)-f_{C'B}|^{p},
\]
we can then conclude that
\[
\dashint_{B} |f-Pf_{\hat{B}}|^{p} \,d\nu \ls \dashint_{C'B} |f-f_{C'B}|^{p}\,d\nu.
\]
Since $T(Pf) = f$ a.e.~ by Proposition \ref{extend to trace}, we can apply Proposition \ref{Poincare type inequality} to conclude that for any upper gradient $g$ of $Pf$ on $H^{C'B}$, 
\begin{equation}\label{proto extension type}
\dashint_{B} |f-Pf_{\hat{B}}|^{p} \ls r^{p} \dashint_{H^{C'B}} g^{p}\, d\mu_{\beta}. 
\end{equation}
By using Lemma \ref{hull approximation} we can find a constant $C = C(\alpha,\tau)$ such that $C^{-1}\hat{B} \subset H^{C'B} \subset C\hat{B}$. The desired inequality \eqref{extension equation type} then follows from \eqref{proto extension type} and the doubling property of $\nu$. 
\end{proof}

When $\nu$ satisfies a reverse-doubling condition (see \eqref{reverse doubling}) the inequality \eqref{extension equation type} can be improved. See Proposition \ref{hyperbolic Poincare} at the end of the paper.

We can also use Proposition \ref{extend to trace} to obtain finer information about the sequence of functions $\{T_{n}(Pf)\}_{n \in \Z}$ built from the Poisson extension $Pf$ of a given Besov function $f$ that were used to construct the trace operator in Section \ref{sec:trace}. We will show that these functions actually converge to the trace $T(Pf) = f$ in the Besov norm. The following proposition is inspired by \cite[Theorem 3.3]{BSS18} and \cite[Theorem 3.2]{S16}, and in an appropriate sense is contained within the statement of \cite[Theorem 3.2]{S16} as the functions $T_{n}(Pf)$ we consider are essentially identical to those considered in the referenced results. We note however that our proof below of the convergence of $T_{n}(Pf)$ to $f$ in the Besov norm is quite different than the one given in \cite[Theorem 3.2]{S16}; in particular we do not make use of the Fefferman-Stein maximal theorem. 

\begin{prop}\label{converge Besov}
Let $f \in B^{\theta}_{p}(Z)$ be given. Then there is an $n_{0} = n_{0}(\alpha,\tau)$ such that for each $n \in \Z$ the upper gradient $g$ of $Pf$ on $X_{\rho}$ defined by \eqref{edge gradient} satisfies 
\begin{equation}\label{cutoff converge Besov}
\|T_{n}(Pf)-f\|_{B^{\theta}_{p}(Z)} \ls \|g\|_{L^{p}(X_{\geq n-n_{0}})}. 
\end{equation}
Consequently we have the norm estimate
\begin{equation}\label{estimate partial trace}
\|T_{n}(Pf)\|_{B^{\theta}_{p}(Z)} \ls \|f\|_{B^{\theta}_{p}(Z)},
\end{equation}
and we have $T_{n}(Pf) \rightarrow f$ in $B^{\theta}_{p}(Z)$ as $n \rightarrow \infty$. Furthermore if $f \in \ch{B}^{\theta}_{p}(Z)$ then $T_{n}(Pf) \rightarrow f$ in $\ch{B}^{\theta}_{p}(Z)$ as well. 
\end{prop}

\begin{proof}
Let $f \in B^{\theta}_{p}(Z)$ and $n \in \Z$ be given. We let $g$ be the upper gradient of $Pf$ on $X_{\rho}$ defined by the formula \eqref{edge gradient}. We define $u_{n}: X_{\rho} \rightarrow \R$ by $u_{n} = P(T_{n}(Pf))$. The values of $Pf-u_{n}$ on each edge of $X$ are given by linearly interpolating (with respect to the metric $d_{\rho}$) the values of $Pf-u_{n}$ on the corresponding vertices of that edge, since the values of $u_{n}$ and $Pf$ on this edge are each determined by this same linear interpolation procedure. Thus an upper gradient of $Pf-u_{n}$ on an edge $e$ with vertices $v$ and $w$ is given by 
\[
g_{e,n} = \frac{|u_{n}(v)-Pf(v)-u_{n}(w)+Pf(w)|}{d_{\rho}(v,w)}.
\]
As in the proof of Proposition \ref{extension Besov}, we can then define an upper gradient $g_{n}$ for $Pf-u_{n}$ on $X$ by setting $g_{n}(x) = g_{e,n}$ whenever $x$ belongs to the interior of an edge $e$ in $X$ and setting $g_{n}(v) = 0$ for each vertex $v \in V$. Letting $v \in V$ denoting either vertex on an edge $e$ (and letting $w$ denote the other vertex), we have $d_{\rho}(v,w) \asymp \alpha^{-h(v)}$ by Lemma \ref{filling estimate both} and we have $\mu_{\beta}(e) \asymp \alpha^{-\beta h(v)}\nu(B(v))$ by \eqref{transformed}. Thus
\begin{equation}\label{converging upper gradient}
\int_{e}g_{n}^{p}\,d\mu_{\beta} \asymp \frac{\nu(B(v))}{\alpha^{(\beta-p)h(v)}}|u_{n}(v)-Pf(v)-u_{n}(w)+Pf(w)|^{p}.
\end{equation}

Let's first consider the case that $\max\{h(v),h(w)\} \leq n$, i.e., that both of the vertices on $e$ have height at most $n$. Then 
\[
\alpha^{-n} \leq \min\{r(B(v)),r(B(w))\} = \tau \min\{\alpha^{-h(v)},\alpha^{-h(w)}\}.
\]
Hence we can apply Proposition \ref{Lp ball trace} to obtain that 
\[
\|T(Pf)-T_{n}(Pf)\|_{L^{p}(B(v))}^{p} \ls \alpha^{(\beta-p)n}\|g\|_{L^{p}(H^{B(v)}_{\geq n})}^{p}.
\]
Since $T(Pf) = f$ $\nu$-a.e.~ by Proposition \ref{extend to trace}, this simplifies to
\[
\|f-T_{n}(Pf)\|_{L^{p}(B(v))}^{p} \ls \alpha^{(\beta-p)n}\|g\|_{L^{p}(H^{B(v)}_{\geq n})}^{p}.
\]
Thus by Jensen's inequality,
\begin{align*}
|Pf(v)-P(T_{n}(Pf))(v)|^{p} &\leq \dashint_{B(v)}|f-T_{n}(Pf)|^{p}d\nu \\
&\ls  \frac{\alpha^{(\beta-p)n}}{\nu(B(v))}\|g\|_{L^{p}(H^{B(v)}_{\geq n})}^{p}.
\end{align*}
Furthermore the same inequality holds with $w$ replacing $v$. Using the estimate 
\[
|u_{n}(v)-Pf(v)-u_{n}(w)+Pf(w)|^{p} \ls |u_{n}(v)-Pf(v)|^{p} + |u_{n}(w)-Pf(w)|^{p},
\]
we conclude from \eqref{converging upper gradient} that
\[
\int_{e}g_{n}^{p}\,d\mu_{\beta} \ls \alpha^{(\beta-p)(n-h(v))}\left(\|g\|_{L^{p}(H^{B(v)}_{\geq n})}^{p} + \|g\|_{L^{p}(H^{B(w)}_{\geq n})}^{p}\right). 
\]
Summing this estimate over all edges $e$ in $X_{\leq n}$ and using the bounded overlap property of the hulls $H^{B(v)}$ of vertices $v \in V_{k}$ of a given height from Lemma \ref{hull overlap}, we obtain 
\begin{align*}
\|g_{n}\|_{L^{p}(X_{\leq n})}^{p} &\ls \sum_{k=-\infty}^{n}\alpha^{(\beta-p)(n-k)} \sum_{v \in V_{k}}\|g\|_{L^{p}(H^{B(v)}_{\geq n})}^{p} \\
&\ls \sum_{k=-\infty}^{n}\alpha^{(\beta-p)(n-k)}\|g\|_{L^{p}(X_{\geq n})}^{p} \\
&\ls \|g\|_{L^{p}(X_{\geq n})}^{p},
\end{align*}
recalling that $p > \beta$ so that we can sum the geometric series. 

We now consider the complementary case that there is at least one vertex $v$ on the edge $e$ with $h(v) > n$. We can then assume that the vertices $v$ and $w$ on $e$ satisfy $\min\{h(v),h(w)\} \geq n$. We insert the estimate
\[
|u_{n}(v)-Pf(v)-u_{n}(w)+Pf(w)|^{p} \ls |u_{n}(v)-u_{n}(w)|^{p} + |Pf(v)-Pf(w)|^{p},
\]
into  \eqref{converging upper gradient} to obtain the bound
\begin{equation}\label{close besov}
\int_{e}g_{n}^{p}\,d\mu_{\beta} \ls \frac{\nu(B(v))}{\alpha^{(\beta-p)h(v)}}|u_{n}(v)-u_{n}(w)|^{p} + \int_{e}g^{p}\,d\mu_{\beta},
\end{equation}
using \eqref{transformed} together with the formula \eqref{edge gradient} for $g$ on the edge $e$ to rewrite the second term on the right. To estimate the first term on the right, we recall that if $v$ and $w$ belong to the same edge $e$ then $B(w) \subset 4\alpha B(v)$. Thus by Jensen's inequality and the doubling property of $\nu$, 
\begin{equation}\label{first close}
|u_{n}(v)-u_{n}(w)|^{p} \leq \dashint_{4\alpha B(v)}\dashint_{4\alpha B(v)}|T_{n}(Pf)(x)-T_{n}(Pf)(y)|^{p}\,d\nu(y)d\nu(x). 
\end{equation}

For each $v \in V_{k}$ with $k \geq n$ we let $S_{n}(v)$ denote the set of all vertices $v' \in V_{n}$ such that $B(v') \cap 4\alpha B(v) \neq \emptyset$. Then, since $r(B(v)) \leq r(B(v'))$, we must have $4\alpha B(v) \subset 10\alpha B(v')$. From Lemma \ref{doubling overlap} it follows that there is a constant $M = M(\tau,C_{\nu})$ such that $S_{n}(v)$ has at most $M$ members. For any pair of vertices $v',w' \in S_{n}(v)$ we then have $10\alpha B(v') \cap 10\alpha B(w') \neq \emptyset$ and therefore $B(w') \subset 30\alpha B(v')$. We choose $n_{0} = n_{0}(\alpha,\tau) \geq 1$ such that $30\tau \alpha \leq \alpha^{n_{0}}$ and then use Lemma \ref{large inclusion} to choose a vertex $x_{v} \in V_{n-n_{0}}$ such that $30\alpha B(v') \subset B(x_{v})$. By construction it then follows that $B(w') \subset B(x_{v})$ for each $w' \in S_{n}(v)$. We conclude by Lemma \ref{height connection} that for each $v' \in S_{n}(v)$ there is a vertical geodesic $\gamma_{x_{v}v'}$ joining $v'$ to $x_{v}$ in $X$. 

 We recall that $\psi_{v'}$ is $L\alpha^{n}$-Lipschitz for a constant $L = L(C_{\nu})$ depending only on the doubling constant for $Z$ by Proposition \ref{partition of unity}. We then estimate  for $x,y \in 4\alpha B(v)$, 
\begin{align*}
|T_{n}(Pf)(x)-T_{n}(Pf)(y)| &\leq \sum_{v' \in S_{n}(v)}|Pf(v')-Pf(x_{v})||\psi_{v'}(x)-\psi_{v'}(y)| \\
&\ls \alpha^{n-h(v)}\sum_{v' \in S_{n}(v)}\int_{\gamma_{x_{v}v'}}g\,ds_{\rho} \\
&\ls \frac{\alpha^{\beta n-h(v)}}{\nu(B(x_{v}))}\sum_{v' \in S_{n}(v)}\int_{\gamma_{x_{v}v'}}g\,d\mu_{\beta},
\end{align*}
with the last line being a direct consequence of Lemma \ref{arc length lemma} and the doubling property for $\nu$. Thus by Lemma \ref{convexity lemma} and a Jensen inequality calculation similar to \eqref{Jensen},
\begin{align*}
|T_{n}(Pf)(x)-T_{n}(Pf)(y)|^{p} &\ls \frac{\alpha^{p\beta n-ph(v)}}{\nu(B(x_{v}))^{p}}\sum_{v' \in S_{n}(v)}\left(\int_{\gamma_{x_{v}v'}}g\,d\mu_{\beta}\right)^{p} \\
&\ls \frac{\alpha^{\beta n- ph(v)}}{\nu(B(x_{v}))}\sum_{v' \in S_{n}(v)}\int_{\gamma_{x_{v}v'}}g^{p}\,d\mu_{\beta} \\
&\ls  \frac{\alpha^{\beta n- ph(v)}}{\nu(B(x_{v}))}\int_{\mathcal{U}_{n_{0}}(x_{v})}g^{p}\,d\mu_{\beta},
\end{align*}
where $\mathcal{U}_{n_{0}}(x_{v})$ denotes the union of all upward directed vertical geodesics of length $n_{0}$ in $X$ starting at $x_{v}$. By combining this with \eqref{close besov} and \eqref{first close} we conclude that
\begin{equation}\label{second close}
\int_{e}g_{n}^{p}\,d\mu_{\beta} \ls  \frac{\nu(B(v))}{\nu(B(x_{v}))}\alpha^{\beta(n-h(v))}\int_{\mathcal{U}_{n_{0}}(x_{v})}g^{p}\,d\mu_{\beta} + \int_{e}g^{p}\,d\mu_{\beta}.
\end{equation}
Here $v$ denotes one of the vertices on $e$ and $x_{v} \in V_{n-n_{0}}$ denotes a vertex such that $B(v') \subset B(x_{v})$ for each $v' \in S_{n}(v)$, as we constructed above. By summing this estimate over all edges $e \subset X_{\geq n}$, we conclude that
\begin{equation}\label{third close}
\|g_{n}\|_{L^{p}(X_{\geq n})}^{p} \ls \sum_{k=n}^{\infty}\alpha^{\beta(n-k)}\sum_{v \in V_{k}}\frac{\nu(B(v))}{\nu(B(x_{v}))}\int_{\mathcal{U}_{n_{0}}(x_{v})}g^{p}\,d\mu_{\beta} + \|g\|_{L^{p}(X_{\geq n})}^{p}.
\end{equation}
Since $B(v) \subset B(x_{v})$ by construction, the inner sum on the left can be bounded by
\begin{align*}
\sum_{v \in V_{k}} \frac{\nu(B(v))}{\nu(B(x_{v}))}\int_{\mathcal{U}_{n_{0}}(x_{v})}g^{p}\,d\mu_{\beta} &\leq \sum_{v_{0} \in V_{n-n_{0}}} \sum_{\substack{v \in V_{k} \\ B(v) \subset B(v_{0})}}  \frac{\nu(B(v))}{\nu(B(v_{0}))}\int_{\mathcal{U}_{n_{0}}(v_{0})}g^{p}\,d\mu_{\beta} \\
&\ls \sum_{v_{0} \in V_{n-n_{0}}} \int_{\mathcal{U}_{n_{0}}(v_{0})}g^{p}\,d\mu_{\beta} \\
&\ls \|g\|_{L^{p}(X_{\geq n-n_{0}})}^{p}. 
\end{align*}
Here the second inequality follows from the fact that the balls $B(v)$ are contained inside $B(v_{0})$ with bounded overlap by Lemma \ref{doubling overlap}, and the third inequality follows from the fact that the graph $X$ has uniformly bounded degree by Proposition \ref{doubling degree}, which implies for a given $v_{0} \in V_{n-n_{0}}$ that the intersection $\mathcal{U}_{n_{0}}(v_{0}) \cap \mathcal{U}_{n_{0}}(w_{0})$ can only be nontrivial for a uniformly bounded number of $w_{0} \in V_{n-n_{0}}$. Plugging this into \eqref{third close} and summing the geometric series gives
\[
\|g_{n}\|_{L^{p}(X_{\geq n})}^{p} \ls \|g\|_{L^{p}(X_{\geq n-n_{0}})}^{p}.
\]
By combining our estimates for $\|g_{n}\|_{L^{p}(X_{\leq n})}^{p}$ and $\|g_{n}\|_{L^{p}(X_{\geq n})}^{p}$ we conclude that
\begin{equation}\label{cutoff extension bound}
\|g_{n}\|_{L^{p}(X_{\rho})}^{p} \ls \|g\|_{L^{p}(X_{\geq n-n_{0}})}^{p}.
\end{equation}
Since 
\[
T(P(T_{n}(Pf))-Pf) = T_{n}(Pf)-f,
\]
$\nu$-a.e.~ by Proposition \ref{extend to trace}, we conclude from \eqref{cutoff extension bound} and Proposition \ref{Dirichlet Besov finite} that
\[
\|T_{n}(Pf)-f\|_{B^{\theta}_{p}(Z)}\ls \|g\|_{L^{p}(X_{\geq n-n_{0}})}.
\]
This gives \eqref{cutoff converge Besov}. The bound \eqref{estimate partial trace} follows from the triangle inequality in $B^{\theta}_{p}(Z)$ and the observation that $\|g\|_{L^{p}(X_{\rho})} \ls \|f\|_{B^{\theta}_{p}(Z)}$ by Proposition \ref{extension Besov}. 

Lastly we need to show that $T_{n}(Pf) \rightarrow f$ in $B^{\theta}_{p}(Z)$ as $n \rightarrow \infty$. By \eqref{cutoff converge Besov} it suffices to show that $\|g\|_{L^{p}(X_{\geq n})} \rightarrow 0$ as $n \rightarrow \infty$. As in the proof of Proposition \ref{extension Besov}, we choose $m = m(\alpha,\tau)$ to be the minimal integer such that $\alpha^{m} \geq 8\tau \alpha$ and then sum the estimate \eqref{layer bound} only over $k \geq n$ to obtain
\[
\|g\|_{L^{p}(X_{\geq n})}^{p} \ls \int_{Z} \sum_{k=n}^{\infty} \dashint_{B_{Z}(x,\alpha^{-n+m})} \frac{|f(x)-f(y)|^{p}}{\alpha^{p \theta(-n+m)}}\,d\nu(y)d\nu(x).
\]
We recall from Lemma \ref{Besov estimate} that we have the estimate
\[
\|f\|_{B^{\theta}_{p}(Z)}^{p} \asymp \int_{Z} \sum_{n\in \Z} \dashint_{B_{Z}(x,\alpha^{-n})} \frac{|f(x)-f(y)|^{p}}{\alpha^{-p \theta n}}\,d\nu(y)d\nu(x).
\]
The finiteness of $\|f\|_{B^{\theta}_{p}(Z)}$ implies that the sum on the right converges for $\nu$-a.e.~ $x \in Z$. Thus for $\nu$-a.e.~ $x \in Z$ we must have
\[
\lim_{n \rightarrow \infty} \sum_{k=n}^{\infty}\dashint_{B_{Z}(x,\alpha^{-n})} \frac{|f(x)-f(y)|^{p}}{\alpha^{-p \theta n}}\,d\nu(y) = 0. 
\]
By the dominated convergence theorem it follows that
\[
\lim_{n \rightarrow \infty} \int_{Z}\sum_{k=n}^{\infty}\dashint_{B_{Z}(x,\alpha^{-n})} \frac{|f(x)-f(y)|^{p}}{\alpha^{-p \theta n}}\,d\nu(y)\,d\nu(x) = 0,
\]
from which we conclude that $\|g\|_{L^{p}(X_{\geq n})} \rightarrow 0$ as $n \rightarrow \infty$, as desired. The final assertion concerning $\ch{B}^{\theta}_{p}(Z)$ follows immediately from the bound \eqref{global p trace} on $\|T_{n}(Pf)-f\|_{L^{p}(Z)}$ from Proposition \ref{Lp local trace}, noting again that $T(Pf) = f$ $\nu$-a.e. 
\end{proof}

%A function $f: Z \rightarrow \R$ is \emph{locally Lipschitz} if it is Lipschitz on each ball $B \subset Z$. Since the functions $T_{n}(Pf)$ considered in Proposition \ref{converge Besov} are locally Lipschitz, we obtain the following immediate corollary for each $p \geq 1$ and $0 < \theta < 1$ by choosing $\beta = p(1-\theta) < p$ and applying Proposition \ref{converge Besov}. 

%\begin{cor}\label{local lipschitz dense}
%Let $p \geq 1$ and $0 < \theta < 1$ be given. Then any function $f \in B^{\theta}_{p}(Z)$ is the limit in $B^{\theta}_{p}(Z)$ of a sequence of locally Lipschitz functions on $Z$. Similarly any $f \in \ch{B}^{\theta}_{p}(Z)$ is the limit in $\ch{B}^{\theta}_{p}(Z)$ of a sequence of locally Lipschitz functions on $Z$.
%\end{cor}

%For $\ch{B}^{\theta}_{p}(Z)$ we will show in the next section that this conclusion can be improved to obtain the density of compactly supported Lipschitz functions in $\ch{B}^{\theta}_{p}(Z)$. 

In a different direction, Proposition \ref{extend to trace} shows for any $f \in \t{B}^{\theta}_{p}(Z)$ that we can always construct a $p$-integrable Borel function $g: \bar{X}_{\rho} \rightarrow [0,\infty]$ on $\bar{X}_{\rho}$ and a measurable subset $G \subset Z$ with $\nu(Z \backslash G) = 0$ such that for every curve $\gamma$ in $\bar{X}_{\rho}$ joining two points $x,y \in G$ we have
\begin{equation}\label{hyp upper gradient}
|f(x)-f(y)| \leq \int_{\gamma}g \,ds. 
\end{equation}
Namely we can take $G \subset Z$ to be the full measure subset of $Z$ on which we have $Pf|_{Z} = f$ and we can take  $g$ to be a $p$-integrable upper gradient of $Pf$ on $\bar{X}_{\rho}$. This leads to the following definition. 

\begin{defn}\label{define hyp upper gradient}
Let $f: Z \rightarrow [-\infty,\infty]$ be a measurable function. We say that a Borel function $g: \bar{X}_{\rho} \rightarrow [0,\infty]$ is a \emph{hyperbolic upper gradient} for $f$ if there is a full measure subset $G \subset Z$ such that for every $x,y \in G$ and every curve $\gamma$ in $\bar{X}_{\rho}$ joining $x$ to $y$ the upper gradient inequality \eqref{hyp upper gradient} holds. 
\end{defn}

It is natural to ask whether the existence of a $p$-integrable hyperbolic upper gradient alone is sufficient to characterize membership in the Besov space $B^{\theta}_{p}(Z)$. We will show that this question has a positive answer through the characterization of Besov spaces by \emph{fractional Haj\l asz gradients} introduced by Koskela-Yang-Zhou \cite{KYZ11}. We remark that the connection between fractional Haj\l asz gradients and hyperbolic upper gradients demonstrated in the proof of Proposition \ref{characterize hyp upper} below may be of independent interest.

\begin{prop}\label{characterize hyp upper}
Suppose that $f$ admits a $p$-integrable hyperbolic upper gradient $g$ on $\bar{X}_{\rho}$. Then $f \in \t{B}^{\theta}_{p}(Z)$ with 
\begin{equation}\label{hyp upper norm bound}
\|f\|_{B^{\theta}_{p}(Z)} \ls \|g\|_{L^{p}(X_{\rho})}. 
\end{equation}
\end{prop}

%The implied constant in \eqref{hyp upper norm bound} will depend additionally on $Q$ and the constant $C_{\mathrm{low}}$ in \eqref{lower volume}. The threshold $p \theta > Q$ is essentially the same as the threshold used in Proposition \ref{Holder embedding} for functions in $B^{\theta}_{p}(Z)$ to be locally H\"older continuous. This is reflected in the pointwise nature of the estimates used to prove Proposition \ref{characterize hyp upper}. It is an interesting question whether Proposition \ref{characterize hyp upper} in fact holds for any $p \geq 1$, without the constraint $p \theta > Q$.

\begin{proof}
For each $z \in Z$ and $k \in \Z$ we choose a vertex $v_{k,z} \in V_{k}$ such that $d(\pi(v_{k,z}),z) < \alpha^{-k}$. We then fix an ascending vertical geodesic ray $\gamma_{k,z}$ in $X$ starting from $v_{k,z}$ that is anchored at $z$ as in Lemma \ref{second height connection}. 

For a vertex $v \in V$ we recall that we write $\mathcal{H}(v)$ for the set of all horizontal edges in $X$ having $v$ as a vertex. For each $k \in \Z$ we then define a measurable function $g_{k}:Z \rightarrow [0,\infty]$ by setting 
\[
g_{k}(z) = \alpha^{\theta (k+1)}\left(\int_{\gamma_{k,z}}g\,ds_{\rho} + \int_{\mathcal{H}(v_{k,z})} g\,ds_{\rho}\right).
\]
Let $G \subset Z$ be the full measure subset given as part of Definition \ref{define hyp upper gradient} on which the upper gradient inequality \eqref{hyp upper gradient} holds. If $x,y \in Z$ satisfy $d(x,y) <  \alpha^{-k}$ for a given $k \in \Z$ then we will have $y \in B(v_{k,x})$  and consequently $v_{k,x}$ and $v_{k,y}$ will be joined by a horizontal edge contained inside $\mathcal{H}(v_{k,x})$. Thus if $d(x,y) \geq  \alpha^{-k-1}$ then we will have
\[
|f(x)-f(y)| \leq d(x,y)^{\theta}(g_{k}(x)+g_{k}(y)),
\]
by the definition of the hyperbolic upper gradient. We thus conclude that the sequence $\{g_{k}\}_{k \in \Z}$ of measurable functions on $Z$ defines a fractional $\theta$-Haj\l asz gradient of $f$, as defined in \cite[Definition 1.3]{GKZ13}; the definition given there takes $\alpha = 2$ but an equivalent definition for the purpose of defining Besov spaces is given by replacing $2$ with any $ \alpha > 1$, with the equivalence coming from the doubling property of $\nu$. By the characterization of $B^{\theta}_{p}(Z)$ using fractional $\theta$-Haj\l asz gradients \cite[Theorem 1.2]{GKZ13} we then have
\begin{equation}\label{GKZ comparison}
\|f\|_{B^{\theta}_{p}(Z)}^{p} \ls \sum_{k \in \Z} \|g_{k}\|^{p}_{L^{p}(Z)}, 
\end{equation}
with the current caveat that both sides may potentially take the value infinity. Thus to prove the proposition it suffices to produce an estimate
\[
\sum_{k \in \Z} \|g_{k}\|^{p}_{L^{p}(Z)} \ls \|g\|_{L^{p}(X_{\rho})}^{p} < \infty.
\]
with the last inequality coming from our assumption that $g$ is $p$-integrable on $\bar{X}_{\rho}$ (hence also on $X_{\rho}$) with respect to the measure $\mu_{\beta}$. 

For each vertex $v \in V$ we will write $E(v)$ for the set of all edges in $X$ that have $v$ as a vertex. For $j \in \Z$ and $z \in Z$ we define
\[
u_{j}(z) =  \sum_{v\in V_{j}:\,d(\pi(v),z) <  \alpha^{-j}}\int_{E(v)}g\,ds_{\rho}.
\]
Then by the construction of $g_{k}$ we have for all $z \in Z$, 
\begin{equation}\label{hyp upper sequence}
g_{k}(z) \leq  \alpha^{\theta(k+1)}\sum_{j=k}^{\infty} u_{j}(z).
\end{equation}
For each vertex $v \in V$ we define a function $u_{v}$ on $Z$ by 
\[
u_{v} = \left(\int_{E(v)}g\,ds_{\rho}\right)\chi_{B_{Z}(\pi(v), \alpha^{-h(v)})}. 
\]
Then for each $z \in Z$ we have
\[
u_{j}(z) \leq \sum_{v \in V_{j}} u_{v}(z),
\]
and by Lemma \ref{doubling overlap} the sum on the right has a number of nonzero terms uniformly bounded in terms of $ \alpha$, $\tau$, and $C_{\nu}$ for each $z \in Z$. Thus by Lemma \ref{convexity lemma} we have 
\[
u_{j}^{p}(z) \ls \sum_{v \in V_{j}} u_{v}^{p}(z).
\]
As in the proofs of Propositions \ref{Lp ball trace} and \ref{Dirichlet Besov finite}, when $p > 1$ we let $\la > 0$ be a given parameter and apply H\"older's inequality for sequences to the right side of \eqref{hyp upper sequence} with conjugate exponents $p$ and $q = \frac{p}{p-1}$ to obtain
\begin{align*}
g_{k}(z) &\ls  \alpha^{\theta k} \left(\sum_{j = k}^{\infty} \alpha^{p\la j}\sum_{v \in V_{j}} u_{v}^{p}(z)\right)^{1/p}\left(\sum_{j=k}^{\infty} \alpha^{-q \la j}\right)^{1/q} \\
&\ls  \alpha^{(\theta-\la)k}\left(\sum_{j = k}^{\infty} \alpha^{p\la j}\sum_{v \in V_{j}} u_{v}^{p}(z)\right)^{1/p},
\end{align*}
where the implied constants now depend additionally on $\la$. Thus
\[
g_{k}^{p}(z) \ls  \alpha^{p(\theta-\la)k}\sum_{j = k}^{\infty} \alpha^{p\la j}\sum_{v \in V_{j}} u_{v}^{p}(z).
\]
We remark that this inequality trivially follows from \eqref{hyp upper sequence} when $p = 1$, so we can assume that it holds for $p = 1$ as well. By integrating each side over $Z$ and noting that $B_{Z}(\pi(v), \alpha^{-h(v)}) \subset B(v)$ we obtain
\[
\|g_{k}\|_{L^{p}(Z)}^{p} \ls  \alpha^{p(\theta-\la)k}\sum_{j = k}^{\infty} \alpha^{p\la j}\sum_{v \in V_{j}} \left(\int_{E(v)}g\,ds_{\rho}\right)^{p}\nu(B(v)).
\]
Since the total arclength with respect to the distance $d_{\rho}$ of the collection of edges $E(v)$ is comparable to $ \alpha^{-j}$ for $v \in V_{j}$ (here we are using that $X$ has uniformly bounded vertex degree by Proposition \ref{doubling degree}), we can use Jensen's inequality with respect to the arclength measure $ds_{\rho}$ to conclude that
\[
\left(\int_{E(v)}g\,ds_{\rho}\right)^{p} \leq  \alpha^{(1-p)j}\int_{E(v)}g^{p}\,ds_{\rho}.
\]
via an analogous computation to the one done in inequality \eqref{Jensen}. Thus we conclude that
\[
\|g_{k}\|_{L^{p}(Z)}^{p} \ls  \alpha^{p(\theta-\la)k}\sum_{j = k}^{\infty} \alpha^{p\la j}\sum_{v \in V_{j}}\left(\int_{E(v)}g^{p}\,ds_{\rho}\right) \alpha^{(1-p)j}\nu(B(v)).
\]
By applying Lemma \ref{arc length lemma} to the integral on the right we conclude that 
\[
\|g_{k}\|_{L^{p}(Z)}^{p} \ls  \alpha^{p(\theta-\la)k}\sum_{j = k}^{\infty} \alpha^{p(\la-\theta) j}\sum_{v \in V_{j}}\int_{E(v)}g^{p}\,d\mu_{\beta},
\]
recalling that $\theta = 1-\beta/p$. By setting $\la = \theta/2 > 0$ we conclude that
\[
\|g_{k}\|_{L^{p}(Z)}^{p} \ls \sum_{j = k}^{\infty} \alpha^{(p\theta/2)(k-j)}\sum_{v \in V_{j}}\int_{E(v)}g^{p}\,d\mu_{\beta},
\]
Summing this over all $k \in \Z$ gives 
\[
\sum_{k \in \Z} \|g_{k}\|_{L^{p}(Z)}^{p} \ls \sum_{k \in \Z} \sum_{j = k}^{\infty} \alpha^{(p\theta/2)(k-j)}\sum_{v \in V_{j}}\int_{E(v)}g^{p}\,d\mu_{\beta}.
\]
Using Tonelli's theorem we can then rewrite the sum on the right to obtain the desired bound
\begin{align*}
\sum_{k \in \Z} \|g_{k}\|_{L^{p}(Z)}^{p} &\ls \sum_{j \in \Z} \sum_{v \in V_{j}}\int_{E(v)}g^{p}\,d\mu_{\beta} \left(\sum_{k=-\infty}^{j} \alpha^{(p\theta/2)(k-j)}\right) \\
&\ls \sum_{j \in \Z} \sum_{v \in V_{j}}\int_{E(v)}g^{p}\,d\mu_{\beta} \\
&\ls \|g\|_{L^{p}(X_{\rho})}^{p}. 
\end{align*}
\end{proof}

\section{Properties of Besov spaces}\label{sec:properties}

In this final section we apply the results of Sections \ref{sec:trace} and \ref{sec:extend} to establish a number of properties of the Besov spaces $\t{B}^{\theta}_{p}(Z)$ on a complete doubling metric measure space $(Z,d,\nu)$ for $p \geq 1$ and $0 < \theta < 1$. In particular we prove Corollaries \ref{cor:quasicontinuous}, \ref{cor:holder embed}, and \ref{cor:Lebesgue points}. 

For this section we consider a complete doubling metric measure space $(Z,d,\nu)$. We let $X$ be a hyperbolic filling of $Z$ with parameters $\alpha, \tau > 1$ satisfying \eqref{tau requirement}. We let $\mu$ be the lift of the measure $\nu$ to $X$ defined in \eqref{lift definition}. We then let $X_{\rho}$ be the uniformized hyperbolic filling corresponding to these parameters as defined in Section \ref{sec:fillings}. For a given $p \geq 1$ and $0 < \theta < 1$ we set $\beta = p(1-\theta)$, noting that we then have $p > \beta$ and $\theta = (p-\beta)/p$. Note that, in contrast to Sections \ref{sec:trace} and \ref{sec:extend}, we are considering $\beta$ here as depending on $p$ and $\theta$ instead of considering $\theta$ as depending on $p$ and $\beta$. We then let $\mu_{\beta}$ be the measure on $\bar{X}_{\rho}$ defined by \eqref{beta def}.

Throughout the rest of this section all implied constants will depend only on the parameters $\alpha$ and $\tau$ of the hyperbolic filling $X$, the doubling constant $C_{\nu}$ for $\nu$, and the exponents $p \geq 1$ and $0 < \theta < 1$. The uniformized filling $X_{\rho}$ will always be considered to be equipped with the measure $\mu_{\beta}$ for $\beta = p(1-\theta)$, where $p$ and $\theta$ are given as in the hypotheses of each proposition. If one wants to remove the dependencies on the parameters of the hyperbolic filling then there is no harm in fixing $\alpha = 2$ and $\tau = 4$ in everything that follows.

We first show as a consequence of Proposition \ref{extend to trace} that the Besov capacity can be computed in terms of the $C_{p}^{\bar{X}_{\rho}}$-capacity and vice versa. 

\begin{prop}\label{compute single capacity} 
For any set $G \subset Z$ we have $C_{p}^{\bar{X}_{\rho}}(G) \asymp C_{\ch{B}^{\theta}_{p}}^{Z}(G)$. 
\end{prop}

\begin{proof}
Let $G \subset Z$ be a given subset. Let $f \in \ch{B}_{p}^{\theta}(Z)$ be given such that $f \geq 1$ $\nu$-a.e.~ on a neighborhood $U$ of $G$. By truncating $f$ using Lemma \ref{truncate Besov} and then redefining $f$ on a $\nu$-null set, we can assume that we in fact have $f \equiv 1$ on $U$. Let $P_{0}f \in \t{N}^{1,p}(\bar{X}_{\rho})$ be the extension of $f$ given by Proposition \ref{extension Besov}. Since every point of $U$ is an $L^{1}(Z)$-Lebesgue point for $f$, we conclude by Proposition \ref{q Lebesgue points} that $P_{0}f|_{U} = f$. Thus $P_{0}f = 1$ on $U$. It follows that $P_{0}f$ is admissible for the $C_{p}^{\bar{X}_{\rho}}$-capacity of $G$ and therefore by \eqref{extension truncate bound},
\[
C_{p}^{\bar{X}_{\rho}}(G) \leq \|P_{0}f\|_{N^{1,p}(\bar{X}_{\rho})}^{p} \ls \|f\|_{\ch{B}^{\theta}_{p}(Z)}^{p}.
\]
Minimizing over all admissible $f$ for the $C_{\ch{B}^{\theta}_{p}}^{Z}$-capacity of $G$ then gives $C_{p}^{\bar{X}_{\rho}}(G) \ls C_{\ch{B}^{\theta}_{p}}^{Z}(G)$. 

For the other direction, let $\e > 0$ be given and let $U \subset \bar{X}_{\rho}$ be an open set containing $G$ such that $C_{p}^{\bar{X}_{\rho}}(U \backslash G) < \e$, which we can find by Theorem \ref{outer capacity}. We let $u \in \t{N}^{1,p}(\bar{X}_{\rho})$ be a function such that $u \geq 1$ on $U$ and $\|u\|_{N^{1,p}(\bar{X}_{\rho})} < C_{p}^{\bar{X}_{\rho}}(U) + \e$. Then $u|_{Z} \in \ch{B}^{\theta}_{p}(Z)$ by Proposition \ref{trace characterization} and $u \geq 1$ on an open subset $U \cap Z \subset Z$ of $Z$ that contains $G$. Thus
\begin{align*}
C_{\ch{B}^{\theta}_{p}}^{Z}(G) &\leq \|u|_{Z}\|_{\ch{B}^{\theta}_{p}(Z)} \\
&\ls \|u\|_{N^{1,p}(\bar{X}_{\rho})}\\
&< C_{p}^{\bar{X}_{\rho}}(U) + \e \\
&< C_{p}^{\bar{X}_{\rho}}(G) + 2\e. 
\end{align*}
Letting $\e \rightarrow 0$ gives the desired result.
\end{proof}

As we remarked after the proof of Proposition \ref{extend to trace}, given $f \in \t{B}^{\theta}_{p}(Z)$ the representative $Pf|_{Z}$ of $f$ in $\t{B}^{\theta}_{p}(Z)$ usually has better regularity properties. We can use this observation together with Proposition \ref{extension hyperbolic Poincare} to prove Corollary \ref{cor:quasicontinuous}.

\begin{proof}[Proof of Corollary \ref{cor:quasicontinuous}]
We let $\ch{f} = Pf|_{Z}$ be the restriction of $Pf$ to $Z$ that was considered in the proof of Proposition \ref{extend to trace}; as shown in that proof we have $\ch{f} = f$ $\nu$-a.e.~ on $Z$. Since $Pf \in \t{D}^{1,p}(\bar{X}_{\rho})$, we have by Theorem \ref{outer capacity} and the remarks afterward that $Pf$ is $C_{p}^{\bar{X}_{\rho}}$-quasicontinuous. Thus for each $\eta > 0$ we can find an open subset $U \subset \bar{X}_{\rho}$ such that $C_{p}^{\bar{X}_{\rho}}(U) < \eta$ and $Pf|_{\bar{X}_{\rho} \backslash U}$ is continuous. Setting $W = U \cap Z$, we conclude that $W$ is open in $Z$ and that $\ch{f}|_{Z\backslash W}$ is continuous. By Proposition \ref{compute single capacity} we have
\[
C_{\ch{B}^{\theta}_{p}}^{Z}(W) \ls C_{p}^{\bar{X}_{\rho}}(W) \leq C_{p}^{\bar{X}_{\rho}}(U) < \eta. 
\] 
Since $\eta > 0$ was arbitrary, we conclude that $\ch{f}$ is $C_{\ch{B}^{\theta}_{p}}^{Z}$-quasicontinuous.

It remains to show that $\ch{f}$ has $L^{p}(Z)$-Lebesgue points quasieverywhere with respect to the Besov capacity. By Proposition \ref{Dirichlet characterization} we have for $C_{p}^{\bar{X}_{\rho}}$-q.e.~ $z \in Z$, 
\begin{equation}\label{first Lebesgue use}
\lim_{r \rightarrow 0^{+}} \dashint_{B_{\rho}(z,r)}|Pf-\ch{f}(z)|^{p}\,d\mu_{\beta} = 0,
\end{equation}
where we have used that $\mu_{\beta}$ is extended to $Z = \p X_{\rho}$ by $\mu_{\beta}(\p X_{\rho}) = 0$. By Proposition \ref{absolutely continuous capacity} we then have that $\nu$-a.e.~ point of $Z$ is an $L^{1}(\bar{X}_{\rho})$-Lebesgue point for $u$. Let $g$ be any $p$-integrable $p$-weak upper gradient for $u$ on $\bar{X}_{\rho}$. Then by \cite[Lemma 9.2.4]{HKST} we have for $C_{p}^{\bar{X}_{\rho}}$-q.e.~ $z \in Z$, 
\begin{equation}\label{second Lebesgue use}
\lim_{r \rightarrow 0^{+}} r^{p}\dashint_{B_{\rho}(z,r)}g^{p}\,d\mu_{\beta} = 0.
\end{equation}
By Proposition \ref{compute single capacity} it follows that each of these assertions also hold for $C_{\ch{B}^{\theta}_{p}}^{Z}$-q.e.~ point of $Z$. We can thus complete the proof by showing that any point $z \in Z$ for which the two limits \eqref{first Lebesgue use} and \eqref{second Lebesgue use} hold is an $L^{p}(Z)$-Lebesgue point of $\ch{f}$. 

Let $z \in Z$ be a point for which \eqref{first Lebesgue use} and \eqref{second Lebesgue use} hold. By Proposition \ref{extension hyperbolic Poincare} there is a constant $C = C(\alpha,\tau)$ such that for each ball $B = B_{Z}(z,r) \subset Z$, 
\begin{equation}\label{pre bound}
\dashint_{B}|\ch{f}-Pf_{\hat{B}}|^{p}\,d\nu \ls r^{p} \dashint_{C\hat{B}}g^{p}\,d\mu_{\beta}.
\end{equation}
Thus, using Jensen's inequality and \eqref{pre bound},
\begin{align*}
\dashint_{B}|\ch{f}-\ch{f}(z)|^{p}\,d\nu &\ls \dashint_{B}|\ch{f}-Pf_{\hat{B}}|^{p}\,d\nu + |Pf_{\hat{B}}-\ch{f}(z)|^{p} \\
&\ls r^{p} \dashint_{C\hat{B}}g^{p}\,d\mu_{\beta} + \dashint_{\hat{B}}|Pf-\ch{f}(z)|^{p}\,d\mu_{\beta}.
\end{align*}
The second term converges to $0$ by \eqref{first Lebesgue use} and the first term converges to $0$ by \eqref{second Lebesgue use} and the doubling property of $\nu$. We conclude that $z$ is an $L^{p}(Z)$-Lebesgue point of $\ch{f}$, as desired. 
\end{proof}

We next consider embeddings of Besov spaces into H\"older spaces.  For the proof of this next proposition we recall the notion of \emph{relative lower volume decay} defined in \eqref{lower volume}, and recall that every doubling measure $\nu$ satisfies this condition for $Q = \log_{2} C_{\nu}$.

\begin{prop}\label{prop:holder embed}
Let $f \in \t{B}^{\theta}_{p}(Z)$ be given. Suppose that $\nu$ has relative lower volume decay of order $Q > 0$, set $Q_{\beta} = \max\{1,Q+p(1-\theta)\}$, and assume that $p > Q_{\beta}$. Then $f$ has a representative $\ch{f}$ in $\t{B}^{\theta}_{p}(Z)$ such that for each ball $B \subset Z$ of radius $r > 0$ we have for any $x,y \in B$ and any upper gradient $g$ of $Pf$ on $4\hat{B}$, 
\begin{equation}\label{holder embed inequality}
|\ch{f}(x)-\ch{f}(y)| \ls r^{Q_{\beta}/p}d(x,y)^{1-Q_{\beta}/p}\left(\dashint_{4\hat{B}}g^{p}\,d\mu_{\beta}\right)^{1/p}. 
\end{equation}
Consequently $\ch{f}$ is $(1-Q_{\beta})/p$-H\"older continuous on any ball $B \subset Z$. 
\end{prop}

Corollary \ref{cor:holder embed} follows immediately from Proposition \ref{prop:holder embed}. The constant in inequality \eqref{holder embed inequality} will depend additionally on $Q$ and the constant $C_{\mathrm{low}}$ in \eqref{lower volume}. 

\begin{proof}
Let $f \in \t{B}^{\theta}_{p}(Z)$ be given with $p$ satisfying $p > Q_{\beta}$. As in the proof of the previous corollary we let $\ch{f} = Pf|_{Z}$ denote the restriction of the extension $Pf$ to $Z$, which satisfies $Pf|_{Z} = f$ $\nu$-a.e.~ by Proposition \ref{extend to trace}. By Lemma \ref{doubling dimension} the metric measure space $(\bar{X}_{\rho},d_{\rho},\mu_{\beta})$ has relative lower volume decay of order $Q_{\beta}$. Since $p > Q_{\beta}$ and since the metric measure space $(\bar{X}_{\rho},d_{\rho},\mu_{\beta})$ is geodesic and supports a $1$-Poincar\'e inequality, we conclude from the Morrey embedding theorem \cite[Theorem 9.2.14]{HKST} that for any ball $B' \subset \bar{X}_{\rho}$, any measurable function $u: 4B' \rightarrow [-\infty,\infty]$, and any $p$-integrable upper gradient $g$ of $u$ on $4B'$, we have for $x,y \in B'$, 
\begin{equation}\label{proto holder}
|u(x)-u(y)| \ls (\diam \, B')^{Q_{\beta}/p}d_{\rho}(x,y)^{1-Q_{\beta}/p}\left(\dashint_{4B'}g^{p}\,d\mu_{\beta}\right)^{1/p},
\end{equation}
with implied constant depending additionally on $Q$ and the constant $C_{\mathrm{low}}$ in \eqref{lower volume}. This implies in particular that $u$ is actually finite everywhere on $4B'$, i.e., $u: 4B' \rightarrow \R$. To deduce inequality \eqref{holder embed inequality} from inequality \eqref{proto holder}, we fix a ball $B \subset Z$ of radius $r > 0$ and set $B' = \hat{B}$ and $u = Pf$ in \eqref{proto holder}. Note then that $\diam \, \hat{B} \leq 2r$. The inequality \eqref{holder embed inequality} then follows by specializing \eqref{prop:holder embed} to the case of $x,y \in Z = \p X_{\rho}$. The final claim follows by noting that $Pf$ has a $p$-integrable upper gradient $g$ on $\bar{X}_{\rho}$ by Proposition \ref{extension Besov} which restricts to a $p$-integrable gradient of $Pf$ on $4\hat{B}$ for each ball $B \subset Z$. 
\end{proof}

The calculations after \cite[Proposition 13.7]{BBS21} give more precise details about the settings in which Corollary \ref{cor:holder embed} applies. We note in particular that if $Q \geq 1$ then the requirement $p > Q_{\beta}$ is equivalent to $p > Q/\theta$ and we have $1-Q_{\beta}/p = \theta-Q/p$. 

Lastly we prove Corollary \ref{cor:Lebesgue points}, which strengthens Corollary \ref{cor:quasicontinuous} under the additional hypothesis that $\nu$ satisfies a \emph{reverse-doubling condition} for an exponent $\eta > 0$ and any $0 < r' \leq r$,
\begin{equation}\label{reverse doubling}
\frac{\nu(B_{Z}(z,r'))}{\nu(B_{Z}(z,r))} \leq C_{\mathrm{rev}} \left(\frac{r'}{r}\right)^{\eta},
\end{equation}
for some constant $C_{\mathrm{rev}} \geq 1$. Under the hypothesis that $\nu$ is doubling, the reverse-doubling condition \eqref{reverse doubling} on $\nu$ is equivalent to $Z$ being \emph{uniformly perfect} \cite[Lemma 7]{MO19}, which in turn implies that for each ball $B \subset Z$ we have
\begin{equation}\label{radius to diameter}
r(B) \asymp \diam(B),
\end{equation}
with implied constants depending only on the doubling constant for $\nu$ and the constant $C_{\mathrm{rev}}$ and exponent $\eta$ in \eqref{reverse doubling}.

We then have the following two-weighted Poincar\'e inequality \cite[Proposition 13.5]{BBS21}, which is a special case of an inequality of Bj\"orn-Ka\l amajska \cite[Theorem 3.1]{BK21}; we note that this inequality still applies in our setting since we assume that the reverse-doubling condition \eqref{reverse doubling} holds at \emph{all} scales and since we have that the metric measure spaces $(X_{\rho},d_{\rho},\mu_{\beta})$ and $(\bar{X}_{\rho},d_{\rho},\mu_{\beta})$ support a $p$-Poincar\'e inequality. Below we set $\beta = p(1-\theta) > 0$ and $B = B_{Z}(z,r)$ for $z \in Z$ and $r > 0$. We recall that $\hat{B} = B_{\rho}(z,r)$ then denotes the ball centered at $z$ of the same radius in $\bar{X}_{\rho}$. 

\begin{prop}\label{two-weighted}
Let $1 \leq p < q < \infty$ and $u \in \t{D}^{1,p}(\bar{X}_{\rho})$ be such that $\nu$-a.e.~ point of $Z$ is an $L^{1}(\bar{X}_{\rho})$-Lebesgue point for $u$. Let $g$ be a $p$-integrable $p$-weak upper gradient for $u$ on $\bar{X}_{\rho}$. Suppose further that $\nu$ satisfies the reverse-doubling condition \eqref{reverse doubling}. Then for all balls $B = B_{Z}(z,r) \subset Z$ with $z \in Z$ and $r > 0$, 
\begin{equation}\label{two-weighted inequality}
\left(\int_{B}|u-u_{\hat{B}}|^{q}\,d\nu\right)^{1/q} \ls \Theta_{q}(r)\left(\int_{2\hat{B}}g^{p}\,d\mu_{\beta}\right)^{1/p}.
\end{equation}
with implied constant depending additionally on the constant $C_{\mathrm{rev}}$ and exponent $\eta > 0$ in \eqref{reverse doubling}, where
\[
\Theta_{q}(r) = \sup_{0 < s \leq r} \sup_{z \in B} \; s  \frac{\nu(B_{Z}(z,s))^{1/q}}{\mu_{\beta}(B_{\rho}(z,s))^{1/p}}.
\]
\end{prop}

We remark that our expression for $\Theta_{q}(r)$ is equivalent to the one given in \cite[Proposition 13.5]{BBS21}, since they consider $\nu$ as a measure on $\bar{X}_{\rho}$ by setting $\nu(X_{\rho}) = 0$. In \eqref{two-weighted inequality} $u_{\hat{B}}$ denotes the mean value of $u$ over the ball $\hat{B}$ when $\hat{B}$ is equipped with the measure $\mu_{\beta}$. By Lemma \ref{hull measure} we have
\begin{equation}\label{simplify Theta}
\Theta_{q}(r) \asymp \sup_{0 < s \leq r} \sup_{z \in B}\; s^{1-\beta/p}\nu(B_{Z}(z,s))^{1/q-1/p}.
\end{equation}
For the proof of the final Corollary \ref{cor:Lebesgue points} we assume that $\nu$ has relative lower volume decay of order $Q > p\theta$ and set $Q_{*} = Qp/(Q-p\theta)$. We note that we always have $Q_{*} > p$.

\begin{proof}[Proof of Corollary \ref{cor:Lebesgue points}]
Throughout this proof all implied constants will be allowed to additionally depend on the constant and exponent in the reverse-doubling estimate \eqref{reverse doubling} as well as the constant and exponent in the relative lower volume decay estimate \eqref{lower volume}.  

Let $f \in \t{B}^{\theta}_{p}(Z)$ be given as in the hypotheses. We let $u \in \t{D}^{1,p}(\bar{X}_{\rho})$ be any function such that $u|_{Z} = f$ $\nu$-a.e.~ on $Z$; by Proposition \ref{extend to trace} we can take $u = Pf$, but we will allow for a more general choice of $u$ in the proof. We then set $\ch{f} = u|_{Z}$. We note that this function $u$ agrees $\mu_{\beta}$-a.e.~ with the extension $\hat{u}$ of $u|_{X_{\rho}}$ to $\bar{X}_{\rho}$ given by Proposition \ref{Dirichlet characterization} since $\mu_{\beta}(\p X_{\rho}) = 0$, which implies that $\hat{u} = u$ $C_{p}^{\bar{X}_{\rho}}$-q.e.~ since both of these functions belong to $\t{D}^{1,p}(\bar{X}_{\rho})$. By Proposition \ref{Dirichlet characterization} and H\"older's inequality we then have for $C_{p}^{\bar{X}_{\rho}}$-q.e.~ $z \in Z$, 
\begin{equation}\label{first mod Lebesgue use}
\lim_{r \rightarrow 0^{+}} \dashint_{B_{\rho}(z,r)}|u-\ch{f}(z)|\,d\mu_{\beta} = 0,
\end{equation}
where we have used that $\mu_{\beta}$ is extended to $Z = \p X_{\rho}$ by $\mu_{\beta}(\p X_{\rho}) = 0$. By Proposition \ref{absolutely continuous capacity} we then have that $\nu$-a.e.~ point of $Z$ is an $L^{1}(\bar{X}_{\rho})$-Lebesgue point for $u$. Let $g$ be any $p$-integrable $p$-weak upper gradient for $u$ on $\bar{X}_{\rho}$. Then by \cite[Lemma 9.2.4]{HKST} we have for $C_{p}^{\bar{X}_{\rho}}$-q.e.~ $z \in Z$, 
\begin{equation}\label{second mod Lebesgue use}
\lim_{r \rightarrow 0^{+}} r^{p}\dashint_{B_{\rho}(z,r)}g^{p}\,d\mu_{\beta} = 0.
\end{equation}
By Proposition \ref{compute single capacity} it follows that each of these assertions also hold for $C_{\ch{B}^{\theta}_{p}}^{Z}$-q.e.~ point of $Z$. Similarly to the proof of Corollary \ref{cor:quasicontinuous}, we will show that any point $z \in Z$ for which \eqref{first mod Lebesgue use} and \eqref{second mod Lebesgue use} hold is an $L^{Q_{*}}(Z)$-Lebesgue point of $\ch{f}$. This suffices to deduce the conclusions of Corollary \ref{cor:Lebesgue points}.

Let $z \in Z$ be a point for which \eqref{first mod Lebesgue use} and \eqref{second mod Lebesgue use} hold. By Proposition \ref{two-weighted} applied with $q = Q_{*} > p$ we have for each ball $B = B_{Z}(z,r) \subset Z$, 
\begin{equation}\label{Theta pre bound}
\left(\int_{B}|\ch{f}-u_{\hat{B}}|^{Q_{*}}\,d\nu\right)^{1/Q_{*}} \ls \Theta_{Q_{*}}(r)\left(\int_{2\hat{B}}g^{p}\,d\mu_{\beta}\right)^{1/p}.
\end{equation}
By combining this with the triangle inequality in $L^{Q_{*}}(Z)$ we conclude that
\begin{align*}
\left(\int_{B}|\ch{f}-\ch{f}(z)|^{Q_{*}}\,d\nu\right)^{1/Q_{*}} &\leq \left(\int_{B}|\ch{f}-u_{\hat{B}}|^{Q_{*}}\,d\nu\right)^{1/Q_{*}} + \nu(B)^{1/Q_{*}}|u_{\hat{B}}-\ch{f}(z)|\\
&\ls \Theta_{Q_{*}}(r)\left(\int_{2\hat{B}}g^{p}\,d\mu_{\beta}\right)^{1/p} + \nu(B)^{1/Q_{*}}\int_{\hat{B}}|u-\ch{f}(z)|\,d\mu_{\beta}. 
\end{align*}
By dividing through by  $\nu(B)^{1/Q_{*}}$, we conclude that
\begin{equation}\label{Q star estimate}
\left(\dashint_{B}|\ch{f}-\ch{f}(z)|^{Q_{*}}\,d\nu\right)^{1/Q_{*}} \ls \frac{\Theta_{Q_{*}}(r)\mu_{\beta}(2\hat{B})^{1/p}}{\nu(B)^{1/Q_{*}}}\left(\dashint_{2\hat{B}}g^{p}\,d\mu_{\beta}\right)^{1/p} + \dashint_{\hat{B}}|u-\ch{f}(z)|\,d\mu_{\beta}. 
\end{equation}
The second term on the right converges to $0$ as $r \rightarrow 0$ by \eqref{first mod Lebesgue use}. To show that the first term on the right converges to $0$ it suffices by \eqref{second mod Lebesgue use} to show that we have
\begin{equation}\label{Theta bound}
\frac{\Theta_{Q_{*}}(r)\mu_{\beta}(2\hat{B})^{1/p}}{\nu(B)^{1/Q_{*}}} \ls r. 
\end{equation}
By \eqref{simplify Theta}, Lemma \ref{hull measure}, and the lower volume decay bound \eqref{lower volume} (here we are using an equivalent uncentered version of this bound applied to the containment of balls $B_{Z}(z,s) \subset 2B$ for $z \in B$, $0 < s \leq r$, see \cite[(9.1.14)]{HKST}) we have
\begin{align*}
\frac{\Theta_{Q_{*}}(r)\mu_{\beta}(2\hat{B})^{1/p}}{\nu(B)^{1/Q_{*}}}  &\ls  r^{Q/Q_{*}-Q/p+\beta/p}\sup_{0 < s \leq r} \sup_{z \in B}s^{1-\beta/p-Q/Q_{*}+Q/p} \\
&= r^{Q/Q_{*}-Q/p+\beta/p+2\theta} \\
&= r
\end{align*}
since, recalling that $\beta = p(1-\theta)$, the exponent of $s$ simplifies to 
\[
1-\frac{\beta}{p}-\frac{Q}{Q_{*}}+\frac{Q}{p} = 1-\frac{p(1-\theta)}{p}- \frac{Q-p\theta}{p}+\frac{Q}{p} = 2\theta > 0,
\]
and the exponent of $r$ in the first line simplifies to
\[
\frac{Q}{Q_{*}}-\frac{Q}{p}+\frac{\beta}{p} = \frac{Q-p\theta}{p} -\frac{Q}{p}+\frac{p(1-\theta)}{p} = 1-2\theta.
\]
We conclude that $C_{\ch{B}^{\theta}_{p}}^{Z}$-q.e.~ point of $Z$ is an $L^{Q_{*}}(Z)$-Lebesgue point of $\ch{f}$, as desired. Lastly, for use in Proposition \ref{hyperbolic Poincare} below we note that combining \eqref{Theta pre bound}, \eqref{Theta bound}, and the radius estimate \eqref{radius to diameter}, 
\begin{equation}\label{key hyperbolic upper gradient}
\left(\dashint_{B}|f-u_{\hat{B}}|^{Q_{*}}\,d\nu\right)^{1/Q_{*}} \ls \diam(B)\left(\dashint_{2\hat{B}}g^{p}\,d\mu_{\beta}\right)^{1/p},
\end{equation}
where we have used that $f = \ch{f}$ $\nu$-a.e.~ on $Z$. Here $g$ denotes any $p$-integrable $p$-weak upper gradient for the extension $u \in \t{D}^{1,p}(\bar{X}_{\rho})$, and the implied constant depends on the constants and exponents in \eqref{lower volume} and \eqref{reverse doubling} as well as the parameters $\alpha$ and $\tau$ associated to the hyperbolic filling $X$ and the exponents $p$ and $\theta$. 
\end{proof}

%Add simpler proof using Proposition \ref{extension hyperbolic Poincare} instead that does not require reverse doubling (and gives existence of $L^{p}(Z)$-Lebesgue points quasieverywhere with respect to the Besov capacity). Probably can remove the reverse doubling content to compensate. Doesn't work. 

The inequality \eqref{key hyperbolic upper gradient} implies the following strengthening of Proposition \ref{extension hyperbolic Poincare}.

\begin{prop}\label{hyperbolic Poincare}
Suppose that $\nu$ has relative lower volume decay of order $Q > 0$ and that $\nu$ satisfies the reverse-doubling condition \eqref{reverse doubling}. Let $p \geq 1$ and $0 < \theta < 1$ be given such that $p \theta < Q$ and set $Q_{*} = Qp/(Q-p\theta)$. Let $f \in \t{B}^{\theta}_{p}(Z)$ be given and let $g: \bar{X}_{\rho} \rightarrow [0,\infty]$ be a Borel function that is a $p$-integrable $p$-weak upper gradient of some function $u \in \t{D}^{1,p}(\bar{X}_{\rho})$ such that $u|_{Z} = f$ $\nu$-a.e. Then for any ball $B \subset Z$ we have
\[
\left(\dashint_{B}|f-u_{\hat{B}}|^{Q_{*}}\,d\nu\right)^{1/Q_{*}} \ls \diam(B)\left(\dashint_{2\hat{B}}g^{p}\,d\mu_{\beta}\right)^{1/p},
\]
with implied constant depending only on the constants and exponents in \eqref{lower volume} and \eqref{reverse doubling}, the parameters $\alpha$ and $\tau$ associated to the hyperbolic filling $X$, and the exponents $p$ and $\theta$. 
\end{prop}

\bibliographystyle{plain}
\bibliography{ExtensionTrace}

\end{document}